\documentclass[12pt,oneside]{article}
\newcommand{\Path}{}
\usepackage{ \Path Paper_style_2}
\usepackage{ \Path Keywords_WB}
\usepackage{ \Path Environments}
\usepackage{ \Path Category}
\usepackage{ \Path Paper_keywords}
\newcommand{\figs}{}
\usepackage{tikz-cd}
\usepackage{ \Path tikzcd2}
\usepackage{subcaption}
\usepackage{pb-diagram}

\DeclareMathOperator{\ForwInv}{Inv^{+}}
\DeclareMathOperator{\Cylinder}{Cyl}
\DeclareMathOperator{\Lan}{Lan}
\DeclareMathOperator{\Ran}{Ran}
\renewcommand{\REnv}[2]{ \mathrm{Lan}_{#1}{#2} }
\renewcommand{\LEnv}[2]{ \mathrm{Ran}_{#1}{#2} }
\title{Dynamics, data and reconstruction}

\begin{document}
	\author{ Suddhasattwa Das \footnotemark[1], Tomoharu Suda \footnotemark[2] \footnotemark[3]}
	\footnotetext[1]{Department of Mathematics and Statistics, Texas Tech University, USA}
	\footnotetext[2]{Department of Applied Mathematics, Tokyo University of Science, Japan}
	\footnotetext[3]{RIKEN Center for Sustainable Resource Science, RIKEN, Japan}
	\date{\today}
	\maketitle
	\begin{abstract}
		The goal of data-driven learning of dynamical systems is to interpret time series as a continuous observation of an underlying dynamical system. This task is not well-posed for a variety of reasons - such as multiple co-existing sub-systems,  topologically inter-weaving of these sub-systems; and more importantly, the non-injectivity of the correspondence between dynamical systems and time series. We show how these ambiguities are circumvented if one considers dynamical systems and measurement maps collectively. Dynamical systems, observed dynamical systems, and time series data - each of these three collections have an extensive network of relations within them, which gives them the mathematical structure of a category. One of the new concepts proposed is a rigorous definition of time series data as a chain of measurement sequences with decreasing information content. This definition subsumes the familiar notions of sequences, time series and even subshifts. Using these notions it is shown that the entire process of converting an observed dynamical systems into a time series object is functorial, and passes through a number of phases each bearing its own categorical structure. This discovery sheds new light on the nature of reconstruction algorithms. Under mild conditions of consistency, reconstruction itself is shown to be functorial operation. This provides a new category theoretic perspective on the nature and limits of reconstruction.
	\end{abstract}
	
	
	\section{Introduction} \label{sec:intro}
	
	The mathematical notion of a dynamical system - a space $\Omega$ along with group of transformations $\Phi^t:\Omega \to \Omega$ is present in most natural, mechanical and human-driven systems, such as the climate, ecological models, controls systems, and the stock market. The points in the space $\Omega$ represent the states the system could be in. The index $t$ represents time and can be either a continuous or discrete variable. For each time $t$, $\Phi^t$ represents a transformation or rearrangement of the states. This leads to different kinds of phenomena - statistical, topological and geometric. 
	
	However the pair $(\Omega, \Phi^t)$ itself may not be visible. For example in climate, epidemiological phenomena, the phase space $\Omega$ is not known or too abstract. Sometimes $\Omega$ may be known but the dynamical law $\Phi^t$ is not apparent. The dynamical system is usually indirectly observed via an observation or measurement map $\phi:\Omega \to \real^d$. Then for any initial point $\omega$ of $\Omega$, and an orbit $\SetDef{ \Phi^t \omega }{ 0\leq t \leq T }$ up to time $T$, one obtains a time series $\SetDef{ \phi \paran{ \Phi^t \omega } }{ 0\leq t \leq T }$ of $d$-dimensional data points. The task of \emph{data-driven discovery} of dynamical systems is to attempt a reconstruction of the underlying dynamical system solely from such time series \cite[e.g.]{BerryDas2023learning, BerryDas2024review}.
	
	Our main result is based on a mathematical formulation of the above notions. We state later in Theorem \ref{thm_consistent:3} that an observation-based inner approximation is also observation-exact. This statement requires a precise understanding of the notions of ``observation", "reconstruction", "data-driven", "inner-approximation" and "exactness". These will be gradually explained throughout the paper. The approach to be presented in this paper works not only for space valued dynamics and observations, where the notion of quantitative approximation can be used, but also for a fairly broad types of settings such as measurable spaces or pre-ordered sets, which would be beyond the scope of most reconstruction algorithms. For example, we can observe a measurable dynamics via a measurable partition of the phase space. This will result in a discrete valued data not immediately amenable to quantitative approximations such as least squares methods. However, our result indicates how to approximate dynamics from such qualitative data and what guess will be the most reasonable.
	
	\begin{figure} [!t]
		\centering
		\begin{tikzpicture}[scale=0.5, transform shape, framed, background rectangle/.style={double, ultra thick, draw=gray, rounded corners}]
			\node [text width=1.0\columnA, text centered, minimum height=0.2\rowA, shape=rectangle, draw=ChhaiB] (n1) at (0\columnA, 0\rowA) { \textbf{1.} Choose a network with countable number of nodes, each node representing an \textit{observed dynamical system}. Nodes are graded to indicate levels of complexity. };
			\node [text width=1.0\columnA, text centered, minimum height=0.2\rowA, shape=rectangle, draw=ChhaiB] (n2) at (1.3\columnA, 0\rowA) { \textbf{2.} For any data set $X$ and complexity level $L$ use the network to perform an inner or outer approximation for $X$.};
			\node [text width=1.0\columnA, text centered, minimum height=0.2\rowA, shape=rectangle, draw=ChhaiB] (n3) at (2.6\columnA, 0\rowA) { \textbf{3.} Devise an algorithm that implements this approximation scheme under various constraints such as size of hypothesis space and precision of regression algorithms. };
			\draw[-to] (n1) to (n2);
			\draw[-to] (n2) to (n3);
		\end{tikzpicture}
		\caption{General scheme for devising algorithms for functorial reconstruction of dynamics. }
		\label{fig:algo_scheme}
	\end{figure}
	
	The task of reconstruction involves the following different mathematical notions - dynamical systems, measurement, orbits, time series, data and reconstructed model itself. All these terms are related but different. To assess how effective a reconstruction algorithm is, it is important to separate the notions of a dynamical system, a dataset, and an algorithm. The separation of these concepts lead to the following questions : 
	
	\begin{enumerate} 
		\item An algorithm takes as input a dataset and outputs a dynamical system. The output is created regardless of whether the dataset arises from a dynamical system. In the general case, how is the output related to the dataset?
		\item An algorithm should be consistent. This means that with more information, i.e., a larger dataset, the estimation should get better. What should be the mathematical definition of consistency ? 
		\item Different dynamical systems may yield the same time series, depending on the measurement. How does this non-uniqueness affect the convergence of reconstructed dynamical systems ? 
	\end{enumerate}
	
	Question 1 (\textit{abbrev.} Q1) addresses the difference between a time series generated by a true orbit, and the general notion of dataset. As has been discovered in recent works, the outcome of numerical methods is strongly dictated by an aspect of dynamical systems called its \emph{spectrum} or \emph{spectral measure} \cite[e.g.]{DGJ_compactV_2018, DasGiannakis_delay_2019, DasGiannakis_RKHS_2018, DasJim2017_SuperC}. Most numerical procedures involve some form of averaging, whose limiting value reveals only selected aspects of the dynamics. This reveals one of the many complications associated to Q1. 
	
	Q2 addresses the issue of consistency which is relevant in most branches of approximation theory. One would expect that if a sequence of datasets has increasing information content, an algorithm should be able to utilize that to provide progressively better approximations. Algorithms have been defined in the context of computability and symbolic dynamics \cite[e.g.]{BDWY2020, HertlingSpandl_gap_2008} but a general definition is hard to provide. Our answer to Q2 is via a category called \emph{time series data}, an approach quite different from other characterizations of information flow \cite[e.g.]{jannink1998ontio, kokar2001data}. 
	
	Dynamical systems tend to have a \emph{nested structure} of sub-systems embedded within sub-systems \cite{MarcusLind1995}. One often relies on the assumption of \emph{ergodicity} to set a definite limit for empirical measures and time series. However, even ergodic dynamical systems which are chaotic are known to have a dense collection of periodic orbits \cite{ KatokStepin1966approx, Katok1980periodic}. Thus these periodic orbits mimic the invariant set and measure of the true dynamics. Moreover, they have been shown to mimic many other ergodic and asymptotic aspects of the dynamics, such as entropy, pressure and dimensions \cite{KatokStepin1967, WangSun2010, LiaoEtAl2018, Sanchez_pressure_2017}. However, periodic orbits have  properties different from the general dynamics - their Lyapunov exponents are zero, and they are non-mixing. These results indicate the difficulty of addressing the issues raised in Q3.
	
	We present four new ideas and perspectives in the search for answers to these questions. The first is that the entire process of data generation from measurements of a dynamical system is a \textit{functorial} process, presented in the diagram below :
	\begin{equation} \label{eqn:outline:1}
		\begin{tikzcd} [scale cd = 0.9 ]
			\begin{array}{c} \DSM \\ \mbox{Dynamics} \\ + \mbox{Measurement}  \end{array}
\arrow[d, "\Forget"']
&& \begin{array}{c} \DSMO \\ \mbox{Dynamics} \\ + \mbox{Measurement} \\+ \mbox{Orbits} \end{array}
\arrow[ll, "\Forget"'] \arrow[rr, "\Discretize_{\Delta t}"] \arrow[lld, dashed, bend right=10, "\proj"] \arrow[dd, dashed, "\bar{\Recon}"' ] \arrow[ddrr, dashed, "\Data_{\Delta t}" ]
&& \begin{array}{c} \DSMO_{\num} \\ \mbox{Discrete-time} \\ \mbox{dynamics} \\ + \mbox{Measurement} \\+ \mbox{Orbits} \end{array}
\arrow[d, "\Obs"] \\
\begin{array}{c} \DS \\ \mbox{Dynamical} \\ \mbox{systems} \\ \Functor{\Time}{\Context} \end{array}
&& && \begin{array}{c} \mbox{Subshifts} \\ \Seq \end{array} 
\arrow[d, "\word"] \\
&& \begin{array}{c} \DS \\ \mbox{Dynamical} \\ \mbox{systems} \\ \Functor{\Time}{\Context} \end{array}
&& \begin{array}{c} \TSD \\ \mbox{Timeseries} \\ \mbox{data} \end{array} 
\arrow[ll, "\Recon"]
		\end{tikzcd}
	\end{equation}
	The figure above helps pose Questions 1--3 above in a precise mathematical setting, by presenting the various entities discussed. The precise role and mathematical nature of reconstruction may be understood only if the process of generation of data is modeled correctly. Diagram \ref{eqn:outline:1} presents data generation via a correspondence $\Data$ between two collections labeled $\DSMO$ and $\TSD$ respectively. Each object in $\DSMO$ is a combination of a dynamical system, a measurement, and orbits in the phase space. Each object in $\TSD$ is a collections of chunks of data with increasing information content. The correspondence $\Data$ brings to light that the source of data is not a dynamical system alone, but a more complicated $\DSMO$ object. Thus two different dynamical systems may map into different $\DSMO$ objects, which generate the exact same time series via the transformation labeled as $\Data$. Although dynamics is our primary object of interest, it only forms one of the components of $\DSMO$.
	
	The correspondence $\Data$ is not just a mapping between objects. We will see later in Section \ref{sec:obs} that one can define \emph{morphisms} or relationships between the compound objects in $\DSMO$, which represent a combined conjugacy between dynamics, orbits and observations. These relations are compositional, and lead to a structure called a \emph{category}. Further, we shall see that a morphism $\phi :X \to X'$ between $\DSMO$-objects is reflected as a hierarchy of information content in their generated time series in $\TSD$. Thus $\Data$ is a correspondence that also preserves the structure present in these categories, making it a \emph{functor}.
	In fact each node in Diagram \ref{eqn:outline:1} is a category of its own, and each arrow a functor. So Diagram \ref{eqn:outline:1} is not just a flowchart but the layout of how during the process of measuring and reconstructing a dynamical system, various structure preserving transformations take place. The functorial property of the maps in the diagram are not imposed or assumed, they are derived as consequences. We formulate these notions for a fairly general class of categories.
	
	Our second main idea is a mathematical notion of a \textit{time series data} object $\TSD$, made precise using category theoretic language in Section \ref{sec:tsd}. A $\TSD$ object is a concept independent of the concepts of dynamical system and number sequences. It models a sequence of growing information content that is fed to an algorithm. This is done with a generality offered by our Assumptions \ref{A:concrete} and \ref{A:obs} which allows $\TSD$ objects to take the form of not just numerical sequences but also preorders and partitions.
	
	The third important point we make in Sections \ref{sec:recon} and \ref{sec:consistency} is that under mild assumptions, the limiting or idealized behavior of reconstruction algorithms are themselves functorial. So while an algorithm may not be structure preserving, its underlying scheme, representing its limiting behavior, is functorial. The arrow $\Recon$ in Diagram \ref{eqn:outline:1} represents a general data-driven reconstruction scheme. The confluence of two functors -- $\Data$ and $\Recon$ provides a host of category theoretic tools to analyze them.
	
	Our fourth and final point is the perspective that the design of reconstruction algorithms and judgment of their efficacy should not be based on individual cases, as argued above in Q3. While the results of an algorithm may be compared and evaluated for individual cases, the correctness of an algorithmic \textit{scheme} is hard to be defined. We use our category theoretic framework in Section \ref{sec:consistency} to promote the idea that a reconstruction scheme should ultimately be an inner or outer approximation. An inner or outer approximation requires a collection of reference objects along with a hierarchical structure. This is presented as the second stage of Figure \ref{fig:algo_scheme}. This reference structure and approximation strategy are made rigorous in Definition \ref{def:ceil_floor} and Theorem \ref{thm:comput}. Note that the choice of reference objects also a part of the design and presented as the first stage. The actual algorithm design is an implementation, often imprecise, of this approximation scheme. This is presented as the third and final stage of Figure \ref{fig:algo_scheme}.
	
	Thus, although any reconstruction algorithm may be built without regard for functoriality, their efficacy or correction can only be assessed in a categorical framework. This requires placing the algorithm in a proper context, as suggested in Figure \ref{fig:algo_scheme}.
	
	This work does not prescribe any numerical recipe or technique, but aims to serve as a general framework to describe and analyze such methods in general settings. For example, our analysis indicates that there is no universal data-driven reconstruction technique that provides an exact reconstruction. This is expressed as Conjecture \ref{conj:oj9l} in Section \ref{sec:consistency}. Regardless of which computational tools such as neural networks,  kernel-methods, or DMD are used, algorithms should aim to capture the collective categorical structure present within dynamical systems.

	\paragraph{Outline} We begin with a brief review of category theory in Section \ref{sec:cat}. We next realize the confluence of dynamics, measurements and orbits as a category, using the language of comma categories in Section \ref{sec:obs}. Next Section \ref{sec:orbits} presents the combined notion of dynamical systems being observed along orbits as a mathematical object. Section \ref{sec:cncrt} presents some consequences of our assumptions made in Section \ref{sec:cat}. This results in the creation of a category theoretic notion of subshifts in Section \ref{sec:shift}.  Section \ref{sec:msrmnt} next explains how subshifts arise naturally from measurements of dynamical systems. Next Section \ref{sec:tsd} introduces time series data as a rigorous mathematical object amenable for categorical analysis. Section \ref{sec:recon} discusses the functorial nature of reconstruction, and Section \ref{sec:consistency} discusses various properties of reconstruction as a functor. Some technical lemmas and proofs have been moved to the Appendix.
	
	\section{A categorical perspective for dynamics} \label{sec:cat}
	
	The usual approach to dynamical systems is analytical -- one assumes some properties on a space and a dynamics law, and studies the consequences of those properties. The structure provides the context for the study. In a set-theoretic context one looks for features such as invariant sets, fixed points, and periodic orbits. In a topological context one looks for features such as attractors and stable and unstable manifolds. In a measure theoretic context, features such as invariant measures and ergodic measures. A categorical viewpoint does not make separate considerations for these contexts. All the definitions assume a generic context and only takes into account the network of relations that exist within the context. Spaces and maps lose all of their internal content and exist as symbolic points as arrows. The objective is to re-establish all the various properties as consequences of these network of relations. We begin with a brief introduction to the notion of categories and functors. 
	
	\paragraph{Category} A category $\calC$ is a collection of two kinds of entities : 
	\begin{enumerate} [(i)]
		\item objects : usually representing different instances of the same mathematical construct; We will often write $x \in \calC$ to mean that $x$ is an object of $\calC$.
		\item morphism : connecting arrows between a pair of objects; satisfying the following three properties --
		\item compositionality : given any three objects $a,b,c$ of $\calC$ and two morphisms $a \xrightarrow{f} b$ and $b \xrightarrow{g} c$, the morphisms can be joined end-to-end to create a composite morphism represented as $a \xrightarrow{g \circ f} b$;
		\item associativity : the composition of morphisms is associative;
		\item identity morphism : each object $a$ is endowed with a morphism $\Id_x$ called the \emph{identity} morphism, which play the role of unit element in composition.
	\end{enumerate}
	
	Given two objects $x,y$ in $\calC$, the collection of arrows from $x$ to $y$ is denoted as $\Hom(x;y)$. Note that this collection may be infinite, finite or even empty. The last criterion implies that for each $x$ $\Hom(x;x)$ has at least one member. Whenever there are multiple categories being discussed, one uses the notations $\Hom_{\calC}(x;y)$ or $\calC(x;y)$ to indicate that the morphisms are within the category $\calC$.
	
	\paragraph{Examples} One of the most fundamental categories is $\SetCat$, the category in which objects are sets up to a certain prefixed cardinality, and arrows are arbitrary maps. Similarly $\Topo$ denotes the category of topological spaces, with continuous maps as arrows. We denote by $\VectCat$ the category in which the objects are vector spaces and arrows are linear maps. The collection $\AffineCat$ that we have already defined has the same objects as $\VectCat$ but all affine maps as morphisms. Note that this includes the morphisms in $\VectCat$. This makes $\VectCat$ a \emph{subcategory} of $\AffineCat$. Suppose $\calU$ is any set. Then the power-set $2^{\calU}$ of subsets of $\calU$ is a category, in which the relations are the subset $\subseteq$ relations. Note that there can be only at most one arrow between any two objects $A,B$ of this category, which is to be interpreted as inclusion. Such categories are known as \emph{preorders}, and other examples are the category of ordered natural numbers, real numbers, open covers, and the concept of infinitesimal \cite[see]{Das2023CatEntropy}. 
	
	A particularly important example of categorical structure can be found in semigroups. Any semigroup $\calG$ can be interpreted as a 1-point category, in which each arrow corresponds bijectively to elements of $\calG$. The associativity of semigroup operation becomes the associativity of arrow composition. Figure \ref{fig:semigroups} presents some 1-point categorical representations of the semigroups $\SqBrack{\num_0, +}$, $\SqBrack{\integer, +}$ and $\SqBrack{\real, +}$. See Figure \ref{fig:semigroups} for some illustrations. We next discuss how different categories can be transformed into one another.
	
	\begin{figure}\center
		\includegraphics[width=0.26\textwidth]{\figs 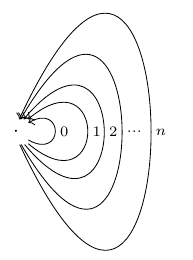}
		\includegraphics[width=0.47\textwidth]{\figs 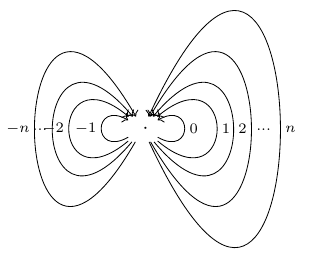}
		\includegraphics[width=0.24\textwidth]{\figs 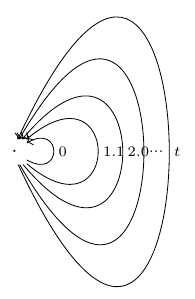}
		\caption{Representing semi-groups as 1-object categories -- $\SqBrack{ \num_0, +}$, $\SqBrack{ \integer, +}$ and $\SqBrack{ \real, +}$.} 
		\label{fig:semigroups}
	\end{figure} 
	
	\paragraph{Functors} Given two categories $\calC, \calD$, a functor $F:\calC\to \calD$ is a mapping between their objects along with the following properties : 
	\begin{enumerate} [(i)]
		\item For each $x,y\in ob(\calC)$, $F$ induces a map $F_{x,y} : \Hom_{\calC}(x;y) \to \Hom_{\calD}( Fx; Fy)$. Thus arrows / morphisms between any pair of points get mapped into morphisms between the corresponding pair of points in the image.
		\item $F$ preserves compositionality : given any three objects $a,b,c$ of $\calC$ and two morphisms $a \xrightarrow{f} b$ and $b \xrightarrow{g} c$, $F(g\circ f) = F(g) \circ F(f)$.
		\item $F$ preserves identity : $F(\Id_a) = \Id_{F(a)}$.
	\end{enumerate}
	
	Functors are thus maps that also preserve categorical structure, mainly by preserving compositionality. One of the simplest notions of functors are monotonic functions from $\real$ to $\real$, and may be interpreted as functors from $\Rplus$, which is the category defined by $\real$ with the preorder $\leq$, to $\Rplus$ or $\Rminus$, which is the category defined by $\real$ with $\geq$. Realizing mathematical transformations as functors leads to deeper insights in the associated field. Some examples are Lebesgue integration \cite{Leinster_integration_2020}, diameter and Lebesgue number \cite{Das2023CatEntropy}, and topological closure \cite{ClementinoGiuliTholen1996}.
	
	One can interpret a functor $F : \calC \to \calD$ as drawing of the $\calC$ within $\calD$. The objects and arrows in $\calC$ are overlaid on objects and arrows in $\calD$. For that reason a functor is also called a \emph{diagram} or $\calC$-diagram if its domain is $\calC$. There is a structure preserving manner in which $\calC$-diagrams may be transformed. This leads us to the next topic : 
	
	\paragraph{Natural transformations} Let $\calC, \calD$ be two categories, and $F, F' : \calC \to \calD$ be two functors. A natural transformation $\alpha$ from $F$ to $F'$, denoted as $\alpha : F \Rightarrow F'$ is a family of arrows $\SetDef{ \alpha_c : F(c) \to F'(c) }{ c\in ob(\calC) }$ such that the following commutations hold
	\begin{equation} \label{eqn:def:nat_transform}
		\forall \begin{tikzcd} c \arrow[d, "f"] \\ c' \end{tikzcd} , \quad 
		\begin{tikzcd}
			F(c) \arrow[r, "\alpha_{c}"] \arrow[d, "F(f)"'] & F'(c) \arrow[d, "F'(f)"] \\
			F(c') \arrow[r, "\alpha_{c'}"'] & F'(c')                        
		\end{tikzcd}
	\end{equation}
	Just like arrows / morphisms connect objects, natural transformations connect functors. In fact the collection $\Functor{\calC}{\calD}$ of functors from $\calC$ to $\calD$ is itself a category, in which the morphisms are natural transformations. This realization that the collection of functors itself is a category bears a lot of significance in the study of homology, simplicial complexes and sheaf theory in general. We will find a special significance in the study of dynamical systems.
	
	\paragraph{Category of dynamical systems} Any dynamical system has a concept of time associated to it. The role of time can be played by any semigroup, such as $(\num_0, +)$, $(\num, +)$, $(\real_0, +)$ and $(\real, +)$. Such a time semigroup shall be denoted by $\Time$. We shall interpret it as a 1-point category. Next let us fix a category $\Context$ and call it the \emph{context} category. It is usually taken to be $\Topo$, $\VectCat$ or $\MeasCat$, the category of measurable spaces and measurable maps. The category $\Context$ represents the nature of spaces and maps that are involved our dynamical systems of interest. The context and time are the two ingredients of a dynamical systems. 
	
	The collection of dynamical systems evolving according to time interpretation $\calT$ and within the context $\Context$, forms the functor category 
	\[\DS_{\Context}^{\Time} := \Functor{\Time}{\Context}.\] 
	If the categories $\Time$ and $\Context$ are clear from context, we shall denote $\DS_{\Context}^{\Time}$ simply as $\DS$. Note that each object in this category is a functor $\Phi:\Time \to \Context$. This involves selecting an object $\Omega$ from $\Context$ and a family of endomorphisms $\Phi^t : \Omega \to \Omega$ indexed by $\braces{ t\in \Time }$. This is exactly the definition of a dynamical system, or more precisely, an action of the semigroup $\calT$ on $\Omega$. The diagram below shows the case when $\Time = \SqBrack{\num_0, +}$. See Figure \ref{fig:dyn_functor} for an illustration. 
	
	\begin{figure}
		\centering
		\includegraphics[width=0.26\textwidth]{\figs N_semigroup.pdf}
		\begin{tikzcd} {} \arrow[rrr, mapsto, "\Phi"] &&& {} \end{tikzcd}
		\includegraphics[width=0.26\textwidth]{\figs 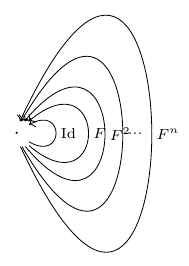}
		\caption{Dynamics as a functor. Section \ref{sec:cat} presents a dynamical system as a functor between a semigroup category $\Time$ representing time, and a category $\calC$ representing the context. Shown here is such a functor $\Phi$ transforming the additive semigroup $\SqBrack{\num_0, +}$ of non-negative integers into self-maps on a space $\Omega$. Since $\SqBrack{\num_0, +}$ is generated by the unit element $1$, the image $F$ of this element generates all iterations of the dynamics. } 
		\label{fig:dyn_functor}
	\end{figure}
	
	This interpretation of dynamical systems as a functor category $\DS$ has provided new insights into the study of dynamical systems and their asymptotic properties \cite[e.g.]{perrone2022kan, Suda2022Poincare, Suda2023dynamical, Suda2022equivalence}. The morphisms in $\DS$ also have a very useful interpretation. Given two objects $\Phi, \Phi' \in \DS$, a morphism between them in $\DS$ is a natural transformation $\alpha : \Phi \Rightarrow \Phi'$. This means that the following commutations hold 
	\begin{equation} \label{eqn:dp30x}
		\begin{tikzcd}
			\Omega \arrow{d}[swap]{\alpha_\Omega} \arrow{r}{\Phi^t} & \Omega \arrow{d}{\alpha_\Omega} \\
			\Omega' \arrow{r}{\Phi'^t} & \Omega'
		\end{tikzcd} ,
		\quad \forall t\in \Time .
	\end{equation}
	as dictated by \eqref{eqn:def:nat_transform}. This commutation turns out to be exactly the concept of \emph{semi-conjugacy} in Dynamical systems theory. Since $\calT$ is a 1-object category, the natural transformation $\alpha$ has only 1 connecting morphism $\alpha_{\Omega}$. Thus without ambiguity and by reuse of notation, we denote the one and only connecting morphism of the natural transform $\alpha$ also as $\alpha$. If $\alpha$ is surjective, then the dynamics $\Phi'^t$ can be considered to be a  partial measurement of the system which evolves autonomously. The following commuting diagram illustrates the case when $\Time = \num_0$
	\[\begin{tikzcd}
		\Omega \arrow[r, "\Phi^1"] \arrow[rr, "\Phi^2", bend left=49] \arrow[rrrr, "\Phi^n", bend left] \arrow[d, "\alpha"] & \Omega \arrow[r, "\Phi^1"] \arrow[d, "\alpha"] & \Omega \arrow[r, "\Phi^1"] \arrow[d, "\alpha"] & ... \arrow[r, "\Phi^1"]  & \Omega \arrow[d, "\alpha"] \\
		\Omega^{'} \arrow[r, "\Phi'^1"'] \arrow[rr, "\Phi'^2"', bend right=49] \arrow[rrrr, "\Phi'^n"', bend right]       & \Omega^{'} \arrow[r, "\Phi'^1"']           & \Omega^{'} \arrow[r, "\Phi'^1"']           & ... \arrow[r, "\Phi'^1"'] & \Omega^{'}           
	\end{tikzcd}\]    
	One dynamical system may be contained in another, or may be a factor of the other. Both relations are special instances of the morphism \eqref{eqn:dp30x}. Thus recognizing the categorical structure of the collection of all time-$\calT$, $\calC$-dynamical systems also implicitly takes into accounts all these relations. At this point we make a technical assumption on the categorical nature of the context $\Context$ :
	
	\begin{Assumption} \label{A:concrete}
		The context category $\Context$ is a topological concrete category, with $U : \Context \to \SetCat$ being its forgetful functor.
	\end{Assumption}
	The notion of a topological concrete category \cite[e.g.]{Herrlich1974topo1, herrlich1974topo2, brummer1984topological, DikranjanEtAl1988topo} extracts certain categorical / structural properties of $\Topo$ to enable a more general and abstract category in which several familiar notions of topology can be realized along with their usual properties. Thus while $\Topo$ remains an example of Assumption \ref{A:concrete}, other categories such as $\MeasCat$ are also examples. Throughout our discussion, several properties of such a category $\Context$ will be invoked, each capturing different aspects of topological space. Section \ref{sec:app:topo_cncrt} presents some technical details.
	
	Time series is not directly generated from a dynamical system, but from measurements made along trajectories of the dynamics. To capture the data-generation process as a functor, we need to establish a category which combines dynamics as well as measurements.
	\section{Observations} \label{sec:obs}
	
	\begin{table} [!t]  
		\caption{ Instances of Assumption \ref{A:obs}. The following instances of $\Context_{obs}$, $\Context$ and $\iota$ are commonly encountered in data-driven approaches to dynamical system reconstruction. }
		\begin{tabularx}{\linewidth}{|l|l|L|}        \hline
			$\Context$ & $\Context_{obs}$ & Interpretation of $\iota : \Context_{obs} \to \Context$ \\ \hline
			$\Topo$ & $\FinSet$ & Observing continuous dynamics on topological spaces via finite-state machines \\ \hline
			$\Topo$ & $\VectCat$ & Real-valued measurements of continuous dynamics on topological spaces, along with linear transforms of the data.  \\ \hline
			$\MeasCat$ & $\FinSet$ & Measurable dynamics being observed by finite measurable partitions \\ \hline
			$\Topo$ & $\textbf{\text{Alex}}$ & Morse-decomposition calculations of continuous dynamics on topological spaces, along with order preserving transformations of the data. See Example \ref{ex:Alexo}  \\ \hline
		\end{tabularx}
		\label{tab:A:obs}
	\end{table}
	
	The simplest interpretation of a dynamical system along with a measurement, is a dynamical system $\Phi^t : \Omega \to \Omega$ along with a map $\phi : \Omega \to \real$. Instead of considering only real-valued measurements, we shall allow more general measurement spaces $\calY$ drawn from some category $\ObsSp$. To prescribe our assumptions on $\ObsSp$ we need to recall certain categorical notions. A morphism $f:a\to b$ is said to be \emph{monic} / \emph{injective} if for any third object $c$ and two morphisms $g, g' : c\to a$, the equality $fg = fg'$ holds iff $g=g'$. Note that in the context of $\SetCat$ this definition is equivalent to the set-theoretic notion of injectivity. One can similarly define a morphism to be \emph{epic} / \emph{surjective}. Given an object $b$, a \emph{subobject} of $b$ refers to any injection $f : a\to b$. Injections and surjections will be depicted by hooked $\begin{tikzcd} {} \arrow[r, hook] & {} \end{tikzcd}$ and double-headed $\begin{tikzcd} {} \arrow[r, two heads] & {} \end{tikzcd}$ arrows respectively. One of the most important properties of being topological concrete is as follows :
	
	\begin{property} \label{Py:product}
		A category $\Context$ satisfying Assumption \ref{A:concrete} has products, finite and infinite.
	\end{property}
	
	In $\SetCat$, $\MeasCat$ and $\Topo$, products are simply cartesian products. A product of an indexed family of spaces $\SetDef{a_i}{i\in I}$ is denoted as $a := \prod_{i\in I} a_i$. It is equipped with coordinate-wise projection maps $\pi_i : a\to a_i$. These projection maps are also \emph{universal} in the sense that any collection of functions $f_i : b\to a_i$ from a common domain $b$ must factor through a unique map $\phi : b\to a$. The categorical notion of products is a structural generalization of these ideas. Section \ref{sec:app:colim} presents more details.
	
	\begin{Assumption} \label{A:obs}
		There is an observation category $\ObsSp$, and a functor $\iota : \ObsSp \to \Context$ s.t. 
		\begin{enumerate} [(i)]
			\item $\ObsSp$ has finite direct products and  $\iota$ preserves them. 
			\item $\ObsSp$ is closed with respect to subobjects in the sense that for every $\Omega \in \ObsSp$ and every injection $A \hookrightarrow \iota \Omega$ in $\mathcal{C}$, there must be some $A' \in \ObsSp$ such that $A = \iota A'$.
		\end{enumerate}
	\end{Assumption}
	
	The category $\ObsSp$ could be $\FinSet$, the sub-category of $\SetCat$ created by finite sets. In this case, when $\Context$ is $\Topo$, $\SetCat$ or $\MeasCat$, $\ObsSp$ is a sub-category and the functor $\iota$ is just inclusion. Table \ref{tab:A:obs} presents several instances of Assumption \ref{A:obs}. An interesting instance worth special mention is as follows :
	
	\begin{example} [Pre-order valued observations] \label{ex:Alexo}
		An Alexandroff topology is a topology which is closed under arbitrary intersection \cite[e.g.]{Arenas1999alex}. There is a one-to-one correspondence between Alexandroff topologies on a set $X$ and preorders (reflexive and transitive relations) on $X$. The preorder naturally associated to any Alexandroff topology is called the \emph{specialization preorder}. Here the relation $x\leq y$ holds if $x$ lies in the closure of $y$. One automatically obtains a subcategory $\textbf{\text{Alex}}$ of $\Topo$. There has been an increased interest in pre-order valued observations, in the form of Morse-decompositions of invariant sets \cite[e.g.]{ConleyZehnder1984, Mischaikow1999Conley, Mischaikow2002topo}. Every dynamical system or more specifically, the omega-limit set of every orbit can thus be associated with a Morse decomposition. This approach to quantifying dynamical systems corresponds to the choice of the category $\ObsSp$ from Assumption \ref{A:obs} being $\textbf{\text{Alex}}$.
	\end{example}
	
	We shall interpret a dynamical system along with a measurement as a 3-tuple $\paran{ \Omega, \Phi^t, \phi }$, where $\phi : \Omega \to \iota(\calY)$ is a morphism in $\Context$ which is to be interpreted as the measurement of the domain $\Omega$. The collection of these 4-tuples also bear a categorical structure. A typical morphism in this category is
	\begin{equation} \label{eqn:def:DSM_morphism}
		\begin{tikzcd}
			\Omega \arrow{d}[swap]{h} \arrow[blue]{r}{f} & \Omega \arrow{d}[swap]{h} \arrow[blue]{r}{\phi} & \iota\calY \arrow{d}[swap]{\iota A} \\
			\Omega' \arrow[Holud]{r}{f'} & \Omega' \arrow[Holud]{r}{\phi'} & \iota\calY^{'}
		\end{tikzcd}
	\end{equation}
	where the blue and yellow rows represent two separate objects in $\DSM$. Note that $h$ is a morphism in $\calC$ representing a change of variables, while $A$ is a morphism in $\calC_{obs}$, which may allow a more restricted class of morphisms. The morphisms in $\DSM$ thus represent change of variables that are consistent not only with the dynamics but also the measurement maps. If two measurement dynamical systems related as in \eqref{eqn:def:DSM_morphism}, then the dataset generated by $\paran{ \Phi^t, \Omega, \phi, \calY }$ can be transformed into that of $\paran{ \Phi'^t, \Omega', \phi', \calY' }$. 
	The categorical structure borne by \eqref{eqn:def:DSM_morphism} is described more conveniently using the language of \emph{comma categories}.
	
	\paragraph{Comma categories} Given any category $\calC$ let $\Hom(\calC)$ denote the collection of all its morphisms. A general arrangement of categories and functors :
	\[\begin{tikzcd}
		\mathcal{A} \arrow[rd, "\alpha"'] & & \mathcal{B} \arrow[ld, "\beta"] \\
		& \mathcal{C} & 
	\end{tikzcd}\]
	creates a special category called a \emph{comma category} $\Comma{\alpha}{\beta}$. Its objects are 
	\[ ob\paran{\Comma{\alpha}{\beta}} := \SetDef{ \paran{ a, b, \phi } }{ a\in ob(\calA), \, b\in ob(\calB), \, \phi \in \Hom_{\calC} \paran{ \alpha a ; \beta b } } , \]
	and the morphisms comprise of pairs
	$\SetDef{ (f,g) }{ f\in \Hom(\calA),\, g\in \Hom(\calB) }$ such that the following commutation holds :
	\[\begin{tikzcd} \blue{ (a,\phi,b) } \arrow{r}{ (f,g) } & \akashi{ (a',\phi',b') } \end{tikzcd} \,\Leftrightarrow \,
	\begin{tikzcd} a \arrow{d}{f} \\ a' \end{tikzcd}, \begin{tikzcd} b \arrow{d}{g} \\ b' \end{tikzcd}, \mbox{ s.t. }
	\begin{tikzcd}
		\blue{ \alpha a } \arrow{r}{\alpha f} \arrow[blue]{d}{\phi} & \akashi{ \alpha a' } \arrow[Akashi]{d}{\phi'} \\
		\blue{ \beta b } \arrow{r}{\beta g} & \akashi{ \beta b' }
	\end{tikzcd}\]
	This category $\Comma{\alpha}{\beta}$ may be interpreted as connections between the functors $\alpha, \beta$, via their common codomain $\calC$. Comma categories contain as sub-structures, the original categories $\calA, \calB$, via the \emph{forgetful} functors
	\[\begin{tikzcd} \calA & \Comma{\alpha}{\beta} \arrow{l}[swap]{\Forget_1} \arrow{r}{\Forget_2} & \calB \end{tikzcd}\]
	whose action on morphisms in $\Comma{\alpha}{\beta}$ can be described as
	\[ \begin{tikzcd} \blue{a} \arrow[blue]{d}{f} \\ \blue{a'} \end{tikzcd}
	\begin{tikzcd} {} & & {} \arrow[ll, "\Forget_1"'] \end{tikzcd}
	\begin{tikzcd}
		\blue{ \alpha a } \arrow[r, "\phi"] \arrow[blue]{d}[swap]{\alpha f} & \akashi{ \beta b } \arrow[Akashi]{d}{\beta g} \\
		\blue{ \alpha a' } \arrow[r, "\phi'"'] & \akashi{ \beta b' }
	\end{tikzcd}
	\begin{tikzcd} {} \arrow[rr, "\Forget_2"] & & {} \end{tikzcd}
	\begin{tikzcd} \akashi{b} \arrow[Akashi]{d}{g} \\ \akashi{b'} \end{tikzcd} \]
	Comma categories prevail all over category theory and mathematics. If either $\calA$ or $\calB$ is taken to be $\star$ the trivial category with a single object , then the resulting comma categories are called left and right \emph{slice-categories} respectively. If $\calA = \calB = \calC$, then the comma category becomes the \emph{arrow-category}. The objects here are the arrows in $\calC$, and the morphisms are commutation squares. Comma, slice and arrow categories thus represent finer structures present within categories.
	
	\paragraph{Dynamics and observation} We have already defined the category $\DS$ of dynamical systems to be $\Functor{\calT}{\Context}$. We have the obvious \emph{domain}-functor
	\[ \dom : \Functor{\calT}{\Context} \to \Context , \quad \begin{tikzcd} \Omega \arrow{d}{f} \\ \Omega \end{tikzcd} \;\mapsto\; \Omega  \]
	which extracts the domain / phase-space of the dynamical system. 
	Now consider the arrangement
	\[\begin{tikzcd}
		\Functor{\calT}{\Context} \arrow[bend right = 20]{dr}[swap]{\dom} & & \ObsSp \arrow[bend left = 20]{dl}{\iota} \\
		& \Context
	\end{tikzcd}\]
	The resulting comma category $ \Comma{\dom}{\iota}$ is the category $\DSM$ of dynamical systems with measurement. 
	
	In all the steps of our analysis as presented in Diagram \ref{eqn:outline:1}, we consider different collections of objects of a particular type, along with their network of inter-relations. We have extended that to the joint concept of a dynamical system with measurement. As outlined in Table \ref{tab:A:obs}, the category $\Context_{obs}$ represents the particular arrangement of measurements and data-transformations that we are tasked with. The role of the functor $\iota$ is to transform every observable object $\calY$ in $\Context_{obs}$ into a $\Context$ object. A choice of a category determines the nature of its morphisms. The measurement $\phi$ acquires the properties not of $\Context_{obs}$ but of $\Context$. For that reason, it binds the domain $\Omega$ not directly to $\calY$ but to its image $\iota(\calY)$ in $\Context$.
	
	So far we have presented categories which collectively represent all dynamical systems, and all dynamical systems along with an measurement map. We next include the notion of an orbit into this. It will be done utilizing yet another categorical construct. 
	
	\section{Orbits} \label{sec:orbits}
	
	\begin{table} [!t]
		\caption{Basic components of the framework. The table summarizes the basic ingredients of our theoretical framework. They are used to construct the various categories and functors in the diagram \eqref{eqn:outline:1}.  See Tables \ref{tab:param2} and \ref{tab:param3} for more advanced constructions based on these components. }
		\begin{tabularx}{\linewidth}{|l|L|}        \hline
			Variable & Interpretation \\ \hline
			$\Context$ & Context category -- the collection of spaces on which we consider various dynamics, along with transformations between them. Examples are measurable spaces, vector spaces and topological spaces. \\ \hline
			$\Time$ & Time category -- a semigroup such as $\paran{\num_0,+}$ or $\paran{\real,+}$ which models time, represented as a 1-point category \\ \hline
			$\Context_{obs}$ & Category of observable space -- a subcategory of $\Context$ consisting of only those objects/spaces which we consider to be observable \\ \hline
			$\Delta t$ & Sampling functor that transforms $\SqBrack{\num_0, +}$ into the general time semi-group $\Time$ (Section \ref{sec:orbits}). \\ \hline
		\end{tabularx}
		\label{tab:param1}
	\end{table}
	
	A full orbit in a deterministic dynamical system is completely determined by its initial point. Thus a pair of a dynamical system and an orbit is in bijective correspondence with a pair of a dynamical system and a point in it. Similarly given a dynamical system, a collection of $m$ orbits corresponds to a collection of $m$ initial points. In this section we discover a categorical structure within such collections. This suggests the following as the category of dynamical systems with orbits :
	
	\begin{definition}
		The category $\DSO$ contains as objects an ordered set of the form $\paran{\Omega, \Phi^t, \omega}$, where $\paran{\Omega, \Phi^t}$ is an object of $\DS$, and $\omega : X\to \Omega$ is some $\Context$-morphism. A morphism $\paran{\Omega, \Phi^t, \omega} \to \paran{\Omega', \Phi'^t, \omega'} $ is a morphism $h:\paran{\Omega, \Phi^t} \to \paran{\Omega', \Phi'^t}$ in $\DS$, and a $\Context$-injection $g:X\to X'$ such that the following commutation holds
		\[\begin{tikzcd} [row sep = large]
			X \arrow[d, "g"', hook] \arrow[rr, "\omega", blue] && \Omega \arrow{d}[swap]{h} \arrow[blue]{r}{\Phi^t} & \Omega \arrow{d}[swap]{h} \\
			X' \arrow[rr, "\omega'", Holud] && \Omega' \arrow[Holud]{r}{\Phi'^t} & \Omega'
		\end{tikzcd}\]
		Thus morphisms in this category are transformations of the domain $\Omega$ that preserve the initial points as well as create a semi-conjugacy.
	\end{definition}
	
	Note that one also has a domain functor $\dom : \DSO \to \Context$ which extracts the domain object from a $\DSO$ object. Now that we have defined the categories $\DSM$ and $\DSO$ we can combine them into a more complex construct, via the following arrangement of functors and categories
	\[\begin{tikzcd}
		\DSO \arrow[bend right=20]{drr}[swap, pos=0.3]{ \dom } && && \Context_{obs} \arrow[bend left=20]{dll}[pos=0.3]{ \iota } \\
		&& \Context 
	\end{tikzcd}\]
	This leads to a comma category which we called the category of dynamics, orbits, and measurements : 
	\begin{equation} \label{eqn:def:DSMO}
		\DSMO := \Comma{ \dom }{ \iota }
	\end{equation}
	The objects in this category is a triplet of a dynamical system $(\Omega, \Phi^t)$, a measurement $\Phi : \Omega \to \calY$, and a morphism $\omega : X\to \Omega$ representing a collection of initial points. Note that the first two members form an object in $\DSM$. This leads to a succession of projection functors
	\[\begin{tikzcd}
		&& \DSM \arrow[drr, "\proj"] \arrow[drrrr, bend left=10, dashed, "\dom"] \\
		\DSMO \arrow[urr, "\proj"] \arrow[drr, "\proj"'] && && \DS \arrow[rr, "\dom"] && \calC \\
		&& \DSO \arrow[urr, "\proj"'] \arrow[urrrr, bend right=10, dashed, "\dom"'] \\
	\end{tikzcd}\]
	We have reused notation to also denote the composite functors from $\DSM$ and $\DSO$ into $\calC$ as $\dom$.
	This functorial relation encodes how the objects and relations in $\DS$, $\DSM$ and $\DSMO$ are contained in the other. We formally define
	
	\begin{definition} [Dynamical systems with measurements along orbits]
		$\DSMO$ is the category in which each object consists of a dynamical system $(\Omega, \Phi^t)$, a measurement map $\phi : \Omega \to \iota Y$, and a morphism $\omega : X\to \Omega$. Such an object is represented as
		\[\begin{tikzcd}
			X \arrow[rr, "\omega"'] && \Omega \arrow["\Phi^t"', loop, distance=2em, in=125, out=55] \arrow[r, "\phi"'] & \iota \calY
		\end{tikzcd}\]
		By Assumption \ref{A:obs} and the existence of epi-mono factorization, which will be explained later, we may assume $\phi$ to be an epimorphism without loss of generality.
	\end{definition}
	
	A typical morphism is a triple of morphisms $h:\Omega \to \Omega'$, $A : \calY \to \calY'$ and an injection $g:X\to X'$ such that the following commutations hold for all times $t$ : 
	\[\begin{tikzcd} [row sep = large]
		X \arrow[rr, "\omega", blue] \arrow[d, "g"'] && \Omega \arrow{d}[swap]{h} \arrow[blue]{r}{\Phi^t} & \Omega \arrow{d}[swap]{h} \arrow[blue]{r}{\phi} & \iota \calY \arrow{d}[swap]{\iota A} \\
		X' \arrow[rr, "\omega'", Holud] && \Omega' \arrow[Holud]{r}{\Phi'^t} & \Omega' \arrow[Holud]{r}{\phi'} & \iota \calY'
	\end{tikzcd}\]
	So a morphism between two $\DSMO$ objects is a re-parameterization of the phase space which preserves the initial points as well as the measurements. We shall see in the next two sections that dynamical systems related this way would yield time series such that the information content of one is included in another. We have so far derived the nodes $\DS$, $\DSM$ and $\DSMO$ of the figure in \eqref{eqn:outline:1}. The category $\DSMO$ may be interpreted as a comma category
	\[ \DSMO = \left[ \begin{tikzcd} \Context \arrow[dr, bend right=20, "\Id"'] & & \DSM \arrow[dl, bend left=20, "\dom"'] \\ & \Context \end{tikzcd} \right]\]

	There is a natural inclusion of $\DSM$ into $\DSMO$ : 
	\begin{equation} \label{eqn:def:DSM_to_DSMO}
		\begin{tikzcd} \DSM \arrow[rr, "\kappa", "\subset"'] && \DSMO \end{tikzcd} , \quad 
		\begin{tikzcd} \paran{\begin{array}{c} \Omega \\ \Phi^t \\ \phi  \end{array}} \arrow[rr, mapsto] && \paran{\begin{array}{c} \Omega \\ \Phi^t \\ \phi \\ \Id_\Omega \end{array}} \end{tikzcd}
	\end{equation}
	This inclusion thus simply adds to every measured dynamical system, the set of all its points. We have so far derived the nodes $\DS$, $\DSM$ and $\DSMO$ of the figure in \eqref{eqn:outline:1}. 
	
	\begin{table} [!t]
		\caption{Secondary set of components of the framework. The table summarizes the various categories and functors created from the building blocks declared in Table \ref{tab:param1}. These categories and functors make a direct appearance in the diagram \eqref{eqn:outline:1}. See Tables \ref{tab:param3} for more advanced constructions based on these components. }
		\begin{tabularx}{\linewidth}{|l|L|}        \hline
			Variable & Interpretation \\ \hline
			$\DS$ & Category of dynamical systems -- a categorical interpretation of self maps on objects from the context category $\Context$ \\ \hline
			$\TimeB$ & Time instants, or elements of the time semigroup $\Time$, existing as an object of $\Context$ \\ \hline
			$\Seq$ & Category of shift spaces based in spaces selected from $\Context_{obs}$  \\ \hline
			$\FinSeq$ & Category of datasets of a fixed length, based in spaces selected from $\Context_{obs}$ \\ \hline
			$\TSD$ & Categorical interpretation of time series, in which morphisms represent the notion of information embedding \\ \hline        
			$\Obs$ & functor $\DSMO \to \Seq$ associating to any measurement along an orbit of a dynamical system, the generated time series \\ \hline        
			$\word$ & functor $\Seq \to \TSD$ embedding any subshift into the larger and more general category of time series data \\ \hline             
			$\Recon$ & functor $\TSD \to \DS$ that represents any general reconstruction scheme that converts time series data into dynamical systems \\ \hline   
		\end{tabularx}
		\label{tab:param2}
	\end{table}
	
	\paragraph{Discretization} Even if the time $\calT$ is a continuum, data is usually collected at discrete instants of time. This conversion of data can also be expressed  functorially :
	
	\begin{definition}
		Given a semi-group category $\Time$, a \emph{sampling} is a functor $\Delta t : \SqBrack{\num_0, +} \to \Time$. Such functors are in one-to-one correspondence with semigroup homomorphisms from $\SqBrack{\num_0, +}$ to $\Time$. Since the semigroup $\SqBrack{\num_0, +}$ is finitely generated, such a functor is uniquely defined by the image of its generating element. By a reuse of notation, this morphism in $\Time$ will also be denoted as $\Delta t$.
	\end{definition}
	
	The sampling $\Delta t$ leads to the arrangement 
	\[\begin{tikzcd}
		\SqBrack{\num_0, +} \arrow[r, "\Delta t"] & \Time \arrow{r}{\Phi} & \Context
	\end{tikzcd}\]
	In other words, a time-$\Time$ dynamics on domain $\Context$ can be converted into a discrete time set-theoretic dynamical system. In fact this pre-composition by $\Delta t$ is itself a functor between functor categories, to be denoted as $\circ \Delta t$. Next note the commutation diagram below on the left :
	\[\begin{tikzcd}
		\DS^{\Time}_{\Context} \arrow{d}{\Delta t} \arrow{r}{\dom} & \Context \arrow{d}{\Id} & \Context_{obs} \arrow[l, "\iota"'] \arrow[d, "\Id"] \\
		\DS^{\num_0}_{\Context} \arrow{r}[swap]{\dom} & \Context & \Context_{obs} \arrow[l, "\subseteq"]
	\end{tikzcd} \imply 
	\begin{tikzcd} \DSM^{\Time}_{\Context} \arrow[d, Shobuj, dotted, "\Discretize_{\Delta t}"]  \\ \DSM^{\num_0}_{\Context} \end{tikzcd}\]
	Note that both the horizontal rows in this diagram constitute two comma categories. We show later in Lemma \ref{lem:oh9d} in the Appendix that in any diagram of this pattern where vertical arrows connect two comma categories, there in induced functor between the comma categories. We call the induced functor $\Discretize_{\Delta t}$ the \emph{discretization} functor. Similarly we have 
	\begin{equation} \label{eqn:def:Dscrt:2}
		\begin{tikzcd}
			\Context \arrow[d, "\Id"] \arrow[r, "\Id"] & \Context \arrow[d, "\Id"] & \DSM^{\Time}_{\Context} \arrow{d}{\Discretize_{\Delta t}} \arrow[l, "\dom"'] \\
			\Context \arrow[r, "\Id"'] & \Context & \DSM^{\num_0}_{\Context} \arrow[l, "\dom"] 
		\end{tikzcd} \imply 
		\begin{tikzcd} \DSMO^{\Time}_{\Context} \arrow[d, Shobuj, dotted, "\Discretize_{\Delta t}"]  \\ \DSMO^{\num_0}_{\Context} \end{tikzcd}
	\end{equation}
	By a reuse of terminology we also call this a discretization functor. If the sampling interval $\Delta t$ is clear from context, it will be dropped from the subscript. It converts any dynamical system on a general mathematical space, and adherent to a general structure of time, into a discrete-time set-theoretic dynamics. Explicitly, this functor maps each $(\Omega, \Phi, \phi, \Xi)$ to $(j\Omega, j\paran{\Phi^{n \Delta t}}, j\phi, j\Xi)$, where $j\Xi = \{j\omega : \omega \in \Xi\}$. Here, the domain does not have any mathematical structure other than that of a loose set of points. This is in essence, the starting point of all data-driven numerical methods. 
	
	This completes the description of yet another link in Diagram \eqref{eqn:outline:1}. We have provided a rigorous and categorical definition of dynamical systems, observed dynamical systems, and dynamical systems being observed along orbits. Note that none of the latter two constructs required the concept of points, sets or maps. They emerge as a category or collection of diagrams. The next items to build are the category $\Seq$ and functor $\Obs$. Before that, we examine some finer properties of topological concrete categories, as established in Assumption \ref{A:concrete}.
	
	\section{Role of the concretization functor} \label{sec:cncrt}
	
	Recall that Assumption \ref{A:concrete} requires a functor $U:\Context\to \SetCat$. In this section, we examine the utility of this functor as well as the strong properties it provides to $\Context$. The functor $U$ allows simple and intuitive set theoretic properties to be adapted for use in the more complex category $\Context$.
	
	\paragraph{Epi-mono factorization} A category is said to have an epi-mono factorization if any morphism $f:a\to b$ can be factorized as a composition of an monomorphism with a epimorphism, as shown below :
	\[\begin{tikzcd} a \arrow[r, two heads] & \mathrm{im} f \arrow[r, hook] & b \end{tikzcd} ,\]
	Moreover, this factorization must be unique up to isomorphism. In our setting of topological concrete categories, monomorphisms can be regarded as injections and epimorphisms are surjections.
	
	Epi-mono factorizations are intuitive and commonly used in set theory as well as in linear algebra. The intermediary object is called the \emph{image} of the morphism $f$, as indicated in the diagram above. It resides as a sub-object of $b$.  One of the most useful consequences of Assumption \ref{A:concrete} is
	
	\begin{property} \label{Py:epi_mono}
		A topological concrete category allows epi-mono factorization of morphisms, where the image of the monomorphism part by $U$ is just a set theoretic inclusion.
	\end{property}
	
	We next investigate an even stronger form of this simple and useful property.
	
	\paragraph{Factorization via image} Consider the following simple diagram in $\SetCat$ on the left below
	\[\begin{tikzcd}
		\tilde{X} \arrow[dd, "\tilde{p}"'] \arrow[rrr, "\tilde{f}"] &&& \tilde{Y} \arrow[dd, "\tilde{q}"] \\
		\\
		\tilde{X}' \arrow[rrr, "\tilde{g}"] &&& \tilde{Y}'
	\end{tikzcd} \;\Rightarrow\; 
	\begin{tikzcd}
		\tilde{X} \arrow[dd, "\tilde{p}"'] \arrow[dr, two heads] \arrow[rrr, "\tilde{f}"] &&& \tilde{Y} \arrow[dd, "\tilde{q}"] \arrow[dl, two heads] \\
		& im( \tilde{p} ) \arrow[dl, hook] \arrow[r, dotted, "!"] & im( \tilde{q} ) \arrow[dr, hook]\\
		\tilde{X}' \arrow[rrr, "\tilde{g}"] &&& \tilde{Y}'
	\end{tikzcd}\]
	Any commutation between maps $\tilde{f}, \tilde{q}, \tilde{g}, \tilde{q}$ means that there as induced map between the images of $\tilde{p}$, $\tilde{q}$ as shown in the right diagram as a dotted arrow. The original commutation square can be thus be factored through the more basic commutation square occupying its top half. This elementary set-theoretic observation extends to $\Context$ as well : 
	
	\begin{lemma} \label{lem:U_set:1}
		Suppose Assumption \ref{A:concrete} holds and we have the following two diagrams in $\Context$ and $\SetCat$ respectively. 
		\[
		\begin{tikzcd}
			X \arrow[d, "p"'] \arrow[r, "f"] & Y \arrow[d, "q"] \\
			X' & Y'
		\end{tikzcd} \;
		\begin{tikzcd}
			UX \arrow[rr, "Uf"] \arrow[d, two heads] &  & UY \arrow[d, two heads] \\
			{\mathrm{im}\,Up} \arrow[rr, "G"]        &  & {\mathrm{im}\,Uq}      
		\end{tikzcd} \;\Rightarrow\;
		\begin{tikzcd}
			X \arrow[rr, "f"] \arrow[d, two heads] &  & Y \arrow[d, two heads] \\
			{\mathrm{im}\,p} \arrow[rr, "g"]      &  & {\mathrm{im}\,q}    
		\end{tikzcd}\]
		Then, there exists a unique morphism $g $ such that $Ug = G$ and the commutation on the right holds.
	\end{lemma}
	
	Lemma \ref{lem:U_set:1} is an example of how topological concrete categories allow set theoretic calculations to be lifted to the abstract category. The following consequence of Lemma \ref{lem:U_set:1} is of importance to us :
	
	\begin{lemma} \label{lem:U_set:2}
		Suppose Assumption \ref{A:concrete} holds and we have the commutative diagram in $\Context$ shown below on the left
		\[
		\begin{tikzcd}
			X \arrow[rr, "f"] \arrow[d, "p"] &  & Y \arrow[d, "q"] \\
			X' \arrow[rr, "k"]               &  & Y'              
		\end{tikzcd} \;\Rightarrow\;
		\begin{tikzcd}
			UX \arrow[rr, "Uf"] \arrow[d, two heads] &  & UY \arrow[d, two heads] \\
			{\mathrm{im}\,Up} \arrow[rr, "Ug"]        &  & {\mathrm{im}\,Uq}      
		\end{tikzcd} ,\, 
		\begin{tikzcd}
			X \arrow[rr, "f"] \arrow[dd, bend right=30, "p"'] \arrow[d, two heads] &  & Y \arrow[d, two heads] \arrow[dd, bend left=30, "q"] \\
			{\mathrm{im}\,p} \arrow[rr, "g"] \arrow[d, hook] &  & {\mathrm{im}\,q} \arrow[d, hook] \\
			X' \arrow[rr, "k"] & & Y' 
		\end{tikzcd}
		\]
		then there is a morphism $g : \mathrm{im} (p) \to \mathrm{im} (q)$ s.t. the  two commutations on the right holds.
	\end{lemma}
	
	The epi-mono factorization property guarantees the existence of images of morphisms. The essence of Lemma \ref{lem:U_set:2} is that any commutation square factors through the images of any pair of opposing morphisms. While this property is intuitive and even trivial for $\SetCat$, it has to be established by the diagrammatic calculus of categories for an abstract category such as in Assumption \ref{A:concrete}. Lemma \ref{lem:U_set:2} will be used to establish the categorical notion of \emph{orbits}.
	
	\paragraph{Invariance} Take any $\DSO$ object $A \xrightarrow{\omega} \Omega \xrightarrow{\Phi^t} \Omega$ and introduce the notations
	\[\tilde{\Phi} := U\Phi , \; \tilde{A} := U( \omega(A) ), \; \tilde{\Omega} := U(\Omega) .\]
	to denote the usual set-theoretic description of the dynamics.
	This leads to the following diagram of sets and maps in the category $\SetCat$ :
	\[\begin{tikzcd} [scale cd = 0.7]
		& \tilde{\Phi}^{t_1} \paran{ \tilde{A} } \arrow[bend left=15, rrrr, dotted, pos=0.8, "\tilde{\Phi}^{t+s}"] \arrow[dl, "\subset"] \arrow[ rr, pos=0.8, "\tilde{\Phi}^t" ] & & \tilde{\Phi}^{t_1+t} \paran{ \tilde{A} } \arrow[dl, "\subset"] \arrow[ rr, pos=0.8, "\tilde{\Phi}^s" ] & & \tilde{\Phi}^{t_1+t+s} \paran{ \tilde{A} } \arrow[dl, "\subset"] \\
		\bigcup_{t\in\Time} \tilde{\Phi}^{t} \paran{ \tilde{A} } \arrow[dd, Itranga, pos=0.4, "\subset"] \arrow[bend right=15, rrrr, dotted, pos=0.9, "\tilde{\Phi}^{t+s}"', Holud] \arrow[ rr, "\tilde{\Phi}^t", Holud ] & & \bigcup_{t\in\Time} \tilde{\Phi}^{t} \paran{ \tilde{A} } \arrow[dd, Itranga, pos=0.4, "\subset"] \arrow[ rr, "\tilde{\Phi}^s", Holud ] & & \bigcup_{t\in\Time} \tilde{\Phi}^{t} \paran{ \tilde{A} } \arrow[dd, Itranga, pos=0.4, "\subset"]  \\
		& \tilde{\Phi}^{t_2} \paran{ \tilde{A} } \arrow[bend right=15, rrrr, dotted, pos=0.8, "\tilde{\Phi}^{t+s}"'] \arrow[ul, "\subset"'] \arrow[ rr, pos=0.8, "\tilde{\Phi}^t"' ] & & \tilde{\Phi}^{t_2+t} \paran{ \tilde{A} } \arrow[ul, "\subset"'] \arrow[ rr, pos=0.8, "\tilde{\Phi}^s"' ] & & \tilde{\Phi}^{t_2+t+s} \paran{ \tilde{A} } \arrow[ul, "\subset"'] \\
		\tilde{\Omega} \arrow[bend right=15, rrrr, dotted, pos=0.9, "\tilde{\Phi}^{t+s}"', Akashi] \arrow[ rr, "\tilde{\Phi}^t", Akashi ] & & \tilde{\Omega} \arrow[ rr, "\tilde{\Phi}^s", Akashi ] & & \tilde{\Omega}
	\end{tikzcd}\]
	The yellow and purple rows represent two separate dynamical systems , and the vertical red arrows serve as a semi-conjugacy. By the initial lift property of topological concrete categories along with Lemma \ref{lem:U_set:2}, this leads to the creation of a semi-conjugacy of dynamics in $\Context$
	\[\begin{tikzcd}
		\Orbit(A; \Phi) \arrow[d, Itranga, "\subseteq"'] \arrow[rr, Holud, "\Phi^t"] && \Orbit(A; \Phi) \arrow[d, Itranga, "\subseteq"] \\
		\Omega \arrow[rr, Akashi, "\Phi^t"'] && \Omega 
	\end{tikzcd}\]
	The object $\Orbit(A; \Phi)$, interpreted as the \emph{orbit} of the sub-object $A$ under the dynamics $\Phi$. is itself a sub-object or semi-dynamical system of $\paran{\Omega, \Phi^t}$. Intuitively,  $\Orbit(A; \Phi)$ represents the minimal forward invariant set that contains $A$. Structurally, $\Orbit(A; \Phi)$ represents the minimal dynamical system that is semi-conjugate to $\paran{\Omega, \Phi^t}$ and contains a copy of $A$. This leads to a functor 
	\begin{equation} \label{eqn:def:ForwInv:1}
		\ForwInv : \DSO \to \DS, \quad \paran{A, \Omega, \Phi^t} \mapsto \paran{\Orbit(A; \Phi), \Phi^t}
	\end{equation}
	One analogously has a functor
	\begin{equation} \label{eqn:def:ForwInv:2}
		\ForwInv : \DSMO \to \DSM, \quad \paran{A, \Omega, \Phi^t, \phi} \mapsto \paran{\Orbit(A; \Phi), \Phi^t, \phi}
	\end{equation}
	where we have reused the notation for simplicity. The notion of invariance enables the concept of sequence spaces, which embody both observations as well as autonomous dynamical systems. 
	
	\section{Shift spaces} \label{sec:shift} 
	
	Shift spaces \cite[e.g.]{MarcusLind1995, BDWY2020} are secondary dynamical systems associated to any dynamical system with measurement. They track the dynamics induced on all possible sequences of measurements. In this section we present a purely categorical notion of subshifts, that retains all the associated notions of symbols and strings. A key tool is the existence of infinite products guaranteed by Property \ref{Py:product}. Given any $\calC$-object $\mathbb{A}$ one can form the infinite product $\mathbb{A}^{\num}$. This product object comes with a collection of projection morphisms indexed as $\pi^{\num}_{i}$ as shown below in yellow :
	\[\begin{tikzcd}
		& & & & \Holud{\mathbb{A}} \\
		X \arrow[rrr, dotted, pos=0.7, "\exists \prod_{i\in \num} f_i", Shobuj] \arrow[rrrrd, bend right = 10, "f_{i_2}"'] \arrow[rrrru, bend left = 10, "f_{i_1}"] & & & \Holud{ \mathbb{A}^{\mathbb{N}} } \arrow[ur, "\pi^{\num}_{i_1}", Holud] \arrow[dr, "\pi^{\num}_{i_2}"', Holud] \arrow[r, "\cdots", Holud] & \Holud{\vdots} \\
		& & & & \Holud{\mathbb{A}}
	\end{tikzcd}\]
	When $\calC$ equals $\SetCat, \Topo$ or $\MeasCat$ these morphisms represent the usual coordinate-wise projections. In a general category it is defined by a universal property : any collection of morphisms $\SetDef{f_i : X\to \mathbb{A}}{i\in\num}$ as shown above must factor through a unique morphism $X\to \mathbb{A}^{\num}$ as shown in green above. This morphisms is suggestively labeled as the product of the $f_i$s. Thus a product is the universal object that encodes joint information of its factors.
	
	\begin{definition}[Shift morphism]
		Let  $\Context$ be as in Assumption \ref{A:concrete}. For an object $\mathbb{A}$ in $\Context$, the \emph{shift morphism} associated to $\mathbb{A}$ is the unique morphism $\sigma: \mathbb{A}^{\num} \to \mathbb{A}^{\num}$ which enables the following commuting diagram:
		\[
		\begin{tikzcd}
			\mathbb{A}^{\num} \arrow[rr, "\sigma", dashed] \arrow[rd, "\pi^{\num}_{n+1}"'] &   & \mathbb{A}^{\num} \arrow[ld, "\pi^{\num}_n"] \\
			& \mathbb{A} &
		\end{tikzcd} , \quad \forall n\in \num.
		\]
		where $\pi^{\num}_n$ is the projection to $n$-th component.
	\end{definition}
	
	The object $\mathbb{A}$ plays the role of an \emph{alphabet}. Note that any $\Context$-object $\mathbb{A}$ thus produces an domain $\mathbb{A}^{\num}$ and an endomorphism $\sigma$. The resulting  discrete time dynamics is precisely the \emph{full shift-space} with alphabet $\mathbb{A}$. A finer object of a full-shift is a \emph{subshift} :
	
	\begin{definition} [Subshift spaces]
		Let $\Context$ be as in Assumption \ref{A:concrete}. Given a $\Context$-object $\mathbb{A}$, a subshift on alphabet $\mathbb{A}$ is a subobject $m:S \hookrightarrow \mathbb{A}^{\num}$ satisfying the following commutative diagram:
		\[
		\begin{tikzcd} [column sep = large]
			S \arrow[rr, bend left=20, dotted, Shobuj, "\sigma|S"] \arrow[d, "m", hook] \arrow[r, two heads] & \sigma(S) \arrow[rd, hook, "m'"] \arrow[r, hook] & S \arrow[d, "m", hook] \\
			\mathbb{A}^{\num} \arrow[rr, "\sigma"] && \mathbb{A}^{\num}
		\end{tikzcd}
		\]
	\end{definition}
	
	In the diagram above, the morphisms $m'$ and $\sigma \rvert S$ are defined by the epi-mono factorization of $\sigma \circ m$. They are guaranteed to exist regardless of the nature of the morphism $m$. The extra requirement is that the injection $m'$ factorizes through $m$ as shown in the upper right commuting triangle. As a result, the composite morphism created by the top horizontal row becomes a self-map on $S$, which we denote as $\sigma|S$. Note that $\sigma|S$, $\sigma$ and $m$ together create a semi-conjugacy as discussed earlier in Section \ref{sec:cat}. Sub-shifts are thus sub-systems embedded within a full shift space. We next discover their own categorical structure.
	
	\begin{definition} [Category of subshifts]
		Let $\Context$ be as in Assumption \ref{A:concrete}.
		The subcategory of $\DS_{\num_0}^{\Context} = \Functor{\num_0}{\Context}$ generated by the collection of all subshifts form a category $\Seq_{\Context}$.
	\end{definition}
	
	If the context $\Context$ is clear, we drop the subscript to denote the category simply as $\Seq$. By virtue of being a sub-category, a morphism between two subshifts $S, S'$ comprises of a $\Context$-morphism $f$ such that the following diagram commutes:
	\[
	\begin{tikzcd}
		S \arrow[d, "f"'] \arrow[r, two heads] & \sigma S \arrow[r, hook] & S \arrow[d, "f"] \\
		S' \arrow[r, two heads] & \sigma S' \arrow[r, hook] & S'
	\end{tikzcd}
	\]
	We denote the underlying set $\mathbb{A}$ of object $S$ by $|S|$. It is referred to as the \emph{alphabet} of the shift space $S$. The most commonly investigated subshifts correspond to those with a finite alphabet $\mathbb{A}$. Suppose that the phase space of a dynamical system is partitioned into $\abs{ \mathbb{A} }$ subsets with each piece of the partitioned indexed by some element of $\mathbb{A}$. Then each orbit in the phase space is associated with an itinerary of the partition-cells it visits. This itinerary is just an element of $\mathbb{A}^{\num}$. The collection of all possible itineraries is a subshift with alphabet $\mathbb{A}$. 
	Shift spaces are dynamical systems of their own right, situated in $\Context$. The space $\Seq$ of subshifts has a finer sub-structure to it, based on the notion of decidability. Not every morphism in $\Seq$ is decidable in the sense that the image can be determined by finite length observation of the sequence.
	
	\begin{example} Consider two objects in $\Seq_\SetCat$: the full shift $T = \{0,1\}^\num$ on $2$ symbols, and the subshift
		\[S = \{\xi \in \{0,1\}^\num : \text{ there exists }N\text{ s.t. } \xi_k = 0 \text{ for all } k \geq N\}.\]
		We set up a map $\phi: S \to T$ by
		\[\phi(\xi)_k = 
		\begin{cases}
			0 & \text{ if } \xi_i = 0 \text{ for all } i\geq k\\
			1 &\text{ otherwise }
		\end{cases} \]
		Then, $\phi$ is a morphism. However, $\phi(\xi)$ cannot be determined by a finite observation of $\xi$.
	\end{example}
	
	To codify the meaning of decidability, we need the notion of \emph{words}. To establish these notions we take note of a unique  morphism $\pi_{\num\to n}: \mathbb{A}^{\num} \to \mathbb{A}^n$ defined by the following commutations:
	\[
	\begin{tikzcd}
		\mathbb{A}^{\num} \arrow[rr, "\pi_{\num\to n}", dotted] \arrow[rd, "\pi^{\num}_{i}"'] & & \mathbb{A}^n \arrow[ld, "\pi^n_i"] \\
		& \mathbb{A} & 
	\end{tikzcd}, 
	\quad \forall 1\leq i \leq n.
	\]
	Again, this projection $\pi_{\num\to n}$ is easy to decipher for Cartesian products. In a general categorical formulation, they are defined structurally by the commutations above.
	
	\begin{definition} [Words]
		Fix an object $S$ in $\Seq$ and an $n \in \num$. consider the top horizontal row of the following diagram :
		\[
		\begin{tikzcd}
			S \arrow[r, hook] \arrow[rd, two heads] & \mathbb{A}^{\num} \arrow[r, "\pi_{\num\to n+1}"] & \mathbb{A}^{n+1} \\
			& \word_n(S) \arrow[ru, hook] & 
		\end{tikzcd}
		\]
		The epi-mono factorization of the top row leads to an object labeled $\word_n(S)$. It is called the collection of $n$ words of the subshift $S$.
	\end{definition}
	
	Intuitively the $n$-th word function associates to each object $S$ in $\Seq$ the set of all finite subsequences of length $n+1$ occurring in $S$. 
	
	\begin{example}
		When $\Context = \SetCat$ and $S$ is an object in $\Seq$, then for each $n\in \num$ the word functor acts as 
		\[ \word_n(S) := \{\xi_0 \xi_1 \cdots \xi_n : \xi \in S\} \]
	\end{example}
	
	\begin{definition} [Decidable morphisms]
		A morphism $f:S \to S'$ in $\Seq$ is an \emph{$n$-morphism} if there is a morphism $f_0: \word_n(S) \to \word_0(S')$ such that :
		\[
		\begin{tikzcd}
			S \arrow[rr, "f"] \arrow[d, two heads] & & S' \arrow[d, two heads] \\
			\word_n(S) \arrow[rr, "f_0"] & & \word_0(S') 
		\end{tikzcd}
		\]
	\end{definition}
	
	The morphism $f_0$ in this definition is determined uniquely from morphism $f$. We will call this to be the \emph{generator} of the morphism $f$. The uniqueness follows from the surjectivity of the vertical $\word$-morphisms. The collection of $n$-morphisms for various $n$-s are called \emph{decidable}, \emph{finitely decidable} or simply \emph{finite} morphisms.
	
	\begin{example}
		In a set theoretic context, a morphism $\phi: S \to T$ in $\Seq$ is an \emph{$n$-morphism} if the first $n+1$ entries of a sequence determines the first entry of the image. Namely,     if $\xi, \xi' \in S$ satisfies $\xi_0 \xi_1 \cdots \xi_n = \xi'_0 \xi'_1 \cdots \xi'_n$, then we have $\phi(\xi)_0 = \phi(\xi')_0$. 
	\end{example}
	
	\begin{example}
		\begin{enumerate}
			\item The identity map $\Id_S:S \to S$ is a 0-morphism.
			\item The shift map $\sigma: S \to S$ is a 1-morphism.
			\item An $n$-morphism is an $m$-morphism for all $m \geq n$.
		\end{enumerate}
	\end{example}
	
	Data processing procedures, such as moving average or difference, can be formulated as finite morphisms. To give an explicit description of the structure of $n$-morphisms, we need the following observation.
	
	\begin{lemma} \label{lem_rep}
		Let $f:S \to S'$ be an $n$-morphism, and $f_0$ its generator. Then, we have
		\[
		(Uf)(\xi)_i = (Uf_0) \paran{ \xi_i\xi_{i+1}\cdots\xi_{i+n} }, \quad \forall \xi \in U(S), \; i \in \num.
		\]
	\end{lemma}
	
	Lemma \ref{lem_rep} is proved in Section \ref{sec:app:subshift} in the Appendix. The image under $U$ of any subshift is also a set-theoretic subshift. Lemma \ref{lem_rep} provides an explicit formula for the generator of this subshift.
	
	\paragraph{Terminal object} A terminal object $\bar{x}$ in a category $\calX$ has the property that for any object $x$ of $\calX$, there is a unique morphism $x \to \bar{x}$. A terminal object is usually denoted as $1_{\calX}$. The right slice category $\Comma{1_{\calX}}{\calX}$ of $1_{\calX}$ will be called the category of \emph{pointed objects} of $\calX$. The objects are morphisms $a : 1_{\calX} \to A$, and a morphism between any two such objects $a : 1_{\calX} \to A$ and $a' : 1_{\calX} \to A'$ is a morphism $\phi : A \to A'$ such that $\phi \circ a = a'$. Note that each object of $\Comma{1_{\calX}}{\calX}$ requires a choice of an object $A$, followed by the choice of a morphism $a : 1_{\calC} \to A$. The latter is interpreted as a point in the object $a$, hence the name. Fortunately, such an object is guaranteed to us : 
	
	\begin{property} \label{Py:terminal}
		A category $\Context$ satisfying Assumption \ref{A:concrete} has a terminal object $1_{\Context}$.
	\end{property}
	
	\begin{definition} [Strings]
		Let $S$ be a subshift on alphabet $\mathbb{A}$. A \emph{string} of length $n+1$ is a morphism of the form
		\[\begin{tikzcd} 1_{\Context} \arrow[rr, "w"] && \word_n(S) \end{tikzcd}\]
	\end{definition}
	
	\begin{example}
		When $\Context$ is either $\SetCat, \Topo$ or $\MeasCat$, any morphism originating from the terminal object represents a point. In these context, a string fits the usual definition of being a point in a collection of words.
	\end{example}
	
	\begin{definition} [Cylinders]
		Let $S$ be a subshift on alphabet $\mathbb{A}$, and $1_{\Context} \xrightarrow{w} \word_n(S)$ be a string. Then the cylinder around the string $w$ is the unique object $\Cylinder(w)$ created from the following pullback diagram
		\[\begin{tikzcd} [column sep = large]
			\Cylinder(w) \arrow[r, dotted, "!"] \arrow[d, dotted, hook] & 1_{\Context} \arrow[d, "w"] \\
			S \arrow[r, "\word_n"'] & \word_n(S)
		\end{tikzcd}\]
	\end{definition}
	
	\begin{example}
		When $\Context$ is either $\SetCat, \Topo$ or $\MeasCat$, the cylinder $\Cylinder(w)$ denotes the subset or subspace of $S$ whose first $n+1$ symbols coincide with the symbols in $w$. 
	\end{example}
	
	\begin{lemma} \label{lem:cylinder}
		Let Assumption \ref{A:concrete} hold. Then for every string $w$, the cylinder $\Cylinder(w)$ is a subobject of $S$.
	\end{lemma}
	
	Lemma \ref{lem:cylinder} is proved in Section \ref{sec:app:subshift} in the Appendix.
	
	\begin{lemma} \label{lem_steps}
		Let $f:S \to S'$ be an $n$-morphism. Then, we have a family of morphisms $f^{(m)}: \word_{n+m}(S) \to \word_{m}(S') $ indexed by $m\in\num$ such that the following commutations hold :
		\[
		\begin{tikzcd}
			S \arrow[rr, "f"] \arrow[d, two heads] & & S' \arrow[d, two heads] \\
			\word_{m+n}(S) \arrow[rr, "f^{(m)}"] & & \word_m(S') 
		\end{tikzcd} .
		\]
	\end{lemma}
	
	Lemma \ref{lem_steps} is proved in Section \ref{sec:app:subshift} in the Appendix. Another important observation to be made is that finite morphisms are closed under composition : 
	
	\begin{lemma} \label{lem:fin_compose}
		The composition of an $n$-morphism with an $m$-morphism produces an $(m+n)$-morphism.
	\end{lemma}
	
	Lemma \ref{lem:fin_compose} is proved in Section \ref{sec:app:subshift} in the Appendix. The proof of Lemma \ref{lem:fin_compose} explicitly provides the generator for the morphism, whose law of generation is given in \eqref{eqn:def:induced:1}. This statement can be expanded as
	
	\begin{lemma} \label{lem:def:induced:2}
		Given a $m$-morphism $\phi : S\to S'$ and an $l$-morphism $\phi' : S'\to S''$, $\phi$, one has
		\begin{equation} \label{eqn:def:induced:2}
			(\phi' \circ \phi )^{(m)} = \phi'^{(m)} \circ \phi ^{(m+l)} .
		\end{equation}
	\end{lemma} 
	
	Lemma \ref{lem:def:induced:2} is a direct consequence of Lemma \ref{lem_steps}. Lemma \ref{lem_steps} lays down the composition rule for the finite versions of subshift morphisms. 
	
	\begin{example}
		Again consider the case when $\Context=\SetCat$. Given a morphism $\phi$, let $\hat{\phi} = \pi_0 \circ \phi$ denote the projection to its first coordinate. Thus, $\phi$ is an $n$-morphism if $\hat{\phi}$ is constant on all $n$-cylinders. A morphism must commute with the shift map. Thus we must have
		\[\pi_k \circ \phi = \pi_0 \circ \sigma^k \circ \phi = \pi_0 \circ \phi \circ \sigma^k = \hat{\phi} \circ \sigma^k ,\]
		for every index $k\in \num_0$. Thus finitely defined morphisms between shift spaces are completely defined by their first coordinate-projections. Given any $n$-morphism $\phi : S \to S'$, one thus has an induced map
		\begin{equation} \label{eqn:def:induced:1}
			\phi^{(m)}_* : \word_{m+n}(S) \to \word_{m}(S'), \quad \paran{s_0, \ldots, s_{m+n}} \mapsto \paran{s'_0, \ldots, s'_{m}}, \; s'_j := \hat{\phi} \paran{s_{j}, \ldots, s_{j+n}} ,
		\end{equation}
		for every index $m\in \num$. This map $\hat{\phi}$ will be called the \emph{generator} for the morphism $\phi$.
	\end{example}
	
	\begin{example}
		Consider the case when $\Context = \SetCat$. If $\phi: S \to T$ is an $n$-morphism and $\psi:T \to Y$ is an $m$-morphism, then $\psi \phi: S \to Y$ is an $m+n$-morphism. We can check that the generator of the morphism $\psi \phi$ is   
		\[ \paran{ s_0 , \ldots, s_{m+n} } \,\mapsto\,  \hat{\psi}\paran{ \hat{\phi} \paran{ s_0, \ldots , s_n } , \ldots,\hat{\phi} \paran{ s_{m}, \ldots , s_{m+n} } } . \]
	\end{example}
	
	Lemma \ref{lem:fin_compose} leads to the following discovery of categorical structure.
	
	\begin{theorem} \label{thm:gfkf0}
		Finite morphisms define a subcategory $\Seq_{<\infty}$ of $\Seq$. Similarly, 0-morphisms define a subcategory $\Seq_0$ of $\Seq$.
	\end{theorem}
	
	We have thus progressed to another node $\Seq$ from Diagram \eqref{eqn:outline:1}. The objects here are quite distant from the $\DSMO$ objects from where they may originate. A data-driven algorithm only has available, or finite sets or sequences of data. We later encapsulate their properties in a category labeled $\TSD$. There is no pre-existing information of an explicit dynamics within the data, or of a topological or measure space. Objects in $\Seq$ are midway between the objects in $\DSMO$ and $\TSD$. A $\Seq$ object is collection of infinitely long records of measurements on the original space. At the same time, they also have an in-built dynamics due to the shift map.
	\section{Measurement along orbits} \label{sec:msrmnt} 
	
	We now come to the realization that shift spaces arise naturally from dynamical systems with measurements.  In the previous section we arrived at a purely categorical understanding of Cartesian products $\mathbb{A}^{\num}$. Such products naturally arise when a dynamical system is being observed. An observation $\phi : \Omega \to \iota Y$  involves an object $Y$ from $\Context_{obs}$. The sequences of observations at intervals of $\Delta t$ naturally lead to sequences in $(\iota Y)^{\num}$. We begin by studying some fundamental connections of this phase with the original domain.
	
	Our first observation is that For each object $(\Omega, \Phi, \phi)$ in $ \DSM$, we have a morphism $\xi_{(\Omega, \Phi, \phi)}:\Omega \to (\iota Y)^{\num}$ defined by the commutations
	\[
	\begin{tikzcd}
		\Omega \arrow[rr, dotted, "\xi", Shobuj] \arrow[rd, "\phi \circ \Phi^n"'] & & (\iota Y)^{\num} \arrow[ld, "\pi_n"] \\
		& \iota Y & 
	\end{tikzcd} , \quad 
	\forall n\in \num_0.
	\]
	This morphism $\xi$ may be denoted as $\xi_{(\Omega, \Phi, \phi)}$ to indicate the dependence on its factors. We next take two copies of the diagram above to get the diagram on the left
	\begin{equation} \label{eqn:def:Obs:DSM}
		\begin{tikzcd}
			& & \calY^{\num} \arrow[ld, "\pi^{(\num)}_n"'] \\
			\Omega \arrow[r, "\phi \circ \Phi^n"] \arrow[rru, "\xi", dotted, bend left] & \calY & \\
			\Omega \arrow[u, "\Phi"] \arrow[r, "\phi \circ \Phi^n"'] \arrow[rrd, "\xi"', dotted, bend right] & \calY & \\
			& & \calY^{\num} \arrow[lu, "\pi^{(\num)}_n"] \arrow[uuu, "\sigma"]
		\end{tikzcd} \imply
		\begin{tikzcd}
			\Omega \arrow[r, two heads] & \mathrm{im} \xi \arrow[r, hook] & \calY^{\num} \\
			\Omega \arrow[u, "\Phi"] \arrow[r, two heads] & \mathrm{im} \xi \arrow[r, hook] \arrow[u, "\sigma", dotted] & \calY^{\num} \arrow[u, "\sigma"']
		\end{tikzcd} \imply 
		\begin{tikzcd} \DSM \arrow[d, Shobuj, "\Obs"] \\ \Seq \end{tikzcd}
	\end{equation}
	The left-diagram shows how the dynamics $\Phi$, the measurements $\phi\circ \Phi^n$ along the orbits, and shift on the symbol space $\calY^{\num}$ and the projection maps commute.  Lemma \ref{lem:U_set:2} may be applied to the peripheral commuting square of this diagram to get the commutative diagram in the middle. Thus induced endomorphism on the image of $\xi$ has also been denoted as $\sigma$. Thus we get a dynamical system which is subobject or sub-system of the full shift. This can be summarized as the existence of a functor $\Obs$ as shown on the right above. One can perform the same steps with $\DSMO$. Recalling the notation from \eqref{eqn:def:ForwInv:2} we can draw :
	\begin{equation} \label{eqn:def:Obs:DSMO}
		\begin{tikzcd} [scale cd = 0.4]
			& & & \calY^{\num} \arrow[ld, "\pi^{(\num)}_n"'] \\
			\ForwInv (A; \Phi) \arrow[r, hook] \arrow[urrr, bend left, dotted, "\xi"] & \Omega \arrow[r, "\phi \circ \Phi^n"] & \calY & \\
			\ForwInv (A; \Phi) \arrow[r, hook] \arrow[u, "\Phi"] \arrow[rrrd, "\xi"', dotted, bend right] & \Omega \arrow[u, "\Phi"] \arrow[r, "\phi \circ \Phi^n"'] & \calY & \\
			& & & \calY^{\num} \arrow[lu, "\pi^{(\num)}_n"] \arrow[uuu, "\sigma"]
		\end{tikzcd} \imply
		\begin{tikzcd} [scale cd = 0.4]
			\ForwInv (A; \Phi) \arrow[r, two heads] & \mathrm{im} \xi \arrow[r, hook] & \calY^{\num} \\
			\ForwInv (A; \Phi) \arrow[u, "\Phi"] \arrow[r, two heads] & \mathrm{im} \xi \arrow[r, hook] \arrow[u, "\sigma", dotted] & \calY^{\num} \arrow[u, "\sigma"']
		\end{tikzcd} \imply 
		\begin{tikzcd} [scale cd = 0.6] \DSMO \arrow[d, Shobuj, "\Obs"] \\ \Seq \end{tikzcd}
	\end{equation}
	We now formally state the functor property of these correspondences $\Obs$ :
	
	\begin{theorem} [Observation functor] \label{thm:Obs_functor}
		Suppose Assumptions \ref{A:concrete} and \ref{A:obs} hold. Then the correspondences shown in \eqref{eqn:def:Obs:DSM} and \eqref{eqn:def:Obs:DSMO} are functorial.
	\end{theorem}
	
	Theorem \ref{thm:Obs_functor} is proved in Section \ref{sec:app:Seq_mes} in the Appendix, with the help of Lemma \ref{lem:U_set:1}. The functor $\Obs$ thus describes a systematic way in which measurements along orbits is converted into sequences, and how semiconjugacy relations are also preserved as subshift maps. In fact, the morphisms in $\DSMO^{\num}$ are converted into $0$-morphisms in $\Seq$. The objects in $\Seq$ are a record of all possible sequences of measurements one can obtain from the $\DSMO$ object.
	
	\begin{example}
		When $\Context = \SetCat$, the action of the functor $\Obs : \DSMO_{\num, \text{Set}} \to \Seq_{<\infty}$ can be illustrated as :
		\[  
		\begin{tikzcd} X \arrow[drr, "\subset"', shift right = 4pt] && \\ \Omega \arrow{drr}[swap, bend right = 20]{\phi} &&  \Omega \arrow{ll}{\Phi^{n } } \\ && Y \end{tikzcd}
		\begin{tikzcd} {} \arrow[rr, mapsto] && {} \end{tikzcd}
		\begin{array}{c}
			\left\{ \braces{ \phi \paran{ \Phi^{ (n + n_0) } x } }_{n=0}^{\infty} \right.: 
			\left. x \in X, \; n_0 \in \num \right\}
		\end{array} .
		\]
		A typical morphism as in \eqref{eqn:def:DSM_morphism} induces the following $\Seq$-morphism : 
		\[ \braces{ \phi \paran{ \Phi^{n+n_0}(x) } }_{n=0}^{\infty} \mapsto \braces{ A\circ \phi \paran{ \Phi^{n+n_0}(x) } }_{n=0}^{\infty} = \braces{ \phi' \paran{ \Phi'^{n+n_0}(h\paran{x}) } }_{n=0}^{\infty} .\]
	\end{example}
	
	This completes the description of yet another link in Diagram \eqref{eqn:outline:1}. 
	
	Since $\Seq$ is a subcategory of $\DS^{\num}$ there is an obvious inclusion functor :
	\[ \PathF : \Seq
	\to \Functor{\num_0}{\SetCat} = \DS^{\num_0} .  \]
	which assigns to any subshift $S\subset \mathbb{A}^{\num}$ the dynamical system $(S, \sigma)$. We have talked about measurements over an entire orbit or trajectory. We have represented an orbit by its initial state, thereby providing an indirect characterization of orbits. We now have the ingredients for a direct description of the notion of an orbit. Intuitively, an orbit of a point is the smallest dynamical system that contains all the forward iterates of that point. This is formalized by the \emph{orbit} functor $\Orbit$ shown below as a dashed arrow : 
	\begin{equation} \label{eqn:def:Orbit:1}
		\begin{tikzcd} [row sep = large, scale cd = 1]
			\DSMO\arrow[dashed, d, "\Orbit_{\Delta t}"', Shobuj] \arrow[rr, "\Discretize_{\Delta t}"] &&\DSMO^{\num} \arrow[d, "\Obs"]\\
			\DS^{\num}
			&&\Seq \arrow[ll, "\PathF"]  
		\end{tikzcd}
	\end{equation}
	With all the concepts in place, we can proceed to define the final category of objects of Diagram \eqref{eqn:outline:1} - time series data. The concept of sampling and subshifts will play an important role.
	\section{Time series data} \label{sec:tsd}
	
	Data-driven reconstruction algorithms abound in various fields of dynamical systems. The overall goal is to create a mathematical object from some collection $\calC$ of data. To compare different techniques or to evaluate a particular technique, one needs a metric or means of comparison between the members of $\calC$. In our case the collection is $\DS$ and comparison is provided by its categorical structure. The other ingredient needed is way to model data itself as a mathematical object. While the nature of data can be very diverse, there are some basic principles to keep in mind when defining data as a mathematical object.  
	\begin{enumerate} [(i)]
		\item All algorithms on a finite memory machine must receive finite datasets as inputs. 
		\item Different algorithms may receive different data-sets and thus the comparison should not be based on a particular choice of dataset.
		\item The performance of an algorithm should be expected to improve with an increase in dataset. Thus an algorithm should be evaluated by its limiting results as the dataset increases in information.
		\item Thus the definition of a dataset must allow a rigorous notion of increase in information or content.
		\item The definition of a dataset must be able to incorporate the increasing information content in time series.
	\end{enumerate}
	
	We begin by defining finitely sized data set objects : 
	
	\begin{definition} [Category of finite sequences]
		This is the category $\FinSeq$ in which objects are triples $(\mathbb{A}, n, X)$ with $\mathbb{A}$ being a $\calC$-object, $n$ an integer, and $X$ a sub-object of $\mathbb{A}^n$. A morphism between two such objects $(\mathbb{A}, n, X)$ and $(\mathbb{A}', n', X')$ is an usual $\calC$-morphism $X\to X'$.
	\end{definition}
	
	Any such subobject $X \hookrightarrow \mathbb{A}^n$ will be denoted as a triple $(n,\mathbb{A},X)$ or more simply just as $X$. The integer $n$ will be called the \emph{length} of the finite sequence, and is represented as $\len X$. The set $\mathbb{A}$ is interpreted as the \emph{alphabet} of $X$, and is represented as $|X|$, and the integer $n$. A $\FinSeq$-object represents a chunk of data, one of the simplest units of information. Any numerical method operates on one such data chunk at a time. An infinite sequence of these chunks create a time-series object.
	
	\begin{definition} [Start and finish morphisms]
		Let $\Context$ be as in Assumption \ref{A:concrete} and $\mathbb{A}$ be an object in $\Context$. For each $n$, there are morphisms $\start_n:\mathbb{A}^n \to \mathbb{A}^{n-1}$ and $\finish_{n}:\mathbb{A}^n \to \mathbb{A}^{n-1}$ defined by the following diagrams:
		\[
		\begin{tikzcd}
			\mathbb{A}^{n} \arrow[rr, "\start_n", dotted] \arrow[rd, "\pi_{i+1}^{(n)}"'] & & \mathbb{A}^{n-1} \arrow[ld, "\pi_i^{(n-1)}"] \\
			& \mathbb{A}  
		\end{tikzcd} , \;
		\begin{tikzcd}
			\mathbb{A}^n \arrow[rd, "\pi^{(n)}_i"'] \arrow[rr, "\finish_{n}", dotted] & & \mathbb{A}^{n-1} \arrow[ld, "\pi^{(n-1)}_i"] \\
			& \mathbb{A}
		\end{tikzcd} ; \quad
		\forall 1\leq i\leq n-1
		\]
		These morphisms $\start_n$ and $\finish_n$ will be called the \emph{start} and \emph{finish} morphisms respectively.
	\end{definition}
	
	The symbols $\start$ and $\finish$ respectively represent the act of deleting the starting symbol, and final symbol. The simple ideas of a start and finish morphism are the basis of the notion of data. 
	
	\begin{definition} [Time series data object]
		Let $\Context$ be as in Assumption \ref{A:concrete}. A $\Context$-valued time series data comprises of (i) an alphabet $\mathbb{A} \in\Context$; (ii) a sequence $\{X_n\}_{n=0}^\infty$ of $\Context$-objects; (iii) and morphisms $\{\start_n, \finish_{n}:X_{n+1} \to X_n\}_n$ such that
		\begin{enumerate} [(i)]
			\item $X_n \hookrightarrow \mathbb{A}^{n+1}$ : $X_n$ is a finite sequence with alphabet $\mathbb{A}$ and length $n+1$.
			\item The following commutations hold
			\[
			\begin{tikzcd}
				X_n \arrow[rr, "{\start_n, \finish_{n}}"] \arrow[d, hook] & & X_{n-1} \arrow[d, hook] \\
				\mathbb{A}^{n+1} \arrow[rr, "{\start_{n+1},\finish_{n+1}}"] & & \mathbb{A}^{n} 
			\end{tikzcd}
			\]
		\end{enumerate}
	\end{definition} 
	
	Note the repeat of notations in the two horizontal rows. They convey the intuition that the morphisms $\start_i, \finish_i: X_i \to X_{i-1}$ which are a restriction of start and finish morphisms on the full symbolic space $\mathbb{A}^{n+1}$. Thus intuitively, a time series data consists of a sequence $X = \{X_i\}_i$ ($i = 0,1,2, \cdots$) of objects in $\FinSeq$ with a common alphabet $\mathbb{A}$ and a collection of start and finish morphisms representing data deletion.
	
	\begin{example}
		When $\Context = \SetCat$, then
		\[ \start_i (x_0 x_1\cdots x_i) = x_1\cdots x_i, \quad \finish_i (x_0 x_1\cdots x_i) = x_0 x_1\cdots x_{i-1} . \]
	\end{example}
	
	We represent a time series data graphically as 
	\begin{equation} \label{eqn:def:TSD}
		\begin{diagram}
			\node{X_0}  \node{X_1}\arrow{w,t}{\start_1}\arrow{w,b}{\finish_1}\node{X_2}\arrow{w,t}{\start_2}\arrow{w,b}{\finish_2}\node{X_3}\arrow{w,t}{\start_3}\arrow{w,b}{\finish_3} \node{\cdots}\arrow{w,t}{\start_4}\arrow{w,b}{\finish_4}
		\end{diagram}
	\end{equation}
	We allow a time series data to be finite length. That is, we may have $X_k = \emptyset$ for all $k$ greater than some integer $N$. Such time series data will be called \emph{finite}. For convenience, we define $\start_0 = \finish_0: X_0 \to X_{-1} =\{\varepsilon\}$ where $\varepsilon$ denotes the null sequence. A time series object is thus a sequence of data-chunks of increasing length. The portion of the $(n+1)$-length chunks obtained by deleting the beginning and ending symbols, must be within the $n$-length chunks.
	
	To define the time series data as objects of a category we need to define morphisms. 
	
	\begin{definition} [Category of time series data]
		Let $\Context$ be as in Assumption \ref{A:concrete}. Then time series data objects assemble to form a category $\TSD_{\Context}$ in which a morphism between two time series data $X$ and $Y$ is an integer $k$ along with a sequence of maps $\psi_n : X_{n+k} \to Y_{n}$ for every $n\in\num$ such that the following commutations hold : 
		\begin{equation} \label{eqn:TSD_commut}
			\begin{tikzcd}
				\akashi{ X_{n+1+k} } \arrow{rr}{ \psi_{n+1} } \arrow[Akashi]{d}{ \finish_{n+k+1} }[swap]{ \start_{n+k+1} } && \Holud{ Y_{n+1} } \arrow[Holud]{d}{ \finish_{n+1} }[swap]{ \start_{n+1} } \\ 
				\akashi{ X_{n+k} } \arrow{rr}[swap]{ \psi_{n} } && \Holud{ Y_{n} }
			\end{tikzcd} , \quad
			\forall n\in\num .
		\end{equation}
		The composition is defined in an obvious manner.
	\end{definition}
	
	As usual, if the context is clear, then $\TSD_{\Context}$ will denoted more briefly as $\TSD$. The vertical columns in \eqref{eqn:TSD_commut} represent two separate objects in $\TSD$, while the horizontal arrows represent the $n$ and $n+1$-th component of a morphism between them. This integer $k$ will be called the \emph{jump} of the morphism. The verification of compositionality is routine and is provided in Section \ref{sec:app:TSD} in the appendix.
	
	The category $\TSD_{\SetCat}$ so defined turns out to be a special case of a functor category used to capture the notion of increasing information content. Following \cite[Sec 7]{fritz2020Stoch}, we may interpret the sequence in \eqref{eqn:def:TSD} to be a sequence of probability spaces, and the pair of morphisms $\paran{\start_n, \finish_n}$ create a Markov transition function. Each point $w$ in $X_n$ is sent with equal probability to $\start_n(w)$ or $\finish_n(w)$. The sequence represents an increasing information content, and the Markov transition function represents a lossy projection map. In this way, $\TSD_\SetCat$ becomes a sub-category of the category of functors from $\SqBrack{ \num_0, \downarrow }$ into $\StochCat$, the category of stochastic processes \cite[e.g.]{MossPerrone2022ergdc}. Such sequences are the object of study in Lauritzen’s notion of transitivity of a sequence of statistics \cite[Def 2.1]{lauritzen2012extremal}.Thus each object in $\TSD_\SetCat$ is a sequence of increasing information content, and the transition from the $n+1$-th piece to the $n$-th piece is a Markov transition in which either the first or last bit of a sequence is lost.
	
	We next make some observations about the generation of morphisms in $\TSD$.
	
	\begin{lemma} \label{lem:factor}
		Let $X$ and $Y$ be time series data, and \eqref{eqn:TSD_commut} holds for some $n\in \num_0$. Then $\psi_{n+1}$ is completely determined by $\psi_n$.
	\end{lemma}
	
	Lemma \ref{lem:factor} is proved in Section \ref{sec:app:TSD} of the appendix.
	
	\begin{example}
		If $\Context=\SetCat$ then Lemma \ref{lem:factor} can be expressed by the identity
		\[ \phi_{n+1}(\xi_0 \xi_1 \cdots \xi_{n+k+1}) = \phi_{n}(\xi_0 \xi_1 \cdots \xi_{n+k})_0 \phi_{n}(\xi_1 \cdots \xi_{n+k+1}), \quad \forall \xi_0 \xi_1 \cdots \xi_{n+k+1} \in X_{n+k+1}. \]
	\end{example}
	
	According to Lemma \ref{lem:factor} the morphisms in $\TSD$ of jump 0 are in one-to-one correspondence with the set maps which generate them. For this reason, we shall denote each $0$-morphism in $\TSD$ by $X \xrightarrow[]{k} Y$ if its generating map is $k: |X| \to |Y|$. This completes the description of the rightmost node of the diagram in \eqref{eqn:outline:1}. 
	
	\paragraph{Time series from subshifts} We have seen how objects in $\Seq$ arise naturally from dynamical systems with measurements. We shall now see how $\TSD$-objects are naturally associated with $\Seq$-objects. Consider the commutative diagram in $\Context$ shown below 
	\[\begin{tikzcd}
		S \arrow[rrr, "\sigma|S"] \arrow[d, hook] &&& S \arrow[d, hook] \\
		\mathbb{A}^{\num} \arrow[d, "\pi_{\num\to n+1}"'] \arrow[bend right=20]{rrr}[pos=0.2, swap]{\sigma} &&& \mathbb{A}^{\num} \arrow[d, "\pi_{\num\to n}"]\\
		\mathbb{A}^{n+1} \arrow{rrr}{\start_{n+1}}[swap]{\finish_{n+1}} &&& \mathbb{A}^{n}
	\end{tikzcd} \]
	This leads to
	\begin{equation} \label{eqn:did93k}
		\begin{tikzcd}
			S \arrow[dr, two heads] \arrow[rrr, "\sigma|S"] \arrow[d, hook] &&& S \arrow[d, hook] \arrow[dl, two heads] \\
			\mathbb{A}^{\num} \arrow[d, "\pi_{\num\to n+1}"'] \arrow[,bend right=20]{rrr}[pos=0.2, swap]{\sigma} & \word_n(S) \arrow[dl, bend right=20, hook] \arrow[Shobuj]{r}{\start_{n}}[swap]{\finish_{n}} & \word_{n-1}(S) \arrow[dr, bend left=20, hook] & \mathbb{A}^{\num} \arrow[d, "\pi_{\num\to n}"]\\
			\mathbb{A}^{n+1} \arrow{rrr}{\start_{n+1}}[swap]{\finish_{n+1}} &&& \mathbb{A}^{n}
		\end{tikzcd}
	\end{equation}
	Lemma \ref{lem:U_set:2} can be applied to this diagram to get the diagram on the right. The $n$-th word associated to a subshift is thus the image of its projection to the first $n+1$ components. This can be formalized as
	
	\begin{theorem} \label{thm:TSD_from_Seq}
		There is a functor $\word: \Seq_{\Context,<\infty} \to \TSD$ set up by $\word(S)_n := \word_n(S)$ as described in \eqref{eqn:did93k}.
	\end{theorem}
	
	We have already established the action of this functor on objects. The verification of compositionality is done in Section \ref{sec:app:TSD}.
	
	\begin{example}
		When $\Context=\SetCat$ then 
		\[ \word_n : \Seq \to \SetCat, \quad S \mapsto \braces{ \mbox{ all } n-\mbox{strings in } S } . \]
	\end{example}
	
	We have thus described another node and segment of the outline diagram \eqref{eqn:outline:1}. The $\word$-functor \eqref{eqn:did93k}, discretization functor $\Discretize$ \eqref{eqn:def:Dscrt:2}, and observation functor $\Obs$ (Theorem \ref{thm:Obs_functor}) together compose into the \emph{data} functor 
	\begin{equation} \label{eqn:def:data_fnctr}
		\begin{tikzcd}
			\DSMO \arrow[dashed, bend right=10]{drrrr}[swap, pos=0.33]{ \Data_{\Delta t} } \arrow[rr, "\Discretize_{\Delta t}", Shobuj] && \DSMO_{\num, \text{Set}} \arrow{rr}{ \Obs } && \Seq \arrow{d}{ \word } \\
			&& && \TSD
		\end{tikzcd}
	\end{equation}
	The diagram \eqref{eqn:def:data_fnctr} breaks down the process by which an observed dynamical system generates data. The data is in the form of time series objects, which are obtained as words or strings from shift spaces created from the dynamics. Observed dynamics are arranged by commuting relations as shown in \eqref{eqn:def:DSM_morphism}. These relations are transformed into semi-conjugacy relations in $\Seq$, which in turn are converted into the commutations of the form \eqref{eqn:TSD_commut}. As usual $\Delta t$ may be dropped from the subscript if it is clear from context. 
	
	We now present several instances of $\DSMO$ objects to illustrate our results. In all these examples, time $\Time$ is $\num_0$ by default. 
	
	\begin{example} [Circle rotation 1] \label{ex:S1:1}
		Suppose $\Omega = S^1$ and the dynamics is
		\[\Phi : S^1 \to S^1, \quad \theta \mapsto \theta + \rho \bmod 2\pi\]
		Suppose that the system is being observed along a single trajectory starting from $A = \braces{(0,0)}$. Next suppose that the observation map is
		\[\phi : S^1 \to \real, \quad \theta \mapsto \cos(\theta). \]
		Then the resulting time series data $X$ will be such that each $X_N$ comprises of sequences of length $N+1$ of the form $\braces{ \cos( (n_0 + n) \rho \Delta t }_{n=0}^{N}$ for various $n_0$. Any trigonometric reconstruction scheme would produce an exact reconstruction.
	\end{example}
	
	\begin{example} [Circle rotation 2] \label{ex:S1:2}
		Consider the same setup as Example \ref{ex:S1:1} but with $A$ instead being the entirety of $S^1$. Then the resulting time series data $X$ will be such that each $X_N$ comprises of sequences of length $N+1$ which are linear combinations of trigonometric functions whose frequencies are multiples of $\rho \Delta t$. Any trigonometric reconstruction scheme would produce an exact reconstruction.
	\end{example}
	
	\begin{example} [Quasiperiodic torus rotation 1] \label{ex:T2:1}
		Suppose $\Omega = \mathbb{T}^2$, the 2-dimensional torus, and the dynamics is
		\[ \Phi : \mathbb{T}^2 \to \mathbb{T}^2, \quad \theta \mapsto \theta + \rho \bmod 2\pi . \]
		Here $\theta$ is a 2-dimensional angular coordinate, and $\rho$ is a 2-dimensional rotation vector. If the coordinates of $\rho$ are rationally independent, then the motion is known as \emph{quasiperiodic}. Suppose that the set of initial conditions $A$ is the entirety of $\mathbb{T}^2$ and $\phi$ is an embedding of  then the resulting time series data $X$ will be such that each $X_N$ comprises of sequences of length $N+1$ which are linear combinations of trigonometric functions whose frequencies are of the form $\SetDef{ a_1 \rho_1 \Delta t + a_2 \rho_2 \Delta t }{ a_1, a_2\in \integer }$. Again, any trigonometric reconstruction scheme would produce an exact reconstruction.
	\end{example}
	
	\begin{example} [Quasiperiodic torus rotation 2] \label{ex:T2:2}
		Consider the variant of Example \ref{ex:T2:1} above in which the rotation vector $\rho$ is such that there is an $a\in \integer^2$ such that $a\cdot \rho \in \integer$. Then there is another $b\in \integer^2$ such that the resulting time series data $X$ will be such that each $X_N$ comprises of sequences of length $N+1$ which are linear combinations of trigonometric functions whose frequencies multiples of $(b\cdot \rho) \Delta t$. Thus the time series will only bear the signature of a single generating frequency. No reconstruction scheme would be able to produce an exact reconstruction. 
	\end{example}
	
	\begin{example} [Quasiperiodic torus rotation 3] \label{ex:T2:3}
		Consider the variant of Example \ref{ex:T2:1} above in which the observation map is
		\[ \phi : \mathbb{T}^2 \to \real, \quad \theta \mapsto \cos (b\cdot \theta) , \]
		for some $b\in \integer^2$. Then the resulting time series data $X$ will be such that each $X_N$ comprises of sequences of length $N+1$ which are linear combinations of trigonometric functions whose frequencies multiples of $(b\cdot \rho) \Delta t$. Again, no reconstruction scheme would be able to produce an exact reconstruction. 
	\end{example}
	
	An important point to note is that the $\DSMO$ objects from Examples \ref{ex:S1:1}, \ref{ex:T2:2} and \ref{ex:T2:3} produce the same time series if the parameters are chosen correctly. This is one of the simplest examples of the challenge or impossibility of reconstructing a $\DS$-object from a $\TSD$-object which is sourced not from $\DS$ but from $\DSMO$.
	
	We have thus established the entire path in Diagram \eqref{eqn:outline:1} that leads to a generation of data from a dynamical system being measured along orbits. This sets the stage to consider the task of reconstruction. 
	\section{Reconstruction} \label{sec:recon} 
	
	\begin{table} [!t]
		\caption{Tertiary set of components of the framework. The table summarizes the various categories and functors derived from the categories and functors in Table \ref{tab:param2}, and ultimately depending on \ref{tab:param1}. These categories and functors are the key to make conclusions from the arrangement in diagram \eqref{eqn:outline:1}. }
		\begin{tabularx}{\linewidth}{|l|L|}        \hline
			Variable & Interpretation \\ \hline  
			$\DSM$ & Category of dynamical systems along with measurement maps into objects from the observable category $\Context_{obs}$ \\ \hline
			$\DSMO$ & Category of dynamical systems along with measurement maps and an orbit \\ \hline 
			$\Data$ & functor $\DSMO \to \TSD$ resulting from the composition of $\word$ with $\Obs$. Represents how an observation system leads to the creation of a time series object. \\ \hline  
			$\bar{\Recon}$ & functor $\DSMO \to \DS$ resulting from the composition of $\Recon$ with $\Data$. Represents how an observation system leads to a reconstructed dynamical system, depending on the choice of reconstruction scheme $\Recon$. \\ \hline  
		\end{tabularx}
		\label{tab:param3}
	\end{table}
	
	Our study has two parts. The first is about the separate notions of dynamical systems, observations and orbits, and how they combine to provide data. The notion of "data" is formalized as objects of a category called time series-data $\TSD$. One of our main conclusions has been that all these arrangements including $\TSD$ are related functorially. Data-driven reconstruction algorithms operate on time series to create dynamical systems, in an effort to estimate the source of the data. The second part of our investigation explores the consequences of these algorithms being functorial. Before delving into the categorical math we must clarify the assumption on functoriality.
	
	Any algorithm is ultimately a function $\calA$, that accepts instances of certain data-types and provides an output of some other data-type. The output also depends upon some internal parameters. Some of the parameters represent design choices, while some represent unavoidable limitations of computer based calculations. For example all learning algorithms have a choice of hypothesis space $\calH$ which is preferably a dense subset of the space in which the target function lies. During any computation one can only employ a finite subspace of $\calH$. This subspace may be parameterized in different ways, such as the length and breadth of a neural network, the number of basis or frame elements in a regression technique. In either case, there is a number $L$ which represents the size of the actual search space. There are also other parameters involved, such as a \emph{machine precision} limit $\epsilon_{mach}$ which represents the resolution of computer arithmetic. All these intrinsic parameters related to performance can be bundled into a variable $\lambda$ drawn from a space $\Lambda$. Based on the notation so far, an algorithm can be expressed as a map :
	\[\calA : \FinSeq \times \Lambda \to \DS .\]
	We shall assume that the parameter space $\Lambda$ has a point at infinity $\bar{\lambda}$, representing the limiting or terminal value of the parameter variable. The function 
	\[ \bar{\calA} := \calA( \cdot, \bar{\lambda} ) : \FinSeq \to \DS \]
	represents the idealized performance of the algorithm beyond all of its performance limiting factors. For example in a least-squares fit approach, given a dataset $X_N$ in $\real^d$, the following steps are are followed :
	\begin{enumerate} [(i)]
		\item Construct an augmented space $\real^D$ using techniques such as delay-coordinates or reservoir embedding. Without loss of generality, the measurement function can be interpreted directly as an injective map $\phi : \Omega \to \real^D$.
		\item The dataset $X_N$ is some collection of sequences of length $N+1$ of orbits in $\real^D$.
		\item Take a gradation $\calH_1 \subset \calH_2 \subset \ldots$ of the hypothesis space into a sequence of finite dimensional subspace.
		\item Fix a search size $L$.
		\item Find a minimum norm function $w_N : \real^D \to \real^D$ from $\calH_L$ such that the graph of $w_N$ has the least-squares fit with the length $N+1$ time series contained in $X_N$. As a result the output $\calA(X_N, L)$ will be $\paran{\real^D, w_N}$.
		\item The output $\bar{\calA}(X_N) = \paran{\real^D, \bar{w}_N}$ is such that the minimum norm function $\bar{w}_N : \real^D \to \real^D$ in $L^2(\real^D)$ such that the graph of $\Phi$ has the least-squares fit with the length $N+1$ time series contained in $X_N$.
	\end{enumerate}
	
	This is where the assumption of functoriality plays a natural role. Datasets can undergo various transformation such as scaling, rotations, and translations. Categorically, two datasets $X_N, X'_N$ of length $N$ may be related by a morphism in $\FinSeq$, $h : X_N \to X'_N$. We want the idealized algorithm $\bar{\calA}$ to be compatible with such data transformations. This means that if the reconstructions are respectively 
	\[ (\Omega, \Phi^t) = \bar{\calA}(X_N) , (\Omega', \Phi'^t) = \bar{\calA}(X'_N) ,\]
	then the two dynamics should be related by a semi-conjugacy too. An algorithm $\calA$ usually places a regularity bound such as $C^1$ norm or total variation norm on the candidates in its hypothesis space. Thus it might fail to preserve semi-conjugacy under some distortions of the data. On the other hand $\bar{\calA}$ does not place regularity bounds and does not lose information content under data transformations. 
	
	\begin{figure} [!t]
		\centering
		\begin{tikzpicture} \node[draw, inner sep=5pt, draw=ChhaiD, line width=2pt] (box){
				\begin{tikzcd} [scale cd = 0.6, column sep = tiny, row sep = tiny]
					& & & \blue{ \begin{array}{c} \mbox{Limiting}\\ \mbox{parameter } \\ \bar{ \lambda } \end{array} } \arrow[d, blue] & & & & \\
					& \itranga{ \begin{array}{c} \mbox{Parameter : } \\ \lambda \in \Lambda \end{array} } \arrow[d, Itranga] & & \blue{ \bar{\mathcal{A}} ( X_N ) } \arrow[rd, dashed, blue] & \begin{array}{c} \mbox{Input : } \\ X_N \end{array} \arrow[l] \arrow[rd] & & & \\
					\begin{array}{c} \mbox{Input : } \\ X_N \end{array} \arrow[r] \arrow[rd] & \itranga{ \mathcal{A} \left( X_N, \lambda \right)} \arrow[rru, "\lambda\to \bar{\lambda}", dotted] & & & \blue{ \bar{\mathcal{A}} ( X_{N+1} ) } \arrow[rd, dashed, blue] & \begin{array}{c} \mbox{Input : } \\ X_{N+1} \end{array} \arrow[l] \arrow[rd] & & \\
					& \begin{array}{c} \mbox{Input : } \\ X_{N+1} \end{array} \arrow[r] \arrow[rd] & \itranga{ \mathcal{A} \left( X_{N+1}, \lambda \right)} \arrow[rru, "\lambda\to \bar{\lambda}", dotted] & & & \cdots \arrow[rd, dashed, blue] & \cdots \arrow[rd] & \\
					& & \cdots \arrow[rd] & \itranga{\cdots} & & & \blue{ \bar{\mathcal{A}} ( X ) } & \begin{array}{c} \mbox{TSD} \\ X \end{array} \arrow[l] \\
					& & & \begin{array}{c} \mbox{TSD} \\ X \end{array} \arrow[rrru, dotted] & & & \Shobuj{ \Recon(X) } \arrow[u, "\cong", Shobuj] & & 
				\end{tikzcd} 
			}; \end{tikzpicture}
		\caption{Algorithms and their limiting behavior. A $\TSD$-object \eqref{eqn:def:TSD} $X$ is represented as a sequence of sets of $N$-sequences $X_N$. Each such dataset $X_N$ is mapped by the algorithm $\calA$ into an estimate of the dynamical system. This mapping depends on the state of the design and performance parameters $\lambda$. Their limiting behavior is indicated in blue by $\bar{A}$. Although $\calA$ is not functorial, $\bar{\mathcal{A}}$ is. In topological reconstructions, by pointwise convergence, the limit $\bar{A}(X)$ of the estimates $\bar{A}(X_N)$ have a functorial dependence on $X$. The reconstruction functor $\Recon(X)$ is precisely this limit. It does not represent the output of an algorithm but its idealized operation. }
		\label{fig:algo_limit}
	\end{figure}
	
	We next explain why the successive approximations $\braces{ \bar{\calA(X_N)} }_{N=1,2,\ldots}$ converge to a limit which we denote as $\bar{A}(X)$. This limit will be interpreted as the final, ideal reconstruction $\Recon(X)$. Recall from the definition of the $\TSD$ category that each $X_N$ is a collection of strings of an alphabet $\mathbb{A}$, all the strings being of length $N+1$. There is also a pair of deletion maps $\start_N, \finish_N : X_N \to X_{N-1}$. The implication is that $X_{N-1}$ contains every $N$-length substring that can be obtained by deleting the starting or ending symbol in an $N+1$-length string in $X_N$. The converse may not be true, i.e., there may be strings in $X_{N-1}$ not in the image of $\start_N, \finish_N$. Let $\IntHom{\num_0}{\mathbb{A}}$ denote the set of all possible infinite sequences from the alphabet $\mathbb{A}$. This is also called the \emph{path-space} of $\mathbb{A}$. One can associate a subset $\calS_N$ of $\IntHom{\num_0}{\mathbb{A}}$ corresponding to $X_N$ :
	\[ \calS_N := \SetDef{ \phi : \num_0 \to \mathbb{A} }{ \phi(0) \phi(1) \cdots \phi(N) \in X_N } . \]
	Our previous observation implies that these sets are decreasing : $\calS_1 \supseteq \calS_2 \supseteq \cdots$. Moreover, if $\mathbb{A}$ is a compact set, then these form a sequence of nested compact sets, whose intersection $\calS$ is therefore non-empty and compact. Thus any $\TSD$-object $X$ unambiguously determines a compact subset $\calS$ of the path-space  of $\mathbb{A}$.
	
	Therefore as $N$ is increased, the algorithm $\calA$ is fed increasing amount of information about this subset $\calS$. If $\calA$ is a consistent reconstruction method, then its approximations should converge to a final estimate which replicates the sequences in $\calS$. This final, limiting estimate is precisely what is indicated by $\bar{\calA}(X)$ or $\Recon(X)$. The correspondence $X \mapsto \Recon(X)$ is also functorial. For simplicity we assume that at each stage the domain of the reconstruction remains the augmented space $\real^{D}$. Suppose that there is a jump $k\in \num$ and a sequence of maps $h_N : X_N \to X_{N+k}$ which commutes with the start and finish morphisms. Our assumption on data-compatibility implies the following commutation for every $N$ :
	\[\begin{tikzcd}
		\real^{D} \arrow[d, "h_N"'] \arrow[r, "\Phi_N^t"] & \real^{D} \arrow[d, "h_N"] \\
		\real^{D'} \arrow[r, "\Phi_N'^t"] & \real^{D'} 
	\end{tikzcd} , \quad \forall t\in\Time. \]
	Any dynamical system is equivalent to a conjugate form involving an invertible change of variables. Thus without loss of generality we may assume that the two dynamics have two reference points $x_0$ and $x'_0$, and $h_N(x_0) = x'_0$ at every stage $N$. Then consistency of the algorithm implies that the $h_N$ converges pointwise on the closure of the orbit of $x_0$. Our observations have been summarized in Figure \ref{fig:algo_limit}. Thus we define 
	
	\begin{definition} [Reconstruction scheme]
		A data-driven reconstruction scheme for dynamical systems is a functor $\Recon : \TSD \to \DS$.
	\end{definition} 
	
	If a semi-conjugacy relation is invertible, it is called a \emph{conjugacy}. A dynamical system is only unique up to conjugacy. Conjugacies correspond to isomorphisms in the category $\DS$.  By Theorem \ref{thm:Obs_functor}, we have a functor $\word \circ \Obs$ (by an abuse of notation) from $\DSM$ to $\TSD$ constructed by assigning the set of all data generated by a specific measurement and a dynamical system. Therefore, two conjugate systems will yield datasets isomorphic in $\TSD$.  Like many constructs in category theory, the output of any reconstruction scheme should be considered along with all its conjugate forms. This definition of reconstruction completes the last link in the diagram \eqref{eqn:outline:1}. We have thus arrived at the following situation :
	\begin{equation} \label{eqn:what_recon}
		\begin{tikzcd}
			\DSM \arrow{d}{} && \DSMO \arrow[dashed, Shobuj]{dll}[swap]{\proj} \arrow[dotted, d, Shobuj] \arrow{ll}[swap]{} \arrow{rr}{ \Discretize } \arrow[dashed]{drr}{ \Data } && \DSMO_{\num, \text{Set}} \arrow{rr}{ \Obs } && \Seq \arrow{dll}{\word} \\
			\DS && \DS && \TSD \arrow[dotted, Itranga, "?", ll]
		\end{tikzcd}
	\end{equation}
	The left loop extracts the dynamical system from a combined observation + dynamics + orbit object. It can be interpreted as the "true" dynamical system. A data-consistent reconstruction algorithm is a functor shown as the dotted red arrow above, such that the two composite functors shown in dashed green arrows are comparable. The question that now arises is how can two paths be reconciled.
	
	\section{Consistency} \label{sec:consistency}
	
	The reconstruction $\Recon$ only takes measured data as inputs, since the phase space cannot be observed directly. However, if the data is generated from a full observation of the phase space $\Omega$, then one should ideally be able to fully reconstruct the system. For this purpose we define a subcategory $\DS_{obs}$ of $\DS$ in which the domains are images of the inclusion functor $\iota$ from Assumption \ref{A:obs}. There is an inclusion $\iota'_{obs} : \DS_{obs} \to \DSM$ , given by
	\[\begin{tikzcd} \iota\Omega \arrow["\Phi^t"', loop, distance=2em, in=125, out=55] \end{tikzcd}
	\quad\begin{tikzcd} {} \arrow[rr, mapsto] && {} \end{tikzcd} \quad
	\begin{tikzcd} \iota\Omega \arrow["\Phi^t"', loop, distance=2em, in=125, out=55] \arrow[rr, "\Id_{\iota\Omega}"] && \iota\Omega \end{tikzcd}
	\]
	There are also analogous subcategories $\DSO_{obs}$ of $\DSO$ and  $\DSMO_{obs}$ of $\DSMO$ in which the dynamics is on an observable domain. Thus each of the three categories $\DS$, $\DSM$, $\DSMO$ have their observable analogs, and these six categories are organized as follows :
	\begin{equation} \label{eqn:outline:2}
		\begin{tikzcd} [scale cd = 1, row sep = large]
			&& \DS_{obs} \arrow[lld, "="'] \arrow[d, "\iota''_{obs}"'] \arrow[rrd, "\iota'_{obs}"] \arrow[rrrrd, "\iota_{obs}", bend left=10] \arrow[rrrrrd, "\iota'''_{obs}", bend left=20] && && &\\
			\DS_{obs} \arrow[rr, "\iota''_{obs}"] && \DS && \DSM_{obs} \arrow[d, "\subset"] && \DSMO \arrow[dll, "\proj"] & \DSM \arrow[l, "\kappa"] \arrow[llld, pos=0.2, "=", bend left=10] \\
			&& && \DSM \arrow[ull, "\proj"'] \arrow[ullll, "\proj_{obs}", bend left = 10]
		\end{tikzcd}
	\end{equation}
	Note how the various inclusion functors $\iota_{obs}$, $\iota'_{obs}$, $\iota''_{obs}$, and projections $\proj$ and $\proj_{obs}$ combine with the inclusion functor $\kappa$ from \eqref{eqn:def:DSM_to_DSMO}. Some of the functors are left and right inverses, as encoded in the commutation loops involving equalities. Of special significance is the arrow $\iota_{obs} = \kappa \circ \iota'_{obs}: \DS_{obs} \to \DSMO$. This functor maps each $(\iota\Omega,\Phi)$ in $\DS_{obs}$ to $\paran{\iota\Omega, \Phi, \Id_{\iota \Omega}, \Id_{\iota \Omega}}$. We are now ready to delve into the various notions related to reconstruction, centered around the Diagram \eqref{eqn:what_recon}.
	
	\begin{definition} [Consistency]
		A reconstruction $\Recon$ as in Diagram \eqref{eqn:what_recon}  will be called \emph{observation-consistent} if there is a natural transformation $\Forget \circ \iota_{obs} \Rightarrow \Recon \circ \Data \circ \iota_{obs}$, as shown below
		\[\begin{tikzcd}
			\DSMO \arrow[dd, "\Forget"'] \arrow[rrrr, "\Data"] &&&&\TSD \arrow[dd, "\Recon"] \\
			&& \DS_{obs} \arrow[ull, "\iota_{obs}"] \arrow[dashed]{dll}[name = n1, swap]{\iota''_{obs}} \arrow[dashed]{drr}[name = n2]{}  \\
			\DS &&&& \DS
			\arrow[shorten <=1pt, shorten >=1pt, Rightarrow, to path={(n1) to[out=-45,in=225] (n2)} ]{  }
		\end{tikzcd}\]
	\end{definition}
	
	Thus consistency implies that the construction functor applied to a data-object $X$ arising from a fully observed, observable dynamical system would produce a new data-object $X'$ which may not be isomorphic to $X$ but is derivable from $X$. One can define more generally :
	
	\begin{definition} [Parameterized Consistency]
		Suppose there is a functor $F: \calP \to \DSMO$, interpretable as a parameterized sub-family of dynamical systems being measured along orbits. A reconstruction $\Recon$ as in Diagram \eqref{eqn:what_recon}  will be called \emph{consistent} with respect to $F$, or simple $F$-consistent, if there is a natural transformation $\Forget \circ F \Rightarrow \Recon \circ \Data \circ F$, as shown below
		\[\begin{tikzcd}
			\DSMO \arrow[dd, "\Forget"'] \arrow[rrrr, "\Data"] &&&&\TSD \arrow[dd, "\Recon"] \\
			&& \calP \arrow[ull, "F"] \arrow[dashed]{dll}[name = n1, swap]{} \arrow[dashed]{drr}[name = n2]{}  \\
			\DS &&&& \DS
			\arrow[shorten <=1pt, shorten >=1pt, Rightarrow, to path={(n1) to[out=-45,in=225] (n2)} ]{  }
		\end{tikzcd}\]
	\end{definition}
	
	Note how observation-based consistency is a special case of parameterized consistency. Changing the direction of the natural transformation leads to :
	
	\begin{definition} [Simulator]
		A reconstruction $\Recon$ as in Diagram \eqref{eqn:what_recon}  will be called \emph{observation-simulator} if there is a natural transformation $\Recon \circ \Data \circ \iota_{obs} \Rightarrow \Forget \circ \iota_{obs}$, as shown below
		\[\begin{tikzcd}
			\DSMO \arrow[dd, "\Forget"'] \arrow[rrrr, "\Data"] &&&&\TSD \arrow[dd, "\Recon"] \\
			&& \DS_{obs} \arrow[ull, "\iota_{obs}"] \arrow[dashed]{dll}[name = n1, swap]{\iota''_{obs}} \arrow[dashed]{drr}[name = n2]{}  \\
			\DS &&&& \DS
			\arrow[shorten <=1pt, shorten >=1pt, Rightarrow, to path={(n2) to[out=225,in=-45] (n1)} ]{  }
		\end{tikzcd}\]
	\end{definition}
	
	Thus being a observation simulator means that if the construction functor is applied to a data-object $X$ arising from a fully observed, observable dynamical system, the new data-object $X'$ it produces may not be isomorphic to $X$ but can be transformed into $X$. One can define more generally :
	
	\begin{definition} [Parameterized simulator]
		Suppose there is a functor $F: \calP \to \DSMO$, interpretable as a parameterized sub-family of dynamical systems being measured along orbits. A reconstruction $\Recon$ as in Diagram \eqref{eqn:what_recon}  will be called \emph{consistent} with respect to $F$, or simple $F$-consistent, if there is a natural transformation $\Recon \circ \Data \circ F \Rightarrow \Forget \circ F$, as shown below
		\[\begin{tikzcd}
			\DSMO \arrow[dd, "\Forget"'] \arrow[rrrr, "\Data"] &&&&\TSD \arrow[dd, "\Recon"] \\
			&& \calP \arrow[ull, "F"] \arrow[dashed]{dll}[name = n1, swap]{} \arrow[dashed]{drr}[name = n2]{}  \\
			\DS &&&& \DS
			\arrow[shorten <=1pt, shorten >=1pt, Rightarrow, to path={(n2) to[out=225,in=-45] (n1)} ]{  }
		\end{tikzcd}\]
	\end{definition}
	
	One can ask for a stronger property of reconstruction algorithms, to yield exact reconstructions. This suggests the following natural property :
	
	\begin{definition} [Observation-exact reconstruction]
		A reconstruction $\Recon$ as in Diagram \eqref{eqn:what_recon}  will be called \emph{observation-exact} if the isomorphism on the left holds : 
		\begin{equation} \label{eqn:sls0}
			\begin{tikzcd} \Recon \circ \Data \circ \iota_{obs} \arrow[d, leftrightarrow, "
				\cong"'] \\ \Forget \circ \iota_{obs} \end{tikzcd}, \quad
			\begin{tikzcd}
				\DSMO \arrow[dd, "\Forget"'] \arrow[rrrr, "\Data"] &&&&\TSD \arrow[dd, "\Recon"] \\
				&& \DS_{obs} \arrow[ull, "\iota_{obs}", Shobuj] \arrow[d, dotted, Shobuj, "\iota''_{obs}"'] \\
				\DS && \DS \arrow[rr, "\cong"] \arrow[ll, "\cong"'] && \DS
			\end{tikzcd}
		\end{equation}
	\end{definition}
	
	Note that if two dynamical systems are related by an invertible change of variables, then their data will have the exact information content. Thus reconstruction and consistency can be exact only up to isomorphisms. The diagram on the right expresses this equivalence of functors via a commutation diagram. The two sides of the equality correspond to the two composite functors created by the clockwise and counter-clockwise paths. We define similarly :
	
	\begin{definition} [Generalized exact reconstruction] \label{def:F_exact}
		Given any functor $F : \calP \to \DSMO$, a reconstruction is said to be \emph{exact} w.r.t. $F$ if $\Recon \circ \Data \circ F \cong \Forget \circ F$. In other words there is a functor $G : \calP \to \DS$ creating the following commutative diagram :
		\begin{equation} \label{eqn:od0vo}
			\begin{tikzcd}
				\DSMO \arrow[dd, "\Forget"'] \arrow[rrrr, "\Data"] &&&&\TSD \arrow[dd, "\Recon"] \\
				&& \calP \arrow[ull, "F", Shobuj] \arrow[d, dotted, Shobuj, "G"'] \\
				\DS && \DS \arrow[rr, "\cong"] \arrow[ll, "\cong"'] && \DS
			\end{tikzcd}
		\end{equation}
	\end{definition}
	
	This completes a list of several properties of reconstructions based on the relations $\Recon$ might try to achieve. One can also characterize reconstructions by their methodology. They mainly fall under two types - \textit{inner} and \textit{outer} approximations.
	
	\begin{definition} [Observation based inner / outer approximation]
		A reconstruction $\Recon$ as in Diagram \eqref{eqn:what_recon}  will be called an \emph{observation-based} inner approximation if for every $\TSD$-object $X$, $\Recon(X)$ is the colimit of all those observable dynamical systems $(\Omega, \Phi)$ for which there a $\DSMO$ version $(\Omega, \Phi, \phi, A)$ whose data $Y := \Data \paran{\Omega, \Phi, \phi, A}$ transforms into $X$. Similarly, $\Recon$ will be called an observation-based outer approximation $\Recon(X)$ is the limit of all those observable dynamical systems $(\Omega, \Phi)$ for which there a $\DSMO$ version $(\Omega, \Phi, \phi, A)$ whose data $Y := \Data \paran{\Omega, \Phi, \phi, A}$ can be transformed from $X$.
	\end{definition}
	
	An inner approximation $\Recon(X)$ looks for the closest match to the data $X$ from the left, while an outer approximation looks for the closest match from the right. They are analogous to the notions of greatest lower bounds and least upper bounds. One can define more generally :
	
	\begin{definition} [Generalized inner-outer approximation]
		Suppose there is a functor $F: \calP \to \DSMO$, interpretable as a parameterized sub-family of dynamical systems being measured along orbits. A reconstruction $\Recon$ as in Diagram \eqref{eqn:what_recon}  will be called an inner approximation based on the parameterization $F$ if for every $\TSD$-object $X$, $\Recon(X)$ is constructed as follows : consider the collection $\text{Pre}(X)$ of all those parameter values $p\in \calP$ such that the data from $F(p)$ transforms into $X$. Then $\Recon(X)$ is built as the colimit of all dynamical systems arising from such $F(p)$ . One analogously defines outer approximations based on the parameterization $F$.
	\end{definition}
	
	The notions of exactness, consistency, simulation and inner / outer approximations are independent. Our first result in this section establishes how parameterization based inner or outer approximations also achieve a certain universality in terms of consistency or simulations :
	
	\begin{theorem} [Universality of inner / outer approximations] \label{thm_consistent:1}
		Let Assumptions \ref{A:concrete} and \ref{A:obs} hold, and $F:\calP \to \DSMO$ be a parameterized subcategory of $\DSMO$.
		\begin{enumerate} [(i)]
			\item Inner and outer approximations based on $F$ exist. 
			\item The inner approximation is universal in the following sense : Suppose $\Recon' : \TSD \to \DS$ is another reconstruction functor which is a simulator with respect to the parameterization, i.e., there is a natural transformation $\alpha : \Recon' \circ \Data \circ F \Rightarrow \Forget\circ F$. Then there is an even more fundamental natural transformation $\beta : \Recon' \Rightarrow \Recon$ such that $\alpha$ is related to $\beta$ as $\alpha = \beta \circ \paran{ \Data \circ F }$.
			\item The outer approximation is universal in the following sense : Suppose $\Recon' : \TSD \to \DS$ is another reconstruction functor which is consistent with respect to the parameterization, i.e., there is a natural transformation $\alpha : \Forget\circ F \Rightarrow \Recon' \circ \Data \circ F$. Then there is an even more fundamental natural transformation $\beta : \Recon \Rightarrow \Recon'$ such that $\alpha$ is related to $\beta$ as $\alpha = \paran{ \Data \circ F } \circ \beta$.
		\end{enumerate}
	\end{theorem}
	
	The notion of parameterized inner or outer approximations are right and left Kan extensions in the language of category theory :
	\begin{equation} \label{eqn:Recon:Kan}
		\begin{tikzcd} [column sep = large, scale cd = 0.7]
			\calP \arrow[r, "F"] & \DSMO \arrow[d, "\Data"'] \arrow[r, "\Forget"'] & \DS  \\
			& \TSD \arrow[r, "\Recon"'] & \DS
		\end{tikzcd} , \quad 
		\begin{tikzcd} [column sep = large, scale cd = 0.7]
			\DS_{obs} \arrow[r, "\iota_{obs}"] & \DSMO \arrow[d, "\Data"'] \arrow[r, "\Forget"'] & \DS  \\
			& \TSD \arrow[r, "\Recon"'] & \DS
		\end{tikzcd} 
	\end{equation}
	Theorem \ref{thm_consistent:1} is a direct restatement of the existence and universality of Kan extensions. See Section \ref{sec:app:Kan} for more details. Kan extensions are a consequence of purely diagrammatic calculus and have been used to interpret best-approximations in various other contexts. Some examples are topological data analysis \cite[e.g.]{BubenikScott2014prstnt, Curry2015tda, Shiebler2022kan}, general data-science \cite[e.g.]{SpivakWisnesky2020fast, Shiebler2022kan}, and the theory of programming \cite[e.g.]{Hinze2012kan, paterson2012constr, Yanofsky2013kol}. Theorem \ref{thm_consistent:1} thus presents a direct advantage of the categorical path taken in this article. Since observation based inner / outer approximations are a special case of general parametrized  inner / outer approximations we have the following immediate corollary to Theorem \ref{thm_consistent:1} by taking $\calP = \DSM$ and $F = \iota_{obs}$ as shown in \eqref{eqn:Recon:Kan} :
	
	\begin{corollary}
		Let Assumptions \ref{A:concrete} and \ref{A:obs} hold.
		\begin{enumerate} [(i)]
			\item Observation based inner and outer approximations based on $F$ exist. 
			\item The inner approximation is universal in the following sense : Suppose $\Recon' : \TSD \to \DS$ is another reconstruction functor which is a simulator with respect to the parameterization, i.e., there is a natural transformation $\alpha : \Recon' \circ \Data \circ F \Rightarrow \Forget\circ F$. Then there is an even more fundamental natural transformation $\beta : \Recon' \Rightarrow \Recon$ such that $\alpha$ is related to $\beta$ as $\alpha = \beta \circ \paran{ \Data \circ F }$.
			\item The outer approximation is universal in the following sense : Suppose $\Recon' : \TSD \to \DS$ is another reconstruction functor which is consistent with respect to the parameterization, i.e., there is a natural transformation $\alpha : \Forget\circ F \Rightarrow \Recon' \circ \Data \circ F$. Then there is an even more fundamental natural transformation $\beta : \Recon \Rightarrow \Recon'$ such that $\alpha$ is related to $\beta$ as $\alpha = \paran{ \Data \circ F } \circ \beta$.
		\end{enumerate}
	\end{corollary}
	
	Equation \eqref{eqn:sls0} is a special case of \eqref{eqn:od0vo}. The role of the generic functor $F$ in the upper-left corner of \eqref{eqn:od0vo} is played by $\iota_{obs}$ in \eqref{eqn:sls0}.  A reconstruction functor $\Recon$ is defined on datasets in general, and is oblivious of the underlying source. If $\Recon$ satisfies the functorial equality in \eqref{eqn:sls0}, then $\Recon$ achieves consistency on the restricted class of data-sources encoded by the functor $F$.
	
	We have been fixing a parameterization $F: \calP \to \DSMO$ and looking for reconstructions which attain optimality in various sense. One can examine Diagram \ref{eqn:what_recon} from the opposite direction. One can fix an arbitrary reconstruction $\Recon$ and then try to characterize the $\DSMO$ objects which are faithfully recreated by $\Recon$. One can make the following general statement :
	
	\begin{theorem} \label{thm_consistent:2}
		Given any reconstruction functor $\Recon : \TSD \to \DS$, the collection of dynamical systems with measured orbits which allow a reconstruction form a subcategory of $\DSMO$.
	\end{theorem}
	
	Theorem \ref{thm_consistent:2} assumes that the arrow labeled $\Recon$ in Diagram \eqref{eqn:od0vo} is given. The statement of the theorem is that given such a $\Recon$ the arrows $F$ and $G$ exist. The latter pair of arrows forms an \textit{equalizer} diagram, and the existence of an equalizer follows directly from the basic properties of $\CatCat$, the category of categories. An unspecified reconstruction functor $\Recon$ makes it impossible to determine the collection of objects and semi-conjugacies in $\DSMO$ which are successfully recreated by $\Recon$. In spite of this lack of specificity, Theorem \ref{thm_consistent:2} infers a categorical structure purely by categorical arguments.
	We next present our main result of this chapter : 
	
	\begin{theorem} \label{thm_consistent:3}
		Let Assumptions \ref{A:concrete} and \ref{A:obs} hold. Then an observation-based inner approximation is also observation-exact.
	\end{theorem}
	
	Theorem \ref{thm_consistent:3} shows that even if the reconstruction strategy is simply set to be the simplest dynamics generating the given dataset, then it achieves consistency if the dataset comes from a full observation. Diagram \eqref{eqn:outline:1} represents the most general situation of dynamics, data and reconstruction, and it is in general hard to equate reconstruction with the original source. Theorem \ref{thm_consistent:3} provides a sufficient and commonly occurring condition under which this is possible. Theorem \ref{thm_consistent:3} is proved in Section \ref{sec:app:recon}. The main idea behind the proof is that inner and outer extensions correspond to categorical constructs called \textit{right} and \textit{left Kan extensions}, explained in detail in Appendix \ref{sec:app:Kan}. The following corollary investigates the case when the data available is a full shift :
	
	\begin{corollary} \label{corr:full_shift}
		Let $\calC = \Topo$ or $\MeasCat$ and $\calC_{obs} = \mathrm{FinSet}$. If a time series data $X $ is given by $X_n = \mathbb{A}^n$, where $\mathbb{A}$ is a fixed set of alphabets, then we have
		\[
		\REnv{\Data}{\Forget} (X) = \paran{ \sigma, \mathbb{A}^{\num} }.
		\]
	\end{corollary}
	
	This corollary follows from an argument similar to \ref{thm_consistent:2}. The only difference is that we consider $(\sigma,\mathbb{A}^\num, \pi_0:\mathbb{A}^\num \to \mathbb{A})$ as the candidate object of reconstruction.
	
	\begin{example} \label{ex:Tent}
		Suppose $\calC = \MeasCat$, $\Omega = [0,1]$ and $\Phi:[0,1] \to [0,1]$ be defined by
		\[
		\Phi(x) = \begin{cases}
			2x & (0\leq x < \frac{1}{2})\\
			2x-1 & (\frac{1}{2}\leq x \leq 1).
		\end{cases}
		\]
		We define an observable $\phi$ by
		\[
		\phi(x) = \begin{cases}
			0 & (0\leq x < \frac{1}{2})\\
			1 & (\frac{1}{2}\leq x \leq 1).
		\end{cases}
		\]
		Obviously, $\Data (\Omega, \Phi, \phi)_n = \{0,1\}^n$ for every $n \in \num$. By Corollary \ref{corr:full_shift}, the result of the reconstruction based on this data is the binary full shift:
		\[
		\REnv{\Data}{\Forget} \paran{ \Data (\Omega, \Phi, \phi) } = \paran{\sigma, \{0,1\}^{\num}}.
		\]
		This reconstruction is exact.
	\end{example}
	
	Our category theoretic investigation of the entire process of dynamics to data generation, followed by a reconstruction back to a dynamical system allows us to make some statements of great generality such as in Theorems \ref{thm_consistent:1} and \ref{thm_consistent:2}. However, a categorical treatment tends to lose sight of implementation details and convergence properties of specific reconstruction strategies. Our approach makes no assumptions on the domain or dynamics, beyond the minimal and broad assumptions on the domain type in Assumption \ref{A:concrete}. Thus, the usual tools of dynamical systems analysis, such as assumptions on positive Lyapunov exponents, rotational nature, Smale horseshoe structures are unavailable. This lack of specific details is not a disadvantage but an essential feature of any mathematical characterization of data-driven techniques. The commonly used data-driven approaches, such as Koopman based \cite{FroylandEtAl2010coherent, DasGiannakis_delay_2019, DGJ_compactV_2018}, Ulam's method \cite[e.g.]{Froyland1998approx, Froyland1999ulam}, interpolation based \cite{HaraKokubu2024learn} and symbolic dynamics based \cite[e.g.]{AlpernPrasad1989coding, AlpernPrasad2005towers}, avoid any assumption on the dynamics and instead utilize collective properties such as dense subsets having desirable topological or measure-theoretic properties.
	
	\begin{definition} [Ceiling and floor structure] \label{def:ceil_floor}
		A category $\calP$ is said to be equipped with a \emph{floor structure} if there is a functor $\lambda : \Lambda \to \calP$ satisfying :
		\begin{enumerate} [(i)]
			\item There is a functor $\lambda^* : \calP \to \Lambda$ and a natural transformation $\Id_{\calP} \Rightarrow \lambda \circ \lambda^*$.
			\item Every left slice in $\Lambda$ is finite.
		\end{enumerate}
		A functor $\lambda : \Lambda \to \calP$ is said to have \emph{ceiling structure}  if condition (i) above reads $\lambda \circ \lambda^* \Rightarrow \Id_{\calP}$, and in (ii) above, right (instead of left) slices are finite.
	\end{definition}
	
	This definition is inspired by the familiar "ceiling" and "floor" operations in arithmetic. The are functorial and obey the above relations with the inclusion of the integers in the reals, although $\Nplus$ does not have finite right slices.
	\[\begin{tikzcd}
		\Nplus \arrow[d, hook, "\subset"'] &&& \Rplus \arrow[lll, "\text{floor}"'] \\
		\Rplus \arrow[urrr, bend right=10, Rightarrow, dashed] &&& \Nplus \arrow[u, hook, "\subset"'] \arrow[ulll, pos=0.7, "="']
	\end{tikzcd} , \quad 
	\begin{tikzcd} 
		\Rplus \arrow[drrr, bend right=10, Rightarrow, dashed] \arrow[rrr, "\text{ceil}"] &&& \Nplus \arrow[d, hook, "\subset"] \\
		\Nplus \arrow[u, hook, "\subset"] \arrow[urrr, pos=0.7, "="] &&& \Rplus
	\end{tikzcd}\]

	\begin{theorem} [Conditions for computability] \label{thm:comput}
		Consider a parameterization $F : \calP \to \DSMO$, and suppose that $\calP$ has a floor / ceiling structure. Then the $F$ based inner / outer approximations are computable.
	\end{theorem}
	
	Theorem \ref{thm:comput} is a direct consequence Lemma \ref{lem:Kan_id:2} presented later in Section \ref{sec:app:Kan}. For every integer $L\in\num$, let $\Lambda^{(L)}$ be the subcategory $\Lambda$ spanned by objects whose left-slice size is no larger than $L$. Then the category $\Lambda$ is the direct-limit or colimit of the following sequence of inclusions
	\[\begin{tikzcd} \Lambda^{(1)} \arrow[r, "\subseteq"] & \Lambda^{(2)} \arrow[r, "\subseteq"] & \Lambda^{(3)} \arrow[r, "\subseteq"] & \cdots \arrow[r] & \Lambda \end{tikzcd}\]
	One similarly obtains a gradation for right-finite categories. The following flowchart presents a method for computing an inner approximation by utilizing this gradation :
	
	\begin{tikzpicture}[scale=0.6, transform shape, framed, background rectangle/.style={double, ultra thick, draw=gray, rounded corners}]
		\node [style={rect2}] (n1) at (0, 0) {Pick an object $X$ from $\TSD$};
		\node [style={rect2}] (n2) at (0.8\columnA, 0) {Left slice $\Comma{\Data \circ F}{X}$ in $\calP$};
		\node [style={rect2}] (n3) at (1.6\columnA, 0) {Left slice $\Comma{\Data \circ F}{X}^{(L)}$ in $\calP^{(L)}$};
		\node [style={rect2}] (n4) at (2.4\columnA, 0) {Image of the left slice $\Comma{\Data \circ F}{}^{(L)}$ under $\Forget$};
		\node [style={rect2}] (n5) at (2.4\columnA, -1.5\rowA) {Compute the colimit $\paran{ \Omega_L, \Phi_L }$ of this finite diagram in $C$};
		\node [text width=1.0\columnA, text centered, minimum height=0.2\rowA, shape=rectangle, draw=ChhaiB] (n6) at (1.4\columnA, -1.5\rowA) {The diagram \[\Comma{\Data \circ F}{X}^{(L)} \to \DS \] is a colimit of the diagrams \[\Comma{fa}{x}^{(L)} \to \DS\] };
		\node [text width=1.0\columnA, text centered, minimum height=0.2\rowA, shape=rectangle, draw=ChhaiB] (n7) at (0.2\columnA, -1.5\rowA) {$\colim_{L\to \infty} \paran{ \Omega_L, \Phi_L }$ equals $\paran{ \Omega, \Phi }$ which is the inner approximation $\REnv{\Data F \lambda}{\Forget \lambda}(X)$ = $\REnv{\Data F}{\Forget}(X)$};
		\draw[-to] (n1) to (n2);
		\draw[-to] (n2) to (n3);
		\draw[-to] (n3) to (n4);
		\draw[-to] (n4) to (n5);
		\draw[-to] (n5) to (n6);
		\draw[-to] (n6) to (n7);
	\end{tikzpicture}
	
	The procedure for computing outer approximations is analogous. Two useful examples of the second parameterizing category $\Lambda$ are $\Nplus$ and $\Delta$, the simplex category. Theorem \ref{thm:comput} provides a framework for computability in dynamical systems. Any choice of this triple $\paran{\Lambda, \calP, \lambda}$ leads to a distinct computation paradigm. Note that if $\Lambda$ is a category with finite left slices and $\lambda : \Lambda \to \DSMO$ is an embedding, then the image subcategory $\calP := \ran \lambda$ is endowed with a floor structure. Thus an embedding of any category with finite left-slices lends a floor structure to at least one subcategory of $\DSMO$. Characterizing the maximal such subcategory is an open task. Even with the trivial choice $\calP := \ran \lambda$ there are some interesting computational paradigms
	
	\begin{example} [Quasiperiodic approximation]
		Suppose $\calS$ is a dense and countable collection of irrational numbers, with some ordering $s_1, s_2, \ldots$. For every $s\in \calS$, let $\text{Rot}_{s}$ denote the quasiperiodic rotation on $S^1$ with rotation number $s$. Consider the subcategory $\calP$ of $\DS$ spanned by all Cartesian products of finitely many such rotations. Each dynamical system in this collection is a quasiperiodic torus rotation. The natural preorder among the finite subsets of $\calS$ lead to a pre-ordered structure $\Lambda$ among $\calP$. Thus $\calP$ has a floor structure. An inner approximation with respect to $\calP$ leads to a quasiperiodic approximation of the data. This is the basis of all techniques such as DMD \cite[e.g.]{RowleyEtAl09} and spectral approximation of the Koopman operator \cite{DGJ_compactV_2018, ValvaGiannakis2023cnstnt}.
	\end{example}
	
	\begin{example} [Approximation by subshifts]
		For each $n\in\num$ let $S_n$ be the full shift on $n$ symbols. For each $m<n$, $S_m$ may be embedded within $S_n$ in $^n C_m$-many ways. Next note that for any pair of embeddings as shown below
		\[\begin{tikzcd}
			& S_{n_1} \\
			S_m \arrow[ur, "j_1"] \arrow[dr, "j_2"] \\
			& S_{n_2}
		\end{tikzcd} \imply 
		\begin{tikzcd}
			& S_{n_1} \arrow[dr, dotted] \\
			S_m \arrow[ur, "j_1"] \arrow[dr, "j_2"] & & j_1 \vee j_2 \\
			& S_{n_2} \arrow[ur, dotted]
		\end{tikzcd}\]
		one has a pushout diagram resulting in a new sub-shift which is the gluing of $S_{n_1}$ and $S_{n_2}$ along their embedded images of $S_m$. Let $\calP$ be the collection of all sub-shifts obtained by finitely many such gluings. Such a $\calP$ has a floor structure and the inner-approximation it provides attempts to model time series data as chaotic signals of various combinatorial complexity.
	\end{example}
	
	Finally, we comment on the impossibility of general and exact reconstruction scheme. In practice, only countable number of spaces are available as observable objects and therefore amenable to numerical procedures while there are uncountably many isomorphism classes of dynamical systems. Thus, as long as the reconstruction scheme only yields observable dynamics, there should be dynamical systems that cannot be reconstructed. More generally, we conjecture the following result that will exclude possibility of a universal data-driven reconstruction technique that provides an exact reconstruction.
	
	\begin{conjecture} \label{conj:oj9l}
		Suppose $\calC$ has uncountably many isomorphism classes of objects, while $\calC_{obs}$ has countably many. Then, for any reconstruction scheme $\Recon$, the subcategory of exact reconstructions (Theorem \ref{thm_consistent:2}) is a proper sub-category. In other words, there are DSMO objects not reconstructible by $\Recon$.
	\end{conjecture}
	
	Further search for inner / outer approximation paradigms is an interesting avenue for investigation.
	
	\paragraph{Conclusions} The major realizations we have reached are :
	\begin{enumerate} [(i)]
		\item The process of data generation from dynamical systems being measured along orbits is functorial.
		\item Time series data can be modeled as a chain of finite sequences of data, with decreasing information content.
		\item The limiting behavior of consistent reconstruction algorithms are functorial.
		\item Inner and outer approximations exist, additional conditions need to be imposed to ensure that they are computable or exact.
	\end{enumerate}
	
	Category theory makes up for the lack of attention to individual detail by offering a collective perspective. A classical example is homology, which is a record of all the possible ways chain complexes map into a topological space. The countable collection of  simplices and simplicial morphisms provide a computational as well as descriptive framework within the collection of uncountably infinitely many topological objects. Our analysis, along with Theorems \ref{thm_consistent:1} -- \ref{thm:comput} points towards an a similar possibility in the field of dynamical systems and data-driven reconstruction.
	\section*{Acknowledgments}
	The authors are grateful to the referees, whose comments and suggestions helped improve the presentation greatly. 
	
	\section{Appendix}
	\subsection{The notion of colimits} \label{sec:app:colim}
	
	Suppose $J$ and $\calX$ are two categories. Then for any object $x\in ob(\calX)$ there is a constant functor that sends every object of $J$ into $x$, and every morphism into $\Id_x$. This functor is denoted as $x:J\to \calX$. Given a functor $F:J\to \calX$, a \emph{co-cone} for $F$ is a natural transformation $\eta : F \Rightarrow x$ from $F$ to a constant functor $x$. The object $x$ is called the \emph{tip} of the co-cone. A \emph{colimit} of $F$ is a universal co-cone, that means it is a co-cone $\bar{\eta} : F \Rightarrow \bar{x}$ such that for any other co-cone $\eta : F \Rightarrow x$, there is a unique $\phi : \bar{x} \to x$ such that $\eta = \phi \circ \bar{\eta}$. The tip $\bar{x}$ of the universal co-cone is also called the colimit of the functor $F$. Colimits are closely related to their dual notion : \emph{limits}.
	
	Colimits and limits are a sweeping generalization of various fundamental constructs in mathematics, such as limits, convergence, min, max, union, intersection, inverse images, and gluing. The functor $J$ can interpreted as a pattern and $F$ as a diagram existing within $\calX$ in the shape of $J$. $J$ may be finite or infinite. For example a terminal object is thus the colimit of the identity functor $\Id_{\calX} : \calX \to \calX$. The colimit $\bar{x}$ provides a sense of completion or termination of the pattern. Colimits may not exist, but if they do, they are unique up to isomorphism.
	
	A category $\calC$ is called \emph{complete} (resp. cocomplete) if for any category $J$, any functor $F:J\to \calC$ has a limit (resp. colimit). In our case, if the context category $\Context$ is either $\Topo, \MeasCat, \SetCat$, it is co-complete.
	
	\begin{property} \label{Py:cocomplt}
		A category satisfying Assumption \ref{A:concrete} is complete and cocomplete.
	\end{property}  
	
	Property \ref{Py:cocomplt} will be crucial in our study of Kan extensions later.
	Thus topological, measure-theoretic as well as set-theoretic dynamical systems have colimits. One of the simplest notions of limits is products. Categorical products are a generalization of the usual Cartesian products of sets, topological spaces or metric spaces. In all these contexts, if one has an indexed family of objects $\SetDef{A_i}{i\in I}$, then one can form a product object $A:=\prod_{i\in I} A_i$ which lies in the same context as the $A_i$. We shall use the symbols $\prod$ and $\times$ interchangeably to denote a Categorical product. For each $i\in I$ there is a \emph{projection} morphism $\pi_i : A\to A_i$. In the familiar contexts mentioned above, the projections are the usual coordinate-wise projections. In a general category, they are defined by a \emph{universal} property \cite[see]{Maclane2013}. When the index set $I=\num$ and all the $A_i$-s are the same object $\mathbb{A}$, the product is denoted as $\mathbb{A}^{\num}$. In the usual set theoretic interpretation $\mathbb{A}^{\num}$ denotes the collection of all infinite sequences of elements of $\mathbb{A}$, indexed by $0,1,\ldots$. If $\mathbb{A}$ is endowed with a topology, then $\mathbb{A}^{\num}$ inherits the product topology.
	
	The following classical result guarantees the existence of (co)-limits under a simple condition :
	
	\begin{lemma} \label{lem:dod3p}
		If a category $\calC$ has all $J$ limits and / or colimits, then so does the functor category $\Functor{A}{\calC}$ for any category $A$.
	\end{lemma}
	
	\subsection{Adjunction}
	
	An adjunction between categories $\calC, \calD$ is a pair of functors $F : \calD\to \calC$, $G: \calC\to \calD$ such that there is natural isomorphism $\Phi$ between the following two Set-valued functors
	\begin{equation}\begin{split} \label{eqn:lp38}
			\Hom_{\calC}( F\cdot; \cdot) &= \Hom_{\calC} \circ \left( F \times \Id_{\calC} \right) : \calD^{op} \times \calC \to \SetCat . \\
			\Hom_{\calD}( \cdot; G\cdot) &= \Hom_{\cal\cal} \circ \left( \Id_{\calD} \times G \right) : \calD^{op} \times \calC \to \SetCat .
	\end{split}\end{equation}
	This means that for every $c\in ob(\calC), d\in ob(\calD)$, there is a bijection 
	\[ \Phi_{d,c} : \Hom_{\calC}( Fd; c) \xrightarrow{\cong} \Hom_{\calD}(d; Gc) , \]
	which is also natural in the following sense :
	\begin{equation} \label{eqn:oind9}
		\forall \; 
		\begin{tikzcd}[column sep = small]
			d & c \arrow{d}{\psi} \\
			d' \arrow{u}{\phi} & c'
		\end{tikzcd} , \quad
		\begin{tikzcd}[column sep = small]
			Fd & Gc \arrow{d}{G(\psi)} \\
			Fd' \arrow{u}{F(\phi)} & Gc'
		\end{tikzcd} , \quad 
		\begin{tikzcd} [column sep = large]
			\Hom_{\calC}( Fd; c) \arrow{d}{\cong}[swap]{ \Phi_{d,c} } \arrow{r}{ \psi\circ \cdot \circ F(\phi) } & \Hom_{\calC}( Fd'; c') \arrow{d}{ \Phi_{d', c'} }[swap]{\cong} \\
			\Hom_{\calD}(d; Gc) \arrow{r}{ G(\psi) \circ \cdot \circ \phi } & \Hom_{\calD}( d'; Gc' )
		\end{tikzcd}
	\end{equation}
	Equivalently, adjunction means : 
	\begin{equation} \label{eqn:nr83}
		\forall h \in \Hom_{\calC}( Fd; c), \quad \Phi_{d',c'} \left( \psi\circ h \circ F(\phi) \right) = G(\psi) \circ \Phi_{d,c}h \circ \phi.
	\end{equation}
	Note that in \eqref{eqn:lp38}, the relation between $F$ and $G$ is not symmetric. $F$ is called left adjoint to $G$, $G$ is called the right adjoint to $F$. This relation is depicted as $F \dashv G$.
	
	Adjunction is a generalization of the notion of inverse. Left or right inverses if they exist, are left and right adjoints. Thus a pair of invertible functors form an adjunction pair. An inversion is based on a one-to-one correspondence between objects in the domain and codomain categories. In the absence of such a correspondence, one can hope for the closest approximation to an inversion. The notion of "closeness" or "proximity" is borne by morphisms within a category. Morphisms are directional in nature. So an approximation can be from the left or from the right, leading to the notions of left and right adjoints respectively. Adjunction is the basis for the concept of Kan extensions.
	
	\subsection{Kan extensions} \label{sec:app:Kan}
	
	Kan extensions \cite[e.g.]{perrone2022kan, street2004categorical, Riehl_homotopy_2014} are universal constructions which generalize the notion of constrained optimization to the categorical context. Suppose we have the following arrangement of  categories and functors:
	\begin{equation} \label{eqn:dpp3k}
		\begin{tikzcd}
			D\\
			X \arrow{r}{F} \arrow{u}[swap]{K} & E 
		\end{tikzcd}
	\end{equation}
	When we regard $K$ to be an inclusion of $X$ into $D$, Kan extensions can be interpreted as the most general generalizations of the functor $F$ to the domain $D$. One can alternatively interpret $F$ as the constraint of the optimization problem where we seek the best functor $D \to E$ that satisfies it.
	
	Definitions are as follows. A \emph{left Kan extension} of $F$ along $K$ is a functor $\Lan_{K}{F} : D\to E$ along with a minimum natural transformation $\eta : F\Rightarrow \Lan_{K}{F} \circ K$. Moreover, this pair $\paran{ \Lan_{K}{F}, \eta }$  is minimum / universal in the sense that for every other functor $H:D\to E$ along with a natural transformation $\gamma : F \Rightarrow H\circ K$, there is a unique natural transformation $\tilde\gamma : \Lan_{K}{F} \Rightarrow H$ s.t. $\gamma = \paran{ \tilde\gamma \star \Id_{K} } \circ \eta$.  This is shown in the diagram below.
	\[ \mathcal{L} := \Lan_{F} (K), \quad 
	\begin{tikzcd}
		& & & E \\
		E & X \arrow[dashed, bend left = 10]{urr}[name=x1]{} \arrow[dashed]{drr}[name=x2]{} \arrow{l}[name=F]{F} \arrow{r}{K} & D \arrow{dr}[name=H]{H} \arrow{ur}[swap, name=L]{ \mathcal{L} } & \\
		& & & E 
		\arrow[shorten <=2pt, shorten >=3pt, Rightarrow, to path={(F) to[out=90,in=180] (x1)} ]{  }
		\arrow[shorten <=2pt, shorten >=3pt, Rightarrow, to path={(F) to[out=-90,in=225] (x2)} ]{  }
		\arrow[shorten <=1pt, shorten >=1pt, Rightarrow, to path={(L) to[out=-45,in=45] (H)} ]{  }
	\end{tikzcd}\]
	The \emph{right Kan extension} of $F$ along $K$ is a dual notion to the left Kan extension and the definition is similar except the directions of arrows. It is a functor $\Ran_{K}{F} : D\to E$ along with a natural transformation $\epsilon : \Ran_{K}{F} \circ K \Rightarrow F$, which is maximum among similar pairs of functor $H:D \to E$ and natural transformation $\gamma: H \circ K \Rightarrow F$ in the sense there exists a unique natural transformation $\tilde{\gamma} : H \Rightarrow \Ran_K F$ such that $\gamma = \epsilon \circ \paran{ \tilde\gamma \star \Id_{K} } $.
	If $E$ is a co-complete category, the left Kan extension always exists. Similarly if $E$ is a complete category, the right Kan extension always exists. In case both left and right Kan extensions of $F$ along  $K$ exist, they combine to produce the following diagram :
	\begin{equation} \label{eqn:L_R_Env}
		\begin{tikzcd} [column sep = large]
			E & E & E \\
			& X \arrow[dashed]{ul}[name=x2]{} \arrow[dashed]{ur}[name=x1]{} \arrow{d}{K} \arrow{u}[name=F]{F} \\ 
			& D \arrow[]{uur}[swap]{\REnv{K}{F}} \arrow[]{uul}{\LEnv{K}{F}} &
			\arrow[shorten <=2pt, shorten >=3pt, Rightarrow, to path={(F) -- (x1)},]{}
			\arrow[shorten <=2pt, shorten >=3pt, Rightarrow, to path={(x2) -- (F)},]{}
		\end{tikzcd}
	\end{equation}
	It is often possible to explicitly determine action of a Kan extension on objects, as shown below :
	
	\begin{lemma} \label{lem:Kan_pointwise}
		Consider the arrangement of \eqref{eqn:dpp3k}. Then 
		\begin{equation} \label{eqn:Kan_pointwise}
			\begin{split}
				\mbox{If } E \mbox{ is cocomplete, } \mbox{ then }& \Lan_{K}{F}(d) = \colim \SetDef{Fx}{ Kx \to d } \\ 
				\mbox{If } E \mbox{ is complete, } \mbox{ then } & \Ran_{K}{F}(d) = \lim \SetDef{Fx}{ d \to Kx } \\ 
			\end{split}
		\end{equation}
	\end{lemma}
	
	This construction is known as the \emph{pointwise} definition of Kan extensions. The colimit and limit in \eqref{eqn:Kan_pointwise} are along slices of the object $d$ along $K$. Note that each slice, left or right, can be interpreted as a constraint on the objects of $K$. The act of finding limits or colimits is analogous to finding the minimum or maximum under this constraint.   
	
	\begin{lemma} \label{lem:psi93}
		Suppose Property \ref{Py:cocomplt} hold. Then any arrangement of functors 
		\[\begin{tikzcd}
			\calY && \calX \arrow[ll, "F"'] \arrow[rr, "G"] && \DS
		\end{tikzcd}\]
		has a right and left Kan extensions of $G$ along $F$.
	\end{lemma}
	
	\begin{proof}
		Since by Property \ref{Py:cocomplt} $\Context$ is (co)-complete, the dynamics category $\DS = \Functor{\Time}{\Context}$ is (co)-complete too. Thus by Lemma \ref{lem:Kan_pointwise}, the Kan extensions exist, and have the pointwise definition expressed in \eqref{eqn:Kan_pointwise}.
	\end{proof} 
	
	\paragraph{Pre- and post- adjoints} Let $T : A \to B$ be a functor. Then $T^* : B \to A$ is said to be a \textit{pre-right adjoint} to $T$ if $T T^* \Rightarrow Id_B$. Similarly, $T^* : B \to A$ is said to be a \textit{post-right adjoint} \cite[e.g.]{Das2023CatEntropy, Das2025null} to $T$ if $Id_B \Rightarrow T T^*$. The next lemma presents some useful conditions under which Kan extensions are preserved under pre-compositions.
	
	\begin{lemma} \label{lem:Kan_id:2}
		\cite[Lem 7.5]{Das2025null} Consider the following arrangement of functors $a, b, f$ as shown below.
		\begin{equation} \label{eqn:Kan_id:2}
			\begin{tikzcd}
				A \arrow[r, "a"] & B \arrow[rr, "b"] \arrow[d, "f"'] && C \\
				& E \arrow[rr, "\text{L/Ran}_{f} (b)"'] && C
			\end{tikzcd}
		\end{equation}
		\begin{enumerate} [(i)]
			\item If $C$ is co-complete and $a$ has a post-right adjoint, then $\REnv{b}{f} = \REnv{ba}{fa}$.
			\item If $C$ is complete and $a$ has a pre-right adjoint, then $\LEnv{b}{f} = \LEnv{ba}{fa}$.
			\item In particular if $C$ is co-complete and $a$ has a right inverse then $\REnv{b}{f} = \REnv{ba)}{fa}$.
			\item In particular if $C$ is complete and $a$ has a right inverse then $\LEnv{b}{f} = \LEnv{ba}{fa}$.
		\end{enumerate}
	\end{lemma}
	
	\subsection{Topological concrete categories} \label{sec:app:topo_cncrt}
	
	Topological concrete categories provide a powerful framework for abstracting and unifying concepts across different areas of mathematics and its applications. They are categories whose objects are sets equipped with some "topological-like" structure, and whose morphisms are functions that preserve that structure. A category $\calC$ is said to be topological concrete over a category $\calS$ if there is a functor $U:\calC\to\calS$ with the following property : given any object $X\in \calS$, any indexed family $\SetDef{\Omega_i}{i\in I}$ of objects in $\calC$, and $\calS$-morphisms $f_i : X\to U(\Omega_i)$, then :
	\begin{enumerate} [(i)]
		\item there exists a $\calC$-object $\bar{\Omega}$ such that $U(\bar{\Omega}) = X$;
		\item there exist $\calC$-morphisms $g_i : \bar{\Omega} \to \Omega_i$ such that for each $i\in I$, $U(g_i) = f_i$.
	\end{enumerate}
	This arrangement $\SetDef{ g_i : \bar{\Omega} \to \Omega_i }{ i\in I }$  is universal in the sense that for any other arrangement $\SetDef{ g'_i : \bar{\Omega'} \to \Omega_i }{ i\in I }$ and a morphism $\phi : U( \bar{\Omega'} ) \to X$ satisfying $U(g'_i) = \phi \circ f_i$ for each $i\in I$, there is a unique morphism $\psi : \bar{\Omega'} \to \bar{\Omega}$ such that
	\begin{enumerate} [(i)]
		\item $U(\psi) = \phi$;
		\item $g'_i = \psi \circ g_i$ for each $i\in I$.
	\end{enumerate}
	The object $\bar{\Omega}$ is called the \textit{initial lift} of the diagram $\SetDef{g_i : \bar{\Omega} \to \Omega_i}{i\in I}$. It is the simplest or minimal object in $\Context$ which allows that diagram in $\calS$ to be realized. In our article, the role of $\calS$ is played by $\SetCat$.
	
	Many fundamental categories in mathematics, such as the category $\Topo$ of topological spaces or $\MeasCat$ of measurable spaces, are examples of topological concrete categories. This framework allows theorems proven in the context of topological concrete categories to be immediately applied to these diverse areas.
	
	A key property of topological concrete categories is that the forgetful functor $U$ allows certain fundamental constructions in topology, like forming products, subspaces (equalizers), and quotient spaces (coequalizers) (which are specific types of limits and colimits), can be defined by performing the set-theoretic operation and then equipping the resulting set with the appropriate initial or final topology/structure. A few of the important properties of topological concrete categories are :
	
	\begin{property} \label{Py:set_inv}
		If Assumption \ref{A:concrete} holds, then there is a functor $U^* : \SetCat \to \Context$ which is a left adjoint of $U$.
	\end{property}
	
	\begin{property} \label{Py:U_faithful}
		If Assumption \ref{A:concrete} holds, then the functor $U$ is faithful, i.e., injective on morphisms.
	\end{property}
	
	\begin{property} \label{Py:U_terminal}
		If Property \ref{Py:set_inv} holds, then the functor $U$ maps $1_{\Context}$ into a single element set in $\SetCat$.
	\end{property}
	
	\begin{lemma} \label{lem_topcat}
		Let $\mathcal{C}$ be a topological concrete category and $U: \mathcal{C} \to \mathrm{Set}$ be the associated forgetful functor. Then
		\begin{enumerate}
			\item $\mathcal{C}$ is complete and cocomplete.
			\item A morphism  $f:X \to Y$ in $\mathcal{C}$ is monic (epic) if $Uf:UX \to UY$ is injective (surjective).
			\item Let  $f:X \to Y$, $f:Y \to Z$, $p:X \to Y'$ and $q:Y' \to Z$ be morphisms in $\mathcal{C}$. If $Ug \circ U f = U q \circ U p$, then $q\circ p = g \circ f$.
			\item Let $f: X \to Y$ be a morphism in $\mathcal{C}$. Then, we can define an epi-mono factorization of $f$,
			\[
			\begin{tikzcd}
				X \arrow[rd, two heads] \arrow[rr, "f"] &                                   & Y \\
				& {\mathrm{im}\,f} \arrow[ru, hook] &  
			\end{tikzcd}
			\]
			by taking $\mathrm{im}\,f$ to be the final lift of the sink $UX \twoheadrightarrow \mathrm{im}\, U f$ in $\mathrm{Set}$.
		\end{enumerate}
	\end{lemma}

	\subsection{Functors induced between comma categories} \label{sec:app:general}
	
	One can derive a functor between these two comma categories, as explained in the following lemma.	
	
	\begin{lemma} \label{lem:oh9d}
		\cite[Prop 6]{Das2024slice} Consider the arrangement of categories $\calA, \calB, \calC, \calD, \calE$
		\begin{equation} \label{eqn:gpd30}
			\begin{tikzcd}
				\calA \arrow{d}{I} \arrow[Holud]{r}{F} & \calB \arrow{d}[swap]{J} & \calC \arrow[Holud]{l}[swap]{G} \arrow{d}{K} \\
				\calA' \arrow[Akashi]{r}[swap]{F'} & \calB' & \calC' \arrow[Akashi]{l}{G'}
			\end{tikzcd}
		\end{equation}
		Then there is an induced functor between comma categories
		\begin{equation} \label{eqn:odl9}
			\Psi : \Comma{F}{G} \to \Comma{F'}{G'}, 
		\end{equation}
		Moreover, the following commutation holds with the marginal functors :
		\begin{equation} \label{eqn:dgbw}
			\begin{tikzcd} [column sep = large]
				\calA \arrow[d, "I"'] & \Comma{F}{G} \arrow[l, "\Forget_1"'] \arrow[r, "\Forget_2"] \arrow[d, "\Psi"] & \calC \arrow[d, "K"] \\
				\calA' & \Comma{F'}{G'} \arrow[l, "\Forget_1"] \arrow[r, "\Forget_2"'] & \calC'  
			\end{tikzcd}
		\end{equation}
	\end{lemma}
	
	Note that the two horizontal rows of \eqref{eqn:gpd30} correspond to two comma categories. Lemma \ref{lem:oh9d} states that if comma categories are intertwined in the manner of \eqref{eqn:gpd30}, then there is a functor between these comma categories which contain the vertical functors as projections.
	
	\subsection{Categorical view of subshifts} \label{sec:app:subshift}
	
	\paragraph{Proof of Lemma \ref{lem_rep}} The result follows from the  diagram in $\Context$ shown below on the left.
	\[
	\begin{tikzcd} [scale cd = 0.8]
		S \arrow[rr, "f"] \arrow[d, two heads] & & S' \arrow[d, two heads] \\
		\sigma^i S \arrow[d, hook] & & \sigma^i S' \arrow[d, hook] \\
		S \arrow[rr, "f"] \arrow[d, two heads] & & S' \arrow[d, two heads] \\
		\word_n(S) \arrow[rr, "f_0"] & & \word_0(S') 
	\end{tikzcd} 
	\begin{tikzcd} {} \arrow[rr, mapsto, "U"] && {} \end{tikzcd}
	\begin{tikzcd} [scale cd = 0.8]
		U(S) \arrow[rr, "Uf"] \arrow[d, two heads] & & U(S') \arrow[d, two heads] \\
		\sigma^i U(S) \arrow[d, hook] & & \sigma^i U(S') \arrow[d, hook] \\
		U(S) \arrow[rr, "Uf"] \arrow[d, two heads] & & U(S') \arrow[d, two heads] \\
		\word_n(U(S)) \arrow[rr, "U(f_0)"] & & \word_0(U(S')) 
	\end{tikzcd}
	\]
	The diagram on the right is the image of the one on the left under the functor $U$. The two sides of the identity in Lemma \ref{lem_rep} correspond to the two paths from the top-left to the bottom-right corner of the diagram. This completes the proof of Lemma \ref{lem_rep}. \qed
	
	\paragraph{Proof of Lemma \ref{lem:cylinder}} The pullback square in the construction of cylinders is the categorical analog of inverses. One of the properties of pullback squares is that if any morphism in the original vee-shaped diagram is an injection, then so is its parallel morphism in the pullback square. The morphism opposite to the morphism from $\Cylinder(w) \to S$ originates from the terminal element, and is injective by default. This makes the morphism $\Cylinder(w) \to S$ injective too. This completes the proof of Lemma \ref{lem:cylinder}. \qed 
	
	\paragraph{Proof of Lemma \ref{lem_steps}} We set up the map $R_m :U\word_{m+n}(S) \to U\word_{m}(S') $ by \[R_m(\xi_0\xi_1\cdots\xi_{m+n}) := ((Uf_0)(\xi_0\xi_1\cdots \xi_n) (Uf_0)(\xi_1\xi_2\cdots \xi_{n+1})\cdots (Uf_0)(\xi_m\xi_{m+1}\cdots \xi_{m+n})).\] This is well-defined by Lemma \ref{lem_rep}. By applying Lemma \ref{lem:U_set:1}, we can verify the existence of the desired morphism $f^{(m)}$. This completes the proof of Lemma \ref{lem_steps}. \qed
	
	\paragraph{Proof of Lemma \ref{lem:fin_compose}} The proof follows directly from the diagram
	\[\begin{tikzcd} [column sep = large]
		S \arrow[r, "f"] \arrow[d, two heads] & S' \arrow[r, "f'"] \arrow[d, two heads] & S'' \arrow[d, two heads] \\
		\word_{m+n}(S) \arrow[r, "f^{(m)}"] & \word_{m}(S') \arrow[r, "f'_0"] & \word_0(S'')
	\end{tikzcd}\]
	This completes the proof of Lemma \ref{lem:fin_compose}. \qed
	
	
	\subsection{Subshifts from measurements} \label{sec:app:Seq_mes}
	
	\begin{lemma} \label{lem:lidf0l}
		Suppose Assumptions \ref{A:concrete} and \ref{A:obs} hold. Then for each morphism $(h,A): (\Omega, \Phi, \phi) \to (\Omega', \Phi', \phi')$ in $ \DSM$, we have the following commutative diagram:
		\[
		\begin{tikzcd}
			\Omega \arrow[d, "h"] \arrow[rr, "\xi_{(\Omega, \Phi, \phi)}"] & & (\iota Y)^{\num} \arrow[d, "A_*"] \\
			\Omega' \arrow[rr, "\xi_{(\Omega', \Phi', \phi')}"] & & (\iota Y')^{\num} 
		\end{tikzcd},
		\]
		where $A_*$ is defined by
		\[
		\begin{tikzcd}
			(\iota Y)^{\num} \arrow[d, "\pi_n"] \arrow[rr, "A_*", dotted] & & (\iota Y')^{\num} \arrow[d, "\pi_n'"] \\
			\iota Y \arrow[rr, "A"] & & \iota Y' 
		\end{tikzcd}
		\]
	\end{lemma}
	
	Lemma \ref{lem:lidf0l} formally proves that morphisms in $\DSM$ naturally lead to morphisms between the full shift spaces corresponding to the observation spaces. These morphisms commute with the projection maps at all times $n$. 
	
	\paragraph{Proof of Lemma \ref{lem:lidf0l}} The morphism $\xi_{(\Omega, \Phi, \phi)}$ exists uniquely by the universality of products. The claim now follows from the diagram:
	\[
	\begin{tikzcd}
		\Omega \arrow[rrr, "\xi_{(\Omega, \Phi, \phi)}"] \arrow[ddd, "h"] \arrow[rrd, "\phi\circ \Phi^n"] & & & (\iota Y)^{\num} \arrow[ddd, "A_*"] \arrow[ld, "\pi_n"] \\
		& & \iota Y \arrow[d, "A"] & \\
		& & \iota Y' & \\
		\Omega' \arrow[rrr, "\xi_{(\Omega', \Phi', \phi')}"'] \arrow[rru, "\phi'\circ \Phi'^n"] & & & (\iota Y')^{\num} \arrow[lu, "\pi_n'"] 
	\end{tikzcd},
	\]
	where commutativity of inner triangles and rectangles have already been established. The identity to be proved lies on the periphery of this diagram. This completes the proof of Lemma \ref{lem:lidf0l}. \qed
	
	\paragraph{Proof of Theorem \ref{thm:Obs_functor}} We only check for functoriality for \eqref{eqn:def:Obs:DSM} as the proof for \eqref{eqn:def:Obs:DSMO} will be analogous. Fix a $\DSM$ object $(\Omega, \Phi, \phi)$. Consider the epi-mono factorization of the morphism $\xi_{(\Omega, \Phi, \phi)}$ defined above :
	\[\begin{tikzcd}
		& \Obs(\Omega, \Phi, \phi) \arrow[dr, hook] \\
		\Omega \arrow[ur, two heads] \arrow[rr, dotted, "\xi_{(\Omega, \Phi, \phi)}"] && (\iota Y)^{\num}
	\end{tikzcd}\]
	Then the correspondence of $(\Omega, \Phi, \phi)$ with $\Obs(\Omega, \Phi, \phi)$ is functorial. Its action on any morphism $(h,A): (\Omega, \Phi, \phi) \to (\Omega', \Phi', \phi')$ is the unique morphism shown below as a dotted, green line that enables the commutation :
	\[\begin{tikzcd}
		\Omega \arrow[rr, "h"] \arrow[d, two heads] & & \Omega' \arrow[d, two heads] \\
		\Obs(\Omega, \Phi, \phi) \arrow[rr, "{\Obs(h,A)}", dotted, Shobuj] \arrow[d, hook] & & \Obs(\Omega', \Phi', \phi') \arrow[d, hook] \\
		(\iota Y)^{\num} \arrow[rr, "A_*"] & & (\iota Y')^{\num} 
	\end{tikzcd}\]
	The test for compositionality is routine and will be omitted. This completes the proof of Theorem \ref{thm:Obs_functor}. \qed
	
	\subsection{Time series data} \label{sec:app:TSD}
	
	\paragraph{Compositionality in $\TSD$} Suppose there is another morphism $(l,\psi) : Y \to Z$. Then $(l,\psi)$ composes with $(k,\phi)$ as shown below, yielding a morphism $(k+l, \psi \circ \phi): X \to Z$:
	\[\begin{tikzcd}
		\Shobuj{X_{n+1+k+l}} \arrow[bend left=30]{rrrr}{ \paran{\psi \circ \phi}_{n+1} } \arrow{rr}{ \phi_{n+1+l} } \arrow[Shobuj]{d}{ \finish_{n+k+1} }[swap]{ \start_{n+1+k} } && \akashi{ Y_{n+1+l} } \arrow{rr}{ \psi_{n+1} } \arrow[Akashi]{d}{ \finish_{k+1} }[swap]{ \start_{n+k+1} } && \Holud{ Z_{n+1} } \arrow[Holud]{d}{ \finish_{n+1} }[swap]{ \start_{n+1} } \\ 
		\Shobuj{X_{n+k+l}}  \arrow[bend right=30]{rrrr}[swap]{ \paran{\psi \circ \phi}_{n} } \arrow{rr}[swap]{ \phi_{n+l} } && \akashi{ Y_{n+l} } \arrow{rr}[swap]{ \psi_{n} } && \Holud{ Z_{n} }
	\end{tikzcd} , \quad
	\forall n\in\num .\]
	We can check that the composition is associative by using the property $(\psi\circ \phi)_n = \psi_n \circ \phi_{n+l}$. \qed
	
	\begin{lemma}\label{lem_cylinder}
		Let $\Context$ be as in Assumption \ref{A:concrete} and $\mathbb{A}$ be an object in $\Context$. For each $n$, we have
		\[ \finish_{n} \circ \pi_{\num\to n} = \pi_{\num \to n-1}, \quad \pi_{\num \to n-1} \circ \sigma = \start_n \circ \pi_{\num\to n}.\]
	\end{lemma}
	\begin{proof}
		The result follows from universality of the projection $\pi_{\num \to n-1}$ along with the diagrams which hold for every $1\leq i \leq n-1$ : 
		\[\begin{tikzcd} [column sep = large]
			\mathbb{A}^{\num} \arrow[r, "\pi_{\num\to n}"] \arrow[rrd, "\pi^{(\num)}_i"'] & \mathbb{A}^{n} \arrow[r, "\finish_{n}"] \arrow[rd, "\pi^{(\num)}_{i}"] & \mathbb{A}^{n-1} \arrow[d, "\pi^{(n-1)}_i"] \\
			& & \mathbb{A}
		\end{tikzcd}  ;\; 
		\begin{tikzcd} [column sep = large]
			\mathbb{A}^{\num} \arrow[r, "\pi_{\num\to n}"] \arrow[d, "\sigma"] \arrow[rrd, "\pi^{(\num)}_{i+1}"'] & \mathbb{A}^n \arrow[r, "\start_n"] \arrow[rd, "\pi^{(n)}_{i+1}"] & \mathbb{A}^{n-1} \arrow[d, "\pi^{(n-1)}_i"] \\
			\mathbb{A}^{\num} \arrow[rr, "\pi^{(\num)}_i"'] & & \mathbb{A}
		\end{tikzcd} . \]
	\end{proof}
	
	The following technical lemma is needed for the proof of Lemma \ref{lem:factor}.
	
	\begin{lemma} \label{lem:nf0ep}
		Suppose Assumption \ref{A:concrete} holds, $I$ is an indexing set, $c$ is an object, $\SetDef{a_i}{i\in I}$ is a family of objects in $\Context$ and $f_i : c \to a_i$ is a family of morphisms. Next suppose that $a$ is a subobject of the product $ \prod_{i\in I} a_i$. Then if there is a morphism $\phi$ which creates the commutation shown below
		\[\begin{tikzcd}
			c \arrow[d, "f_i"] \arrow[dotted, "\phi", r] & a \arrow[d, hook] \\
			a_i & \prod_{i\in I} a_i \arrow[l, "\pi_i"']
		\end{tikzcd} ,\]
		then such a $\phi$ is unique.
	\end{lemma}
	
	Lemma \ref{lem:nf0ep} directly follows from the universality of the product.
	
	\paragraph{Proof of Lemma \ref{lem:factor}} Let $\mathbb{A}$ and $\mathbb{B}$ respectively denote the alphabet of the $\TSD$ objects $X$ and $Y$. The following diagram places \eqref{eqn:TSD_commut} in a larger context
	\[\begin{tikzcd}
		\mathbb{A}^{n+k+1} \arrow[d, "\start_{n+k+1}", "\finish_{n+k+1}"'] && X_{n+k+1} \arrow[ll, hook] \arrow[d, "\start_{n+k+1}", "\finish_{n+k+1}"'] \arrow[rr, "\psi_{n+1}"] && Y_{n+1} \arrow[rr, hook] \arrow[d, "\start_{n+1}", "\finish_{n+1}"'] && \mathbb{B}^{n+1} \arrow[d, "\start_{n+1}", "\finish_{n+1}"'] \\
		\mathbb{A}^{n+k} && X_{n+k} \arrow[ll, hook] \arrow[ll, hook] \arrow[rr, "\psi_n"'] && Y_{n} \arrow[rr, hook] && \mathbb{B}^{n}
	\end{tikzcd}\]
	The proof follows directly from Lemma \ref{lem:nf0ep} by taking $I = \braces{1,\ldots,n+1}$, $a_i = \mathbb{B}$ and $a = Y_{n+1}$, $c=X_{n+k+1}$ and $\phi = \psi_{n+1}$. \qed
	
	\paragraph{Proof of Theorem \ref{thm:TSD_from_Seq}} We check the functoriality of the construction in \eqref{eqn:did93k}. Let $\phi:S \to S'$ be an $n$-morphism. We set up $\phi_m: \word_{m+n}(S) \to \word_{m}(S) $ using Lemma \ref{lem_steps}. By a direct calculation in $\mathrm{Set}$, we can check that the following diagrams commute:
	\[
	\begin{tikzcd}
		S \arrow[d, two heads] \arrow[r, "\phi"] & S' \arrow[d, two heads] \\
		\word_{m+n}(S) \arrow[r, "\phi_m"] \arrow[d, "{f_{m+n},\start_{m+n}}"] & \word_{m}(S) \arrow[d, "{f_m,\start_m}"] \\
		\word_{m+n-1}(S) \arrow[r, "\phi_{m-1}"] & \word_{m-1}(S) 
	\end{tikzcd}
	\]
	Thus $\{\phi\}_m$ indeed defines a morphism in TSD category. Compositionality follows from the uniqueness of the generator. This completes the proof of Theorem \ref{thm:TSD_from_Seq}. \qed
	
	\subsection{Delay coordinates} \label{sec:app:delay}
	
	\begin{definition} [Delay coordinates]
		Suppose Assumptions \ref{A:concrete} and \ref{A:obs} hold, $(\Omega, \Phi, \phi)$ is a $\DSM$ object, and $k\in\num$ is any integer. Let $\iota Y$ be the codomain of $\phi$. Then there is a unique morphism $\phi^{(k)}: \Omega \to \iota Y^{k+1}$ which is defined uniquely by the following commutations :
		\[
		\begin{tikzcd}
			\Omega \arrow[rr, "\phi^{(k)}", dashed] \arrow[rrd, "\phi \circ \Phi^i"'] & & \iota (Y^k) = (\iota Y)^k \arrow[d, "\pi^{(k)}_i"] \\
			& & \iota Y 
		\end{tikzcd}
		, \quad \forall 1\leq i \leq k.
		\]
	\end{definition}
	
	\begin{lemma} [Delay coordinate functor ] \label{lem:delay_coord:fnctr}
		Suppose Assumptions \ref{A:concrete} and \ref{A:obs} hold, and $k\in\num$ is any integer. Then there is a functor
		\begin{equation}
			\text{Delay}^{(k)} : \DSM \to \DSM, \quad (\Omega, \Phi, \phi) \mapsto (\Omega, \Phi, \phi^{(k)}).
		\end{equation}
		whose action on a morphism $(h,A) : (\Omega, \Phi, \phi) \to (\Omega', \Phi', \phi')$ is $(h, A^k)$. Moreover, there is a natural transformation $\alpha^{(k)} : \text{Delay}^{(k)} \to \Id$.
	\end{lemma}
	
	The proof of Lemma \ref{lem:delay_coord:fnctr} requires the following elementary observation :
	
	\begin{lemma}
		Let $\Context$ be as in Assumption \ref{A:concrete} and $A$ be an object. For each $n$ and $k$, there is a morphism $r_{A,n,k}: (A^{k+1})^{n+1} \to A^{n+k+1}$ such that
		\[
		\begin{tikzcd}
			(A^{k+1})^{n+1} \arrow[rr, "r_{A,n,k}", dashed] \arrow[d, "\pi_0"] & & A^{n+k+1} \arrow[d, "\pi_i"] \\
			A^{k+1} \arrow[rr, "\pi_i"] & & A 
		\end{tikzcd}
		\]
		for $i=0,1,\cdots, k$ and
		\[
		\begin{tikzcd}
			(A^{k+1})^{n+1} \arrow[rr, "r_{A,n,k}", dashed] \arrow[d, "\pi_{i-k}"] & & A^{n+k+1} \arrow[d, "\pi_i"] \\
			A^{k+1} \arrow[rr, "\pi_k"] & & A 
		\end{tikzcd}
		\]
		for $i=k+1,\cdots, k+n$. 
	\end{lemma}
	
	\paragraph{Proof of Lemma \ref{lem:delay_coord:fnctr}} The action on morphisms can be verified from this diagram
	\[\begin{tikzcd}
		& & & \iota Y \arrow[ddd, "A"] \\
		\Omega \arrow[r, "\Phi^i"] \arrow[d, "h"'] \arrow[rr, "\phi^{(k)}", dotted, bend left=20, Shobuj] & \Omega \arrow[d, "h"'] \arrow[rru, "\phi", bend left] & (\iota Y)^k \arrow[ru, "\pi^A_i"] \arrow[d, "A^k"] & \\
		\Omega' \arrow[r, "\Phi'^i"'] \arrow[rr, "\phi'^{(k)}"', dotted, bend right=20, Shobuj] & \Omega' \arrow[rrd, "\phi'"', bend right] & (\iota Y')^k \arrow[rd, "\pi^{A'}_i"'] & \\
		& & & \iota Y' 
	\end{tikzcd}\]
	The compositionality of this action is trivial. We construct a family of morphisms
	\[
	\alpha^{(k)}_n: \Data(\Omega, \Phi, \phi^{(k)})_n \to \Data(\Omega, \Phi, \phi)_{n+k}
	\]
	that commutes with $\finish$ and $\start$. Observe that 
	\[
	r_{A,n,k} \circ \pi_{\num\to n+1} \circ \xi_{(\Omega, \Phi, \phi^{(k)})} = \pi_{\num\to n+k+1} \circ \xi_{(\Omega, \Phi, \phi)}
	\]
	by a direct calculation. This implies that we have the following diagram:
	\[
	\begin{tikzcd}
		\Omega \arrow[dd, "\pi_{\num\to n+1}\circ\xi_{(\Omega, \Phi, \phi^{(k)})}"'] \arrow[rr, "1"] & & \Omega \arrow[dd, "\pi_{\num\to n+k+1}\circ\xi_{(\Omega, \Phi, \phi)}"] \\
		& & \\
		(\iota Y^{k+1})^{n+1} \arrow[rr, "r_n"] & & \iota Y^{n+k+1} 
	\end{tikzcd}
	\]
	From Lemma \ref{lem:U_set:2} and the definition of $\Data$, we obtain a morphism $\alpha^{(k)}_n :\Data(\Omega, \Phi, \phi^{(k)})_n \to \Data(\Omega, \Phi, \phi)_{n+k}$. By considering the underlying set structure, we can check that it commutes with $f$ and $\start$. This completes the proof of Lemma \ref{lem:delay_coord:fnctr}. \qed
	
	The main utility of Lemma \ref{lem:delay_coord:fnctr} is in showing that delay-coordinates may be used to convert jump-$k$ morphisms into jump-$0$ morphisms.
	
	\begin{lemma} [Delay coordinated data] \label{lem:delay_coord:data}
		Suppose Assumptions \ref{A:concrete} and \ref{A:obs} hold, and $\psi:\Data(\Omega, \Phi, \phi) \to X$ be a TSD morphism of jump $k$. Then, there exists a TSD morphism $\psi^{(k)}: \Data(\Omega, \Phi, \phi^{(k)}) \to X$ of jump $0$.
	\end{lemma}
	
	\begin{proof} The composition $\psi_{n+k} \circ \alpha^{(k)}_n: \Data(\Omega, \Phi, \phi^{(k)})_n \to X_n$  gives the desired TSD morphism of jump 0. \end{proof}
	
	\subsection{Reconstruction} \label{sec:app:recon}
	
	\begin{lemma}\label{lem_semiconj}
		Let $\psi: \Data(\Omega',\Phi', \phi')\to \Data\circ \iota'''_{obs} (\Omega, \Phi) $ be a TSD morphism of jump $0$. Then, there is a morphism $h: (\Omega', \Phi') \to (\Omega, \Phi)$ in $\mathrm{DS}$.
	\end{lemma}
	\begin{proof}
		From the definition of $\Data$, $\psi$ is determined uniquely by the following morphism $\psi_0$ :
		\[
		\begin{tikzcd}
			\Omega' \arrow[d, two heads] & & \iota \Omega \arrow[d, two heads] \\
			{\Data(\Omega',\Phi', \phi')_0} \arrow[d, hook] \arrow[rr, "\psi_0"] & & {\Data(\Omega,\Phi, 1_\Omega)_0} \arrow[d, hook] \\
			\iota Y' & & \iota \Omega 
		\end{tikzcd}
		\]
		Note that a morphism $h:\Omega'\to \iota \Omega$ in $\Context$ emerges from this diagram. Now we check that $h$ is a semiconjugacy. By Property \ref{Py:U_faithful} it is sufficient to show that $U(\Phi\circ h) = U(h\circ \Phi')$. This allows the analysis to be conducted in $\SetCat$. We begin by noting that :
		\[
		Uh(\omega') = U(\psi_0\circ \phi')(\omega'), \quad \forall \omega' \in U (\Omega') ,
		\]
		by construction. Now consider the set of length-$2$ sequences $U \paran{ \Data(\Omega',\Phi', \phi')_1 }$ Pick a typical element $\paran{ U \phi'(\omega'), U (\phi'\circ \Phi')(\omega') }$ from here. Then there must be an $\omega\in \Omega$ such that
		\[
		\paran{\omega, U \Phi(\omega)} = \paran{U\psi_1} \paran{ U \phi'(\omega'), U (\phi'\circ \Phi')(\omega') } \quad  \in U\circ\Data\circ \iota'''_{obs} (\Omega, \Phi).
		\]
		By applying $\start_1$ to both sides of the equality, we obtain 
		\[
		\omega = U\psi_0 \circ U \phi'(\omega').
		\]
		and if we apply $f_1$, we obtain
		\[
		U \Phi(\omega) = U\psi_0 \circ U (\phi'\circ \Phi')(\omega').
		\]
		Together they imply
		\[
		U(\Phi\circ h)(\omega') = U(h\circ \Phi')(\omega'). 
		\]
		As $\omega'$ is arbitrary, this implies $U(\Phi\circ h) = U(h\circ \Phi')$ as claimed.
	\end{proof}
	
	\paragraph{Proof of Theorem \ref{thm_consistent:3}} We begin by drawing a large commutation diagram :
	\[\begin{tikzcd}
		\DSM \arrow[rd, "\kappa"'] \arrow[dd, "Data'"', dashed, Shobuj] \arrow[rrrd, "\Forget'", Shobuj] & & & & & \DS_{obs} \arrow[lllll, "\iota'''_{obs}"'] \arrow[dd, "\subset"', Itranga, "\iota''_{obs}"] \arrow[lllld, "\iota_{obs}"'] \\
		& \DSMO \arrow[ld, "\Data"', Akashi] \arrow[rr, pos=0.7, "\Forget"', Akashi] & & \DS & & \\
		\TSD \arrow[rrr, "\mathrm{Lan}_{\Data} \Forget", Akashi] & & & \DS & & \itranga{\DS} \arrow[ll, "\cong", Itranga] \arrow[llu, "\cong"', Itranga] 
	\end{tikzcd}\]
	The commutations upheld by the red arrows are to be verified. The $\Data$ and $\Forget$ functors from $\DSMO$ have been pulled back along $\kappa$ to create analogous functors $\Data'$ and $\Forget'$ functors respectively from $\DSM$. Note that the functor $\kappa$ has a right inverse. Then by Lemma \ref{lem:Kan_id:2}~(iii) we have
	\[ \REnv{\Data}{\Forget} = \REnv{\Data \kappa}{\Forget \kappa} = \REnv{\Data'}{\Forget'} . \]
	This indicates that in analyzing the Kan construction, we are free to base the construction around $\DSMO$ or $\DSM$. We choose the latter for convenience. Thus the goal now is to prove
	\[ \paran{ \REnv{\Data'}{\Forget'} } \circ \Data' \circ \iota'''_{obs} = \iota''_{obs} . \]
	Fix an $(\Omega, \Phi)$ in $\DS_{obs}$. By the colimit formula of left Kan extension, we have
	\begin{equation}\label{def_lan}
		\REnv{\Data'}{\Forget'} \paran{\Data' \circ \iota'''_{obs} (\Omega, \Phi)} = \mathrm{colim}\,\{(\Omega', \Phi') : \Data(\Omega',\Phi', \phi')\to \Data\circ \iota'''_{obs} (\Omega, \Phi)\}.
	\end{equation}
	Thus the colimit is over a certain collection $(\Omega', \Phi')$ of dynamical systems which are a part of some $\DSM$ object which yield a data object $X'$ on the left slice of $X$. Our goal is to show that the colimit is the original dynamical system $\iota'_{obs} \paran{\Omega, \Phi}$. For simplicity of notation we denote this object as $\paran{\Omega, \Phi}$.
	
	First we check that there is a cocone with vertex $(\Omega, \Phi)$. We recall that the colimit in equation (\ref{def_lan}) is taken over the slice category $(\Data \downarrow \Data\circ \iota'''_{obs} (\Omega, \Phi))$. An object in this category is a pair of a $ \DSM$ object $(\Omega',\Phi', \phi')$ and a $\TSD$ morphism $\psi: \Data(\Omega',\Phi', \phi')\to \Data\circ \iota'''_{obs} (\Omega, \Phi)$. A morphism $((\Omega',\Phi', \phi'), \psi) \to ((\Omega'',\Phi'', \phi''), \psi')$ is a $ \DSM$ morphism $(h, A): (\Omega',\Phi', \phi') \to (\Omega'',\Phi'', \phi'')$ such that
	\[
	\begin{tikzcd}
		{\Data(\Omega',\Phi', \phi')} \arrow[rr, "{\Data(h,A)}"] \arrow[rd, "\psi"] & & {\Data(\Omega'',\Phi'', \phi'')} \arrow[ld, "\psi'"] \\
		& {\Data\circ \iota'''_{obs} (\Omega, \Phi)} & 
	\end{tikzcd}.
	\]
	Lemmas \ref{lem:delay_coord:data} and \ref{lem_semiconj} guarantee that we have the following semi-conjugacies of dynamics
	\[ \gamma := \gamma_{(\Omega',\Phi', \phi'),\psi}: (\Omega',\Phi') \to \iota'''_{obs} (\Omega, \Phi), \quad \gamma' := \gamma_{(\Omega'',\Phi'', \phi''),\psi'}: (\Omega'',\Phi'') \to \iota'''_{obs} (\Omega, \Phi) \]
	Now suppose that $\psi$ has jump $k$. The semi-conjugacies above are transported by $U$ in semi-conjugacies in $\SetCat$-valued dynamics. Combining with the $\DSM$-morphism $(h, A)$ we get :
	\begin{equation*}
		\begin{split}
			U \gamma_{(\Omega',\Phi', \phi'),\psi} (\omega') &= U(\psi_0^{(k)} \circ \phi'^{(k)}) (\omega') \\
			&= \paran{ U{\psi'}_0^{(k)} } \circ \paran{ U A_* } \circ \paran{ U\phi'^{(k)} } (\omega')\\
			&= \paran{ U{\psi'}_0^{(k)} } \circ \paran{ U {\phi''}^{(k)} } \circ \paran{ Uh } (\omega') \\
			&= \paran{ U \gamma_{(\Omega'',\Phi'', \phi''),\psi} } \circ \paran{ Uh } (\omega'),
		\end{split}
	\end{equation*}
	for every $\omega' \in \Omega'$. Therefore we have $\gamma_{(\Omega',\Phi', \phi'),\psi} = \gamma_{(\Omega'',\Phi'', \phi''),\psi} \circ h$ in $\mathrm{DS}$. Thus, $\gamma_{(\Omega',\Phi', \phi'),\psi} $ is a cocone with vertex $(\Omega, \Phi)$.
	
	Next we check that any cocone $\alpha$ with vertex $W$ is factorized by $\gamma$ when the object $((\Omega',\Phi', \phi'),\psi)$ has jump 0:
	\[
	\alpha_{(\Omega',\Phi', \phi'),\psi} = v \circ \gamma_{(\Omega',\Phi', \phi'),\psi},
	\]
	where the morphism $v: W\to (\Omega, \Phi)$ does not depend on the choice of $((\Omega',\Phi', \phi'),\psi)$ and determined uniquely.
	Since we have $\TSD$ morphism $1:\Data\circ \iota'''_{obs} (\Omega, \Phi)\to \Data\circ \iota'''_{obs} (\Omega, \Phi)$, there is $\alpha_{(\Omega,\Phi, 1_\Omega),1}: (\Omega,\Phi) \to W$. We show that \[\alpha_{(\Omega',\Phi', \phi'),\psi}= \alpha_{(\Omega,\Phi, 1_\Omega),1}\circ \gamma_{(\Omega',\Phi', \phi'),\psi}\]
	for all $((\Omega',\Phi', \phi'),\psi)$ with jump 0.
	Here we note that, if $\psi: \Data(\Omega',\Phi', \phi')\to \Data\circ \iota'''_{obs} (\Omega, \Phi)$ is a morphism of jump $0$, the morphism \[ \tilde\psi_0: \Data(\Omega',\Phi', {\phi'} )_0 = \mathrm{im}\, {\phi'} \xrightarrow{\psi_0} \Data\circ \iota'''_{obs} (\Omega, \Phi)_0\hookrightarrow\Omega\] satisfies
	\begin{equation}\label{eqn_rep}
		\begin{tikzcd}
			{\Data(\Omega',\Phi', {\phi'})} \arrow[rr, "{\Data(\gamma_{(\Omega',\Phi', \phi'),\psi},\tilde\psi_0)}"] \arrow[rd, "\psi"] & & {\Data\circ \iota'''_{obs} (\Omega, \Phi)} \arrow[ld, "1"] \\
			& {\Data\circ \iota'''_{obs} (\Omega, \Phi)} & 
		\end{tikzcd}
	\end{equation}
	This is shown as follows. First we check that $(\gamma_{(\Omega',\Phi', \phi'),\psi},\tilde\psi_0)$ is a $\DSM$ morphism from $(\Omega',\Phi', {\phi'})$ to $(\Omega,\Phi, 1_\Omega)$. This amounts to verifying that $\gamma_{(\Omega',\Phi', \phi')} = \tilde\psi_0 \circ {\phi'},$ but it follows readily from the definition of $\gamma_{(\Omega',\Phi', \phi')}$.
	Next we show that $U\psi_i = U\Data(\gamma_{(\Omega',\Phi', \phi'),\psi},\tilde\psi_0)_i$ for all $i$ by induction. The case $i=0$ follows from the definition. Now we assume the claim holds for $i-1$. Let $(\xi_0, \xi_1, \cdots, \xi_i)$ be an element in $ U \Data(\Omega',\Phi',{\phi'})_i$ and set
	\[
	U \psi_i (\xi_0, \xi_1, \cdots, \xi_i) = (\xi'_0, \xi'_1, \cdots, \xi'_i).
	\]
	Application of $U f_i$ shows that 
	\[
	U f_i \circ U \psi_i (\xi_0, \xi_1, \cdots, \xi_i) = U \psi_{i-1} ( \xi_1, \cdots, \xi_i) = ( U \tilde\psi_0(\xi_1), \cdots, U\tilde\psi_0(\xi_i)) 
	\]
	Similarly, application of $U \start_i$ shows that 
	\[
	U \start_i \circ U \psi_i (\xi_0, \xi_1, \cdots, \xi_i) = U \psi_{i-1} ( \xi_0, \cdots, \xi_{i-1}) = ( U \tilde\psi_0(\xi_0), \cdots, U\tilde\psi_0(\xi_{i-1})) 
	\]
	Combining these result, we conclude
	\[
	(\xi'_0, \xi'_1, \cdots, \xi'_i) = U\Data(\gamma_{(\Omega',\Phi', \phi'),\psi},\tilde\psi_0)_i (\xi_0, \xi_1, \cdots, \xi_i).
	\]
	Thus, we have $\psi_i = \Data(\gamma_{(\Omega',\Phi', \phi'),\psi},\tilde\psi_0)_i$ for all $i$.
	From the diagram (\ref{eqn_rep}), we have
	\[
	\alpha_{(\Omega',\Phi', {\phi'}),\psi} = \alpha_{(\Omega,\Phi, 1_\Omega),1}\circ \gamma_{(\Omega',\Phi', \phi'),\psi},
	\]
	which is the desired factorization.
	
	Now, consider the case where $\alpha$ is the universal cocone defining the colimit (\ref{def_lan}). In what follows, we will write $L := \mathrm{Lan}_{\Data} \mathrm{Frgt}$.
	
	Let the vertex of $\alpha$ be $L\circ \iota_{obs}(\Omega, \Phi)$.The cocone $\gamma$ can be factorized as
	\[
	\gamma_{(\Omega',\Phi', \phi'),\psi} = u \circ \alpha_{(\Omega',\Phi', \phi'),\psi},
	\]
	where $u: L\circ \iota_{obs}(\Omega, \Phi) \to (\Omega, \Phi)$. As $((\Omega,\Phi, 1_\Omega),1)$ has jump 0, we obtain
	\[
	1_\Omega = \gamma_{(\Omega,\Phi, 1_\Omega),1} = u \circ v \circ\gamma_{(\Omega,\Phi, 1_\Omega),1} = u \circ v.
	\]
	Conversely, we obtain
	\[
	\alpha_{(\Omega,\Phi, 1_\Omega),1} = v\circ u \circ \alpha_{(\Omega,\Phi, 1_\Omega),1}.
	\]
	From the universal property, we have $v\circ u = 1_{L\circ \iota_{obs}(\Omega, \Phi)}$. Thus, $u:L\circ \iota_{obs}(\Omega, \Phi) \to (\Omega, \Phi)$ is an isomorphism. As $u$ depends only on $(\Omega, \Phi)$, we will write $u_{(\Omega, \Phi)}$.
	
	Finally, we verify $u_{(\Omega, \Phi)}$ gives a natural transformation $L\circ I \Rightarrow \mathrm{Frgt} \circ I$. Let $h: (\Omega, \Phi) \to (\tilde\Omega, \tilde\Phi)$ be a morphism in $\mathrm{DS}_{obs}$. By definition of $L$, we have
	\[
	\tilde\alpha_{(\Omega',\Phi', \phi'),(\Data\circ \iota'''_{obs} h)\psi} = ((L\circ I) h) \circ \alpha_{(\Omega',\Phi', \phi'),\psi},
	\]
	where $\alpha$ and $\tilde \alpha$ are cocones that define colimits $(L\circ I)(\Omega, \Phi)$ and $(L\circ I)(\tilde\Omega,\tilde\Phi)$, respectively.
	
	If we construct a cocone $\tilde{\gamma}$ analogously to $\gamma$, we have
	\[
	\tilde{\gamma}_{(\Omega',\Phi', \phi'), (\Data\circ \iota'''_{obs} h)\psi} = h \circ \gamma_{(\Omega',\Phi', \phi'), \psi}
	\]
	by definition. In particular, we have
	\[
	u_{(\tilde\Omega, \tilde\Phi)} \circ \tilde\alpha_{(\Omega',\Phi', \phi'), (\Data\circ \iota'''_{obs} h)\psi} = h \circ u_{(\Omega, \Phi)} \circ \alpha_{(\Omega',\Phi', \phi'),\psi}.
	\]
	By the universality of $\alpha$, we have $u_{(\tilde\Omega, \tilde\Phi)} \circ ((L\circ I) h) = h \circ u_{(\Omega, \Phi)}$, completing the proof of Theorem \ref{thm_consistent:1}. \qed
	
	\section*{Acknowledgments}
	The authors are grateful to the referees, whose inputs helped improve the presentation greatly. 
	
	\bibliographystyle{\Path unsrt_inline_url}	

\end{document}